\documentclass[letterpaper]{article}%
\usepackage{amssymb}
\usepackage{amsfonts}
\usepackage{amsmath}
\usepackage{geometry}
\usepackage{multicol}%
\usepackage{graphicx}
\newtheorem{theorem}{Theorem}

\newtheorem{example}{Example}

\newenvironment{proof}[1][Proof]{\noindent\textbf{#1.} }{}
\begin{document}

\title{Universal Hyperbolic Geometry I:\\Trigonometry}
\author{N J Wildberger\\School of Mathematics and Statistics\\UNSW Sydney 2052 Australia}
\date{}
\maketitle

\begin{abstract}
Hyperbolic geometry is developed in a purely algebraic fashion from first
principles, without a prior development of differential geometry. The natural
connection with the geometry of Lorentz, Einstein and Minkowski comes from a
projective point of view, with trigonometric laws that extend to `points at
infinity', here called `null points', and beyond to `ideal points' associated
to a hyperboloid of one sheet. The theory works over a general field not of
characteristic two, and the main laws can be viewed as deformations of those
from planar rational trigonometry. There are many new features.

\end{abstract}

\section{Introduction}

Hyperbolic geometry is set out here in a new and completely algebraic way.
This view of the subject, called \textit{universal hyperbolic geometry,} is a
special case of the more general geometry described in \cite{Wild4}, and has
the following characteristics that generally distinguish it from the classical
hyperbolic geometry found in for example \cite{Artzy}, \cite{Beardon},
\cite{Coxeter}, \cite{Greenberg}, \cite{Katok}, \cite{Ramsey Richtmayer}, or
from other approaches to the subject, such as \cite{Klingenberg},
\cite{Szmielew} or \cite{Ungar}.

\begin{itemize}
\item a more direct and intimate connection with the geometry of Einstein's
Special Theory of Relativity in the framework of Lorentz and Minkowski. In
fact \textit{hyperbolic geometry is precisely projective relativistic
geometry. }This is a fundamental understanding. The connection with
relativistic geometry is also a key feature of \cite{Ungar}.

\item the basic set-up allows a consistent development of hyperbolic geometry
\textit{over the} \textit{rational numbers}. This is the simplest and purest
form of the subject.

\item there is a natural development of the subject \textit{over a finite
field}. This ties in with work of \cite{Angel}, \cite{Soto-Andrade},
\cite{Terras} and others.

\item a crucial \textit{duality }between points and lines that connects with,
and clarifies, the \textit{pole-polar} duality with respect to the unit circle
of projective geometry.

\item an unambiguous and concrete treatment of what are traditionally called
`points at infinity', here called \textit{null points,} together with what
former generations of projective geometers called `ideal points' (see for
example \cite{Sommerville}) which lie on \textit{null lines}. In terms of
relativistic geometry, we study the hyperboloid of two sheets with equation
$x^{2}+y^{2}-z^{2}=-1$, the null cone with equation $x^{2}+y^{2}-z^{2}=0$, and
the hyperboloid of one sheet with equation $x^{2}+y^{2}-z^{2}=1$
\textit{together}. The relevance of the latter is also discussed in
\cite{Thurston}.

\item the fundamental metrical measurements of \textit{quadrance} between
points and \textit{spread }between lines are similar to the corresponding
notions in planar rational trigonometry, and the basic laws of hyperbolic
trigonometry may be seen as deformations of those of planar rational
trigonometry (see \cite{Wild1}).

\item transcendental functions, such as $\log x,$ $\sinh x$ or $\cos x$ are
not needed. A prior development of the real number system is not needed.

\item the existence of a rich \textit{null trigonometry}:\ trigonometric
relations that involve null points and null lines.

\item \textit{parallels} play a more specialized role. Somewhat ironically,
\textit{Euclid's parallel postulate holds}.

\item the isometry group of the geometry \textit{does\ not act transitively}
on the space. The universal hyperbolic plane has aspects which appear in
negatively curved Riemannian geometry, other aspects which are Lorentzian, and
other aspects which are Euclidean.

\item the framework is \textit{algebraic geometry}, rather than
\textit{differential geometry}. But it is a form of algebraic geometry that
relates more to the historical approach prior to the twentieth century
direction. The focus is on\textit{\ metrical relations }and \textit{concrete
polynomial identities} which \textit{encode geometric realities}.
\end{itemize}

\subsection{Advantages of the new approach}

The advantages of universal hyperbolic geometry over the classical approach include:

\begin{itemize}
\item \textit{simplicity and elegance:} the subject is simple enough to be
accessible to beginning undergraduates and motivated high school students
without a prior understanding of calculus or real numbers. The elementary
aspects fit together pleasantly.

\item \textit{logical clarity:} traditional treatments of hyperbolic geometry
often have obscure foundations, or are rife with arguments that rely on
pictorial understanding. The purely algebraic framework frees us from logical
difficulties, and allows us to aspire to a complete and unambiguous treatment
of the subject from first principles.

\item \textit{accuracy:} the new theory achieves much greater accuracy in
concrete computations. Many problems can now be solved completely correctly,
whereas the classical theory provides only approximate solutions.

\item \textit{connections with number theory:} the links between geometrical
problems and number--theoretical questions become much more explicit.

\item \textit{new directions for special functions:} The remarkable
\textit{spread polynomials }of planar rational trigonometry also play a key
role in hyperbolic geometry. This family of (almost) orthogonal polynomials
replaces the \textit{Chebyshev polynomials of the first kind}.

\item \textit{extension of classical geometry:} many more traditional results
of Euclidean geometry can now be given their appropriate hyperbolic analogs.
Of the thousands of known results in Euclidean geometry, only a fraction
currently have analogs in the hyperbolic setting. This turns out to be a
consequence of the way we have, up to now, viewed the subject.
\textit{Hyperbolic geometry is a richer and ultimately more important theory
than Euclidean geometry}.

\item \textit{easier constructions:} these are often more direct and elegant
than in classical hyperbolic geometry. Over the rational numbers many are even
simpler than in Euclidean geometry, in that they require only a base circle
and a straightedge.

\item \textit{application to inversive geometry:} universal hyperbolic
geometry is the natural framework for a comprehensive and general approach to
\textit{inversive geometry}.

\item \textit{connections with chromogeometry:} universal hyperbolic geometry
relates naturally to a new three-fold symmetry in planar geometry that
connects Euclidean and relativistic geometries (see \cite{Wild3}).

\item \textit{new theorems:} the language and concepts of universal hyperbolic
geometry allow us to discover, formulate and prove many new and interesting results.

\item \textit{simpler proofs:} the purely algebraic framework allows many
proofs to be reduced to algebraic identities which can be easily verified by
computer. These identities are often quite interesting in their own right, and
warrant further study.

\item universal hyperbolic geometry has the same relation to relativistic
geometry as spherical or elliptic geometry has to solid Euclidean geometry.
Thus this new form of hyperbolic geometry contributes towards a \textit{new
geometrical language to study relativistic geometry}.
\end{itemize}

\subsection{Note to the reader}

This paper represents a major rethinking of this subject, and so careful
attention must be given to the basic definitions, some of which are novel, and
others which are variants of familiar ones. Some familiarity with rational
trigonometry in the Euclidean case is a helpful preliminary. The main
reference is \cite{Wild1}, see also \cite{Wild5}, and the series of YouTube
videos called `WildTrig'.

When we develop geometry seriously, physical intuition and pictorial arguments
ought to be separate from the logical structure. To stress the importance of
such a logically tight approach, we build up the theory \textit{over a general
field}, so there are no pictures. In particular we do not rely on a prior
understanding of the continuum, or equivalently the real number system, thus
avoiding a major difficulty finessed in most traditional treatments.

Over the rational numbers, the subject nevertheless has a highly visual
nature. Section 2.3 describes an extension of the Beltrami Klein model to
visualize hyperbolic points and lines over the rational numbers. The reader is
also invited to investigate \textit{The Geometer's Sketchpad} worksheets on
hyperbolic geometry to be posted at the author's UNSW website:

\begin{quote}
http://web.maths.unsw.edu.au/\symbol{126}norman/index.html
\end{quote}

Finite prime fields are also highly recommended, as many calculations become
simpler, and because this motivates us to think beyond the usual pictures, and
connects with interesting combinatorics and graph theory.

The reader will observe that I avoid the use of `infinite sets'. The reason is
simple---I no longer believe they exist. So in the few places where we refer
to a field $\mathbb{F},$ we have in mind the specification of a particular
\textit{type }of mathematical object, rather than the collection of `all those
of a certain type' into a completed whole. Please be reassured---this has
surprisingly little effect on the actual content of the theorems.

Some notational conventions: equality is represented by the symbol $=$ as
usual, while the symbol $\equiv$ denotes \textit{specification}, so that for
example $A\equiv\left[  3,4\right]  $ is the statement specifying---or
\textit{defining}---$A$ to be the ordered pair $\left[  3,4\right]  $. Proofs
end in square black boxes, examples end in diamonds. We use the phrase
`precisely when' instead of `if and only if'. Definitions are in
\textbf{bold}, while \textit{italics }are reserved for emphasis. Theorems have
names, always beginning with a capital letter.

\subsection{List of theorems}%

%TCIMACRO{\TeXButton{TwoColumns}{\begin{multicols}{2}}}%
%BeginExpansion
\begin{multicols}{2}%
%EndExpansion

1. Join of points

2. Meet of lines

3. Collinear points

4. Concurrent lines

5. Line through null points

6. Point on null lines

7. Perpendicular point

8. Perpendicular line

9. Opposite points

10. Opposite lines

11. Altitude line

12. Altitude point

13. Parallel line

14. Parallel point

15. Base point

16. Base line

17. Parametrizing a line

18. Parametrizing a point

19. Parametrizing a join

20. Parametrizing a meet

21. Parametrization of null points

22. Parametrization of null lines

23. Join of null points

24. Meet of null lines

25. Null diagonal point

26. Null diagonal line

27. Perpendicular null line

28. Perpendicular null point

29. Parametrizing a null line

30. Parametrizing a null point

31. Triangle trilateral duality

32. Quadrance

33. Zero quadrance

34. Spread

35. Zero spread

36. Quadrance spread duality

37. Quadrance cross ratio

38. Triple quad formula

39. Triple spread formula

40. Complementary quadrances spreads

41. Equal quadrances spreads

42. Pythagoras

43. Pythagoras' dual

44. Spread formula

45. Spread law

46. Spread dual law

47. Quadrea

48. Quadreal

49. Quadrea quadreal product

50. Cross law

51. Cross dual law

52. Triple product relation

53. Triple cross relation

54. Midpoints

55. Midlines

56. Triple quad mid

57. Pythagoras mid

58. Cross mid

59. Couple quadrance spread

60. Three equal quadrances

61. Recursive spreads

62. Thales

63. Thales' dual

64. Right parallax

65. Right parallax dual

66. Napier's rules

67. Napier's dual rules

68. Pons Asinorum

69. Isosceles right

70. Isosceles mid

71. Isosceles triangle

72. Isosceles parallax

73. Equilateral

74. Equilateral mid

75. Triangle proportions

76. Menelaus

77. Menelaus' dual

78. Ceva

79. Ceva's dual

80. Nil cross law

81. Doubly nil triangle

82. Triply nil quadreal

83. Triply nil balance

84. Triply nil orthocenter

85. Triply nil Cevian thinness

86. Triply nil Altitude thinness

87. Singly null singly nil Thales

88. Singly null singly nil orthocenter

89. Null perspective

90. Null subtended quadrance

91. Fully nil quadrangle diagonal

92. $48/64$%

%TCIMACRO{\TeXButton{TwoColumnsEnd}{\end{multicols}}}%
%BeginExpansion
\end{multicols}%
%EndExpansion

\subsection{Overview of contents}

Section 2 introduces the main definitions of (\textit{hyperbolic) points }and
(\textit{hyperbolic) lines}, \textit{duality} and \textit{perpendicularity},
and establishes elementary but fundamental theorems on \textit{joins} and
\textit{meets} and \textit{null points} and \textit{null lines}. The context
is in the projective plane associated to a field. Less familiar are the
notions of side and vertex, which are here given rather different
definitions---a \textit{side} is basically an unordered pair of points, while
a \textit{vertex} is an unordered pair of lines. The combination of a point
and a line is called a \textit{couple}. Various important constructions arise
from studying these combinations which are prior to the investigations of
\textit{triangles} (a set of three non-collinear points) and
\textit{trilaterals }(a set of three non-concurrent lines). Null points and
null lines are parametrized, and in terms of these parametrizations
descriptions of joins and meets of null points and null lines are given.

Section 3 introduces the main metrical notions of \textit{quadrance between
points} and \textit{spread between lines}. These are dual notions. Both can be
described purely in terms of projective geometry and cross ratios. The main
laws of hyperbolic trigonometry are established: the \textit{Triple quad
formula }and its dual the \textit{Triple spread formula}, \textit{Pythagoras'
theorem }and its dual, the \textit{Spread law}, which is essentially self
dual, and the \textit{Cross law }and its dual. The important notion of the
\textit{quadrea }of a triangle and dually the \textit{quadreal }of a
trilateral are closely related to the Cross law, which as in the planar case
is the most important of the trigonometric laws. Various theorems have simpler
formulations in the case when sides have midpoints or vertices have midlines
(which play the role of angle bisectors). Finally the spread polynomials are
briefly introduced.

Section 4 studies particular special types of triangles and trilaterals, first
right triangles with the parallax theorem and Napier's rules and its dual.
Isosceles and equilateral triangles are studied. Various triangle proportions
theorems are established, such as Menelaus' theorem and Ceva's theorem and
their duals.

Section 5 introduces the rich subject of null trigonometry, concerned with
special relationships involving null points and null lines. These results are
often new, and sometimes spectacular. In particular we establish two aspects
of the thinness of triangles, some results on subtended quadrances and
spreads, and finally the curious $48/64$ theorem on the diagonal spreads of a
completely nil quadrangle.

\subsection{Thanks}

I would like to thank Rupert McCallum for comments and suggestions on the
paper. I also acknowledge the support of the Faculty of Science and the School
of Mathematics and Statistics at UNSW.

\section{Hyperbolic points and lines}

This section discusses \textit{proportions}, and then introduces the main
objects of \textit{(hyperbolic) points }and \textit{(hyperbolic) lines }and
the relations of \textit{duality }and \textit{perpendicularity}. We briefly
describe a visual model, then establish fundamental facts about \textit{joins,
meets, collinearity }and\textit{\ concurrency}. Points and lines are
parametrized in two different ways, and then we introduce the important
notions of \textit{side, vertex} and \textit{couple}. These definitions will
likely be novel to the reader. Some canonical constructions that arise from
them are investigated. \textit{Null points }and \textit{null lines }are
parametrized and formulas for their meets and joins obtained. Finally the dual
notions of \textit{triangle} and \textit{trilateral} are introduced.

\subsection{Proportions}

We work over a field $\mathbb{F},$ not of characteristic two, whose elements
are called \textbf{numbers}. Those readers who are not comfortable with the
general notion of field may restrict themselves to the field of rational
numbers $%
%TCIMACRO{\U{211a} }%
%BeginExpansion
\mathbb{Q}
%EndExpansion
$. In fact one of the remarkable consequences of this development is that
there is a \textit{complete theory of hyperbolic geometry over the rational
numbers.}

A $2$\textbf{-proportion }$x:y$ is an ordered pair of numbers $x$ and $y$,
\textit{not both zero}, with the convention that for any non-zero number
$\lambda$%
\[
x:y=\lambda x:\lambda y.
\]
This may be restated by saying that
\[
x_{1}:y_{1}=x_{2}:y_{2}%
\]
precisely when
\[
x_{1}y_{2}-x_{2}y_{1}=0.
\]

A $3$\textbf{-proportion }$x:y:z$\textbf{\ }is an ordered triple of numbers
$x,y$ and $z$, \textit{not all zero}, with the convention that for any
non-zero number $\lambda$%
\[
x:y:z=\lambda x:\lambda y:\lambda z.
\]
This may be restated by saying that
\[
x_{1}:y_{1}:z_{1}=x_{2}:y_{2}:z_{2}%
\]
precisely when the following three conditions hold:
\begin{equation}%
\begin{tabular}
[c]{lllll}%
$x_{1}y_{2}-x_{2}y_{1}=0$ &  & $y_{1}z_{2}-y_{2}z_{1}=0$ &  & $z_{1}%
x_{2}-z_{2}x_{1}=0.$%
\end{tabular}
\label{ProportionEqu}%
\end{equation}

If the context is clear, we just refer to \textit{proportions}, instead of $2
$-proportions or $3$-proportions.

\subsection{Points, lines, duality and perpendicularity}

A \textbf{(hyperbolic)} \textbf{point} is a $3$-proportion $a\equiv\left[
x:y:z\right]  $ enclosed in square brackets. A \textbf{(hyperbolic) line} is a
$3$-proportion $L\equiv\left(  l:m:n\right)  $ enclosed in round brackets.

The point $a\equiv\left[  x:y:z\right]  $ is \textbf{dual} to the line
$L\equiv\left(  l:m:n\right)  $ precisely when
\[
x:y:z=l:m:n.
\]
In this case we write $a^{\bot}=L$ or $L^{\bot}=a.$ Then
\[%
\begin{tabular}
[c]{lllll}%
$\left(  a^{\bot}\right)  ^{\bot}=a$ &  & \textrm{and} &  & $\left(  L^{\bot
}\right)  ^{\bot}=L.$%
\end{tabular}
\]
Each point is dual to exactly one line, and conversely. We will often, but not
always, use the notational convention of corresponding \textit{small letters}
for \textit{points} and \textit{capital letters }for \textit{dual lines}, for
example $a$ for a point and $A\equiv a^{\bot}$ for the dual line, or $L$ for a
line and $l\equiv L^{\bot}$ for the dual point.

The point $a\equiv\left[  x:y:z\right]  $ \textbf{lies on} the line
$L\equiv\left(  l:m:n\right)  $, or equivalently $L$ \textbf{passes through}
$a,$ precisely when%
\[
lx+my-nz=0.
\]
The point $a\equiv\left[  x:y:z\right]  $ is \textbf{null} precisely when it
lies on its dual line, in other words when%
\[
x^{2}+y^{2}-z^{2}=0.
\]
The line $L\equiv\left(  l:m:n\right)  $ is \textbf{null }precisely when it
passes through its dual point, in other words when
\[
l^{2}+m^{2}-n^{2}=0.
\]
The dual of a null point is a null line and conversely.

Points $a_{1}\equiv\left[  x_{1}:y_{1}:z_{1}\right]  $ and $a_{2}\equiv\left[
x_{2}:y_{2}:z_{2}\right]  $ are \textbf{perpendicular} precisely when
\[
x_{1}x_{2}+y_{1}y_{2}-z_{1}z_{2}=0.
\]
This is equivalent to the condition that $a_{1}$ lies on $a_{2}^{\bot},$ or
that $a_{2}$ lies on $a_{1}^{\bot}.$

Lines $L_{1}\equiv\left(  l_{1}:m_{1}:n_{1}\right)  $ and $L_{2}\equiv\left(
l_{2}:m_{2}:n_{2}\right)  $ are \textbf{perpendicular} precisely when
\[
l_{1}l_{2}+m_{1}m_{2}-n_{1}n_{2}=0.
\]
This is equivalent to the condition that $L_{1}$ passes through $L_{2}^{\bot
},$ or that $L_{2}$ passes through $L_{1}^{\bot}$.

\textit{There is a complete duality in the theory between points and lines.}
Any result thus has a corresponding dual result, obtained by interchanging the
roles of points and lines. We call this the \textit{duality principle}, and
use it often to eliminate repetition of statements and proofs of theorems. To
maintain this principle, \textit{we treat points and lines symmetrically},
especially initially. Later in the paper we leave the formulation of dual
statements to the reader.

We denote by $\mathbb{F}^{3}$ the three-dimensional space of \textbf{vectors}
$v\equiv\left(  x,y,z\right)  $. If $v\equiv\left(  x,y,z\right)  $ has
coordinates which are not all zero, then $\left[  v\right]  \equiv\left[
x:y:z\right]  $ denotes the corresponding (hyperbolic) point, and $\left(
v\right)  \equiv\left(  x:y:z\right)  $ denotes the corresponding (hyperbolic) line.

\subsection{Visualizing universal hyperbolic geometry}

Although not logically necessary, let's point out one way of visualizing the
subject. This is essentially the Beltrami Klein view, but extended beyond the
unit disk, and with underlying field the rational numbers $\mathbb{Q}.$ Think
of a point $a\equiv\left[  x:y:z\right]  $ as representing the central line
(one-dimensional subspace) in three-dimensional space $\mathbb{Q}^{3}$ through
the origin and the point $\left[  x,y,z\right]  $. Think of the line
$L\equiv\left(  l:m:n\right)  $ as being the central plane (two-dimensional
subspace) with equation $lx+my-nz=0.$ So the notion `$a$ lies on $L$' has the
usual interpretation.

Both points and lines are projective objects, and may be illustrated in
diagrams using their meets with the plane $z=1$ in the usual way. So the
(hyperbolic) point $a\equiv\left[  x:y:z\right]  $ becomes generically the
planar point
\[
\left[  X,Y\right]  \equiv\left[  \frac{x}{z},\frac{y}{z}\right]
\]
while the (hyperbolic) line $L\equiv\left(  l:m:n\right)  $ becomes
generically the planar line with equation%
\[
lX+mY=n.
\]
Otherwise in case $z=0,$ or $l=m=0,$ we have respectively points at infinity
and a line at infinity, represented in the usual way by directions in the
$z=1$ plane.

Null points become planar points on the unit circle $X^{2}+Y^{2}=1.$ (Over
other fields, there may also be two null points at infinity, but this does not
happen over the rational numbers). Null lines become tangent lines to this
unit circle. Both the interior and exterior of the unit circle are important,
the former corresponding to hyperboloids of two sheets, the latter to
hyperboloids of one sheet in the ambient three-dimensional relativistic space.
Neither has preference over the other. In fact \textit{the most important
objects are the null points and null lines}.

The duality between points and lines becomes the pole-polar duality of
projective geometry, since the planar point $\left[  a,b\right]  $ is the pole
of the planar line $aX+bY=1$ with respect to the unit circle $X^{2}+Y^{2}=1.$
Perpendicularity then becomes a straightedge construction, since the pole of a
line or the polar of a point may be so constructed, once the unit circle is
given. The metrical structure comes from the quadratic form $x^{2}+y^{2}%
-z^{2}$ in the ambient space.

For an interesting variant, represent the projective plane by intersecting the
three-dimensional space $\mathbb{Q}^{3}$ with the plane $x=1.$ This provides a
more `hyperbolic view'; less familiar, but also worthy of study.

Let's emphasize that the following development of the subject is logically
independent of any one visual interpretation of it.

\subsection{Joins and meets}

\begin{theorem}
[Join of points]If $a_{1}\equiv\left[  x_{1}:y_{1}:z_{1}\right]  $ and
$a_{2}\equiv\left[  x_{2}:y_{2}:z_{2}\right]  $ are distinct points, then
there is exactly one line $L$ which passes through them both, namely%
\[
L\equiv a_{1}a_{2}\equiv\left(  y_{1}z_{2}-y_{2}z_{1}:z_{1}x_{2}-z_{2}%
x_{1}:x_{2}y_{1}-x_{1}y_{2}\right)  .
\]

\end{theorem}

\begin{proof}
The $3$-proportion
\begin{equation}
y_{1}z_{2}-y_{2}z_{1}:z_{1}x_{2}-z_{2}x_{1}:x_{2}y_{1}-x_{1}y_{2}%
\label{JoinProp}%
\end{equation}
is well-defined, since if we multiply the coordinates of either $a_{1}$ or
$a_{2}$ by a non-zero number, then each term in (\ref{JoinProp}) is
correspondingly changed, and since $a_{1}$ and $a_{2}$ are distinct, at least
one of the three terms is non-zero from (\ref{ProportionEqu}).

The line
\[
L\equiv\left(  y_{1}z_{2}-y_{2}z_{1}:z_{1}x_{2}-z_{2}x_{1}:x_{2}y_{1}%
-x_{1}y_{2}\right)
\]
passes through both $a_{1}$ and $a_{2},$ since
\begin{align*}
\left(  y_{1}z_{2}-y_{2}z_{1}\right)  x_{1}+\left(  z_{1}x_{2}-z_{2}%
x_{1}\right)  y_{1}-\left(  x_{2}y_{1}-x_{1}y_{2}\right)  z_{1}  & =0\\
\left(  y_{1}z_{2}-y_{2}z_{1}\right)  x_{2}+\left(  z_{1}x_{2}-z_{2}%
x_{1}\right)  y_{2}-\left(  x_{2}y_{1}-x_{1}y_{2}\right)  z_{2}  & =0.
\end{align*}
The $3$-proportions $x_{1}:y_{1}:z_{1}$ and $x_{2}:y_{2}:z_{2}$ are by
assumption unequal, so the system of equations
\begin{align*}
lx_{1}+my_{1}-nz_{1} &  =0\\
lx_{2}+my_{2}-nz_{2} &  =0
\end{align*}
has up to a multiple exactly one solution, showing that $L$ is unique.$%
%TCIMACRO{\TeXButton{Proof box}{{\hspace{.1in} \rule{0.5em}{0.5em}}}}%
%BeginExpansion
{\hspace{.1in} \rule{0.5em}{0.5em}}%
%EndExpansion
$
\end{proof}

The line $L\equiv a_{1}a_{2}$ is the \textbf{join} of the points $a_{1}$ and
$a_{2}$.

\begin{theorem}
[Meet of lines]If $L_{1}\equiv\left(  l_{1}:m_{1}:n_{1}\right)  $ and
$L_{2}\equiv\left(  l_{2}:m_{2}:n_{2}\right)  $ are distinct lines, then there
is exactly one point $a$ which lies on them both, namely%
\[
a\equiv L_{1}L_{2}\equiv\left[  m_{1}n_{2}-m_{2}n_{1}:n_{1}l_{2}-n_{2}%
l_{1}:l_{2}m_{1}-l_{1}m_{2}\right]  .
\]

\end{theorem}

\begin{proof}
This is dual to the Join of points theorem.$%
%TCIMACRO{\TeXButton{Proof box}{{\hspace{.1in} \rule{0.5em}{0.5em}}}}%
%BeginExpansion
{\hspace{.1in} \rule{0.5em}{0.5em}}%
%EndExpansion
$
\end{proof}

The point $a\equiv L_{1}L_{2}$ is the\textbf{\ meet} of the lines $L_{1}$ and
$L_{2}.$ Note that for any distinct points $a_{1}$ and $a_{2}$
\[
\left(  a_{1}a_{2}\right)  ^{\bot}=a_{1}^{\bot}a_{2}^{\bot}.
\]
Similarly for any distinct lines $L_{1}$ and $L_{2}$
\[
\left(  L_{1}L_{2}\right)  ^{\bot}=L_{1}^{\bot}L_{2}^{\bot}.
\]

Both of the previous theorems involve implicitly the \textbf{hyperbolic cross
product function }%
\[
J\left(  x_{1},y_{1},z_{1};x_{2},y_{2},z_{2}\right)  \equiv\left(  y_{1}%
z_{2}-y_{2}z_{1},z_{1}x_{2}-z_{2}x_{1},x_{2}y_{1}-x_{1}y_{2}\right)  .
\]
This is a hyperbolic version of the more familiar \textit{Euclidean cross
product}, and it enjoys many of the same properties. It is ubiquitous in the
two-dimensional hyperbolic geometry developed in this paper, and many
computations amount essentially to repeated evaluations of this function.

\subsection{Collinear points and concurrent lines}

Three or more points which lie on a common line are \textbf{collinear}. Three
or more lines which pass through a common point are \textbf{concurrent}.

\begin{theorem}
[Collinear points]The points $a_{1}\equiv\left[  x_{1}:y_{1}:z_{1}\right]  $,
$a_{2}\equiv\left[  x_{2}:y_{2}:z_{2}\right]  $ and $a_{3}\equiv\left[
x_{3}:y_{3}:z_{3}\right]  $ are collinear precisely when
\[
\allowbreak x_{1}y_{2}z_{3}-x_{1}y_{3}z_{2}+x_{2}y_{3}z_{1}-x_{3}y_{2}%
z_{1}+x_{3}y_{1}z_{2}-x_{2}y_{1}z_{3}=0.
\]

\end{theorem}

\begin{proof}
If two of the three points are distinct, say $a_{1}$ and $a_{2},$ then by the
Join of points theorem
\[
a_{1}a_{2}=\left(  y_{1}z_{2}-y_{2}z_{1}:z_{1}x_{2}-z_{2}x_{1}:x_{2}%
y_{1}-x_{1}y_{2}\right)
\]
and $a_{3}$ lies on $a_{1}a_{2}$ precisely when%
\[
\left(  y_{1}z_{2}-y_{2}z_{1}\right)  x_{3}+\left(  z_{1}x_{2}-z_{2}%
x_{1}\right)  y_{3}-\left(  x_{2}y_{1}-x_{1}y_{2}\right)  z_{3}=0.
\]
This is the condition of the theorem. If all points are identical, then they
are collinear, and the expression
\[
\allowbreak x_{1}y_{2}z_{3}-x_{1}y_{3}z_{2}+x_{2}y_{3}z_{1}-x_{3}y_{2}%
z_{1}+x_{3}y_{1}z_{2}-x_{2}y_{1}z_{3}%
\]
is by symmetry zero.$%
%TCIMACRO{\TeXButton{Proof box}{{\hspace{.1in} \rule{0.5em}{0.5em}}}}%
%BeginExpansion
{\hspace{.1in} \rule{0.5em}{0.5em}}%
%EndExpansion
$
\end{proof}

\begin{theorem}
[Concurrent lines]The lines $L_{1}\equiv\left(  l_{1}:m_{1}:n_{1}\right)  $,
$L_{2}\equiv\left(  l_{2}:m_{2}:n_{2}\right)  $ and $L_{3}\equiv\left(
l_{3}:m_{3}:n_{3}\right)  $ are concurrent precisely when
\[
l_{1}m_{2}n_{3}-l_{1}m_{3}n_{2}+l_{2}m_{3}n_{1}-l_{3}m_{2}n_{1}+l_{3}%
m_{1}n_{2}-l_{2}m_{1}n_{3}=0.
\]

\end{theorem}

\begin{proof}
This is dual to the Collinear points theorem.$%
%TCIMACRO{\TeXButton{Proof box}{{\hspace{.1in} \rule{0.5em}{0.5em}}}}%
%BeginExpansion
{\hspace{.1in} \rule{0.5em}{0.5em}}%
%EndExpansion
$
\end{proof}

The formulas of the two previous theorems can be recast as the determinantal
equations$\allowbreak$%
\[%
\begin{tabular}
[c]{lllll}%
$%
\begin{vmatrix}
x_{1} & y_{1} & z_{1}\\
x_{2} & y_{2} & z_{2}\\
x_{3} & y_{3} & z_{3}%
\end{vmatrix}
=0$ &  & \textrm{and} &  & $%
\begin{vmatrix}
l_{1} & l_{2} & l_{3}\\
m_{1} & m_{2} & m_{3}\\
n_{1} & n_{2} & n_{3}%
\end{vmatrix}
=0.$%
\end{tabular}
\]

\begin{theorem}
[Line through null points]Any line $L$ passes through at most two null points.
\end{theorem}

\begin{proof}
If $L\equiv\left(  l:m:n\right)  $ then any null point $\alpha\equiv\left[
x:y:z\right]  $ lying on $L$ satisfies%
\begin{equation}
lx+my-nz=0\label{LineNullpoints1}%
\end{equation}
and
\begin{equation}
x^{2}+y^{2}-z^{2}=0.\label{LineNullpoints2}%
\end{equation}
After substituting (\ref{LineNullpoints1}) into (\ref{LineNullpoints2}) you
get a homogeneous quadratic equation in two variables. This has at most two
solutions up to a factor, so there are at most two solutions $x:y:z$ to this
pair of equations.$%
%TCIMACRO{\TeXButton{Proof box}{{\hspace{.1in} \rule{0.5em}{0.5em}}}}%
%BeginExpansion
{\hspace{.1in} \rule{0.5em}{0.5em}}%
%EndExpansion
$
\end{proof}

\begin{theorem}
[Point on null lines]Any point $a$ lies on at most two null lines.
\end{theorem}

\begin{proof}
This is dual to the Line through null points theorem.$%
%TCIMACRO{\TeXButton{Proof box}{{\hspace{.1in} \rule{0.5em}{0.5em}}}}%
%BeginExpansion
{\hspace{.1in} \rule{0.5em}{0.5em}}%
%EndExpansion
$
\end{proof}

The two previous theorems lead naturally to an important characterization of
non-null points and lines, in terms of their relations with null points and
lines. A non-null point $a$ is defined to be \textbf{internal} precisely when
it lies on no null lines, and is \textbf{external }precisely when it lies on
two null lines. A non-null line $L$ is defined to be \textbf{internal}
precisely when it passes through two null points, and is \textbf{external
}precisely when it passes through no null points. This way points and lines
are either \textit{internal},\textit{\ null }or \textit{external}.

These notions play a major role when we discuss \textit{isometries} in a
future paper, but they are not necessary for the development of basic
trigonometry, indeed the main theorems of trigonometry apply equally to both
internal and external points and lines. So we will not develop these concepts
further in this paper.

\subsection{Sides and vertices}

A \textbf{side} $\overline{a_{1}a_{2}}$ is a set $\left\{  a_{1}%
,a_{2}\right\}  $ of two points. A \textbf{vertex} $\overline{L_{1}L_{2}} $ is
a set $\left\{  L_{1},L_{2}\right\}  $ of two lines. Clearly
\[
\overline{a_{1}a_{2}}=\overline{a_{2}a_{1}}\qquad\mathrm{and}\qquad
\overline{L_{1}L_{2}}=\overline{L_{2}L_{1}}.
\]

If $\overline{a_{1}a_{2}}$ is a side, then $a_{1}a_{2}$ is the \textbf{line
}of the side. The side $\overline{a_{1}a_{2}}$ is a \textbf{null side
}precisely when $a_{1}a_{2}$ is a null line. The side $\overline{a_{1}a_{2}}$
is a \textbf{nil side} precisely when at least one of $a_{1}$ or $a_{2}$ is a
null point. In this case it is a \textbf{singly-nil side} precisely when
exactly one of $a_{1}$ or $a_{2}$ is a null point, and a \textbf{doubly-nil
side} precisely when both $a_{1}$ and $a_{2}$ are null points.

If $\overline{L_{1}L_{2}}$ is a vertex, then $L_{1}L_{2}$ is the \textbf{point
}of the vertex. The vertex $\overline{L_{1}L_{2}}$ is a \textbf{null vertex}
precisely when $L_{1}L_{2}$ is a null point. The vertex $\overline{L_{1}L_{2}%
}$ is a \textbf{nil vertex} precisely when at least one of $L_{1}$ or $L_{2}$
is a null line. In this case it is a \textbf{singly-nil vertex }precisely when
exactly one of $L_{1}$ or $L_{2}$ is a null line, and a \textbf{doubly-nil
vertex} precisely when both $L_{1}$ and $L_{2}$ are null lines.

The \textbf{dual }of the side $\overline{a_{1}a_{2}}$ is the vertex
$\overline{a_{1}^{\bot}a_{2}^{\bot}}$. The \textbf{dual }of the vertex
$\overline{L_{1}L_{2}}$ is the side $\overline{L_{1}^{\bot}L_{2}^{\bot}}.$

The side $\overline{a_{1}a_{2}}$ is a \textbf{right} \textbf{side }precisely
when $a_{1}$ is perpendicular to $a_{2}.$ The vertex $\overline{L_{1}L_{2}}$
is a \textbf{right} \textbf{vertex }precisely when $L_{1}$ is perpendicular to
$L_{2}.$

\begin{theorem}
[Perpendicular point]For any side $\overline{a_{1}a_{2}}$ there is a unique
point $p$ which is perpendicular to both $a_{1}$ and $a_{2},$ namely%
\[
p\equiv a_{1}^{\bot}a_{2}^{\bot}=\left(  a_{1}a_{2}\right)  ^{\bot}.
\]

\end{theorem}

\begin{proof}
Any point $p$ which is perpendicular to both $a_{1}$ and $a_{2}$ must lie on
both $a_{1}^{\bot}$ and $a_{2}^{\bot},$ and since $a_{1}$ and $a_{2}$ are
distinct, the Meet of lines theorem asserts that there is exactly one such
point, namely $p\equiv a_{1}^{\bot}a_{2}^{\bot}=\left(  a_{1}a_{2}\right)
^{\bot}$.$%
%TCIMACRO{\TeXButton{Proof box}{{\hspace{.1in} \rule{0.5em}{0.5em}}}}%
%BeginExpansion
{\hspace{.1in} \rule{0.5em}{0.5em}}%
%EndExpansion
$
\end{proof}

The point $p$ is the \textbf{perpendicular point }of $\overline{a_{1}a_{2}}$.
It may happen that $p$ lies on $a_{1}a_{2}$; this will occur precisely when
$\overline{a_{1}a_{2}}$ is a null side.

\begin{theorem}
[Perpendicular line]For any vertex $\overline{L_{1}L_{2}}$ there is a unique
line $P$ which is perpendicular to both $L_{1}$ and $L_{2},$ namely%
\[
P\equiv L_{1}^{\bot}L_{2}^{\bot}=\left(  L_{1}L_{2}\right)  ^{\bot}.
\]

\end{theorem}

\begin{proof}
This is dual to the Perpendicular point theorem.$%
%TCIMACRO{\TeXButton{Proof box}{{\hspace{.1in} \rule{0.5em}{0.5em}}}}%
%BeginExpansion
{\hspace{.1in} \rule{0.5em}{0.5em}}%
%EndExpansion
$
\end{proof}

The line $P$ is the \textbf{perpendicular line }of $\overline{L_{1}L_{2}}$. It
may happen that $P$ passes through $L_{1}L_{2}$; this will occur precisely
when $\overline{L_{1}L_{2}}$ is a null vertex.

\begin{example}
Consider the distinct points $a_{1}\equiv\left[  x:0:1\right]  $ and
$a_{2}\equiv\left[  0:y:1\right]  $. Then $a_{1}a_{2}=\left(  y:x:xy\right)
\ $and so the perpendicular point of the side $\overline{a_{1}a_{2}}$ is
$p=\left[  y:x:xy\right]  $.$%
%TCIMACRO{\TeXButton{Math Diamond}{\hspace{.1in}\diamond}}%
%BeginExpansion
\hspace{.1in}\diamond
%EndExpansion
$
\end{example}

\begin{example}
Consider the distinct lines $L_{1}\equiv\left(  l_{1}:m_{1}:0\right)  $ and
$L_{2}\equiv\left(  l_{2}:m_{2}:0\right)  $. Then $L_{1}L_{2}=\left[
0:0:1\right]  $ and so the perpendicular line of the vertex $\overline
{L_{1}L_{2}}$ is $P=\left(  0:0:1\right)  $.$%
%TCIMACRO{\TeXButton{Math Diamond}{\hspace{.1in}\diamond}}%
%BeginExpansion
\hspace{.1in}\diamond
%EndExpansion
$
\end{example}

\begin{theorem}
[Opposite points]For any \textit{non-null }side $\overline{a_{1}a_{2}}$ there
is a unique point $o_{1}$ which lies on $a_{1}a_{2}$ and is perpendicular to
$a_{1},$ namely%
\[
o_{1}\equiv\left(  a_{1}a_{2}\right)  a_{1}^{\bot},
\]
and there is a unique point $o_{2}$ which lies on $a_{1}a_{2}$ and is
perpendicular to $a_{2},$ namely
\[%
\begin{tabular}
[c]{l}%
$o_{2}\equiv\left(  a_{1}a_{2}\right)  a_{2}^{\bot}.$%
\end{tabular}
\]
The points $o_{1}$ and $o_{2}$ are distinct. If $a_{1}$ is a null point, then
$o_{1}=a_{1},$ and if $a_{2}$ is a null point, then $o_{2}=a_{2}.$ If
$a_{1}\equiv\left[  v_{1}\right]  $ and $a_{2}\equiv\left[  v_{2}\right]  $
with vectors $v_{1}\equiv\left(  x_{1},y_{1},z_{1}\right)  $ and $v_{2}%
\equiv\left(  x_{2},y_{2},z_{2}\right)  $, then
\begin{align*}
o_{1}  & =\left[  \left(  x_{1}x_{2}+y_{1}y_{2}-z_{1}z_{2}\right)
v_{1}-\left(  x_{1}^{2}+y_{1}^{2}-z_{1}^{2}\right)  v_{2}\right] \\
o_{2}  & =\left[  \left(  x_{2}^{2}+y_{2}^{2}-z_{2}^{2}\right)  v_{1}-\left(
x_{1}x_{2}+y_{1}y_{2}-z_{1}z_{2}\right)  v_{2}\right]  .
\end{align*}

\end{theorem}

\begin{proof}
If $a_{1}a_{2}=a_{1}^{\bot}$ then $a_{1}$ lies on its dual $a_{1}^{\bot},$ so
that $a_{1}$ is a null point, and so $a_{1}^{\bot}=a_{1}a_{2}$ is a null line.
Since we assume that $\overline{a_{1}a_{2}}$ is a non-null side, we conclude
that $a_{1}^{\bot}$ is distinct from $a_{1}a_{2},$ so that $o_{1}\equiv\left(
a_{1}a_{2}\right)  a_{1}^{\bot}$ is well-defined, and is the unique point
lying on $a_{1}a_{2}$ which is perpendicular to $a_{1}.$ Similarly
$o_{2}\equiv\left(  a_{1}a_{2}\right)  a_{2}^{\bot}$ is the unique point lying
on $a_{1}a_{2}$ which is perpendicular to $a_{2}.$

If $o_{1}=o_{2}$ then $o_{1}=a_{1}^{\bot}a_{2}^{\bot}=\left(  a_{1}%
a_{2}\right)  ^{\bot}$ lies on its dual $a_{1}a_{2},$ which would imply that
$a_{1}a_{2}$ is a null line, which contradicts the assumption that
$\overline{a_{1}a_{2}}$ is a non-null side. So $o_{1}$ and $o_{2}$ are distinct.

If $a_{1}$ is a null point, then $a_{1}$ lies on both $a_{1}^{\perp}$ and
$a_{1}a_{2},$ so $o_{1}=a_{1}.$ Similarly if $a_{2}$ is a null point then
$o_{2}=a_{2}.$

Now suppose that $a_{1}\equiv\left[  v_{1}\right]  $ and $a_{2}\equiv\left[
v_{2}\right]  $ where $v_{1}\equiv\left(  x_{1},y_{1},z_{1}\right)  $ and
$v_{2}\equiv\left(  x_{2},y_{2},z_{2}\right)  $ are vectors. Then from the
Join of points theorem and the Meet of lines theorem,
\[
a_{1}a_{2}=\left(  y_{1}z_{2}-y_{2}z_{1}:z_{1}x_{2}-z_{2}x_{1}:x_{2}%
y_{1}-x_{1}y_{2}\right)
\]
and
\begin{align*}
& \left(  a_{1}a_{2}\right)  a_{1}^{\bot}\\
& =\left[
\begin{array}
[c]{c}%
\left(  z_{1}x_{2}-z_{2}x_{1}\right)  z_{1}-y_{1}\left(  x_{2}y_{1}-x_{1}%
y_{2}\right)  :\left(  x_{2}y_{1}-x_{1}y_{2}\right)  x_{1}-z_{1}\left(
y_{1}z_{2}-y_{2}z_{1}\right) \\
:x_{1}\left(  z_{1}x_{2}-z_{2}x_{1}\right)  -\left(  y_{1}z_{2}-y_{2}%
z_{1}\right)  y_{1}%
\end{array}
\right] \\
& =\left[  x_{1}y_{1}y_{2}-x_{1}z_{1}z_{2}-x_{2}y_{1}^{2}+x_{2}z_{1}^{2}%
:x_{1}x_{2}y_{1}-y_{1}z_{1}z_{2}-x_{1}^{2}y_{2}+y_{2}z_{1}^{2}:x_{1}x_{2}%
z_{1}+y_{1}y_{2}z_{1}-x_{1}^{2}z_{2}-y_{1}^{2}z_{2}\right]  .
\end{align*}
On the other hand$\allowbreak$
\begin{align*}
& \left(  x_{1}x_{2}+y_{1}y_{2}-z_{1}z_{2}\right)  \left(  x_{1},y_{1}%
,z_{1}\right)  -\left(  x_{1}^{2}+y_{1}^{2}-z_{1}^{2}\right)  \left(
x_{2},y_{2},z_{2}\right) \\
& =\left(  x_{1}y_{1}y_{2}-x_{1}z_{1}z_{2}-x_{2}y_{1}^{2}+x_{2}z_{1}^{2}%
,x_{1}x_{2}y_{1}-y_{1}z_{1}z_{2}-x_{1}^{2}y_{2}+y_{2}z_{1}^{2},x_{1}x_{2}%
z_{1}+y_{1}y_{2}z_{1}-x_{1}^{2}z_{2}-y_{1}^{2}z_{2}\right)  .
\end{align*}
So
\[
o_{1}\equiv\left(  a_{1}a_{2}\right)  a_{1}^{\bot}=\left[  \left(  x_{1}%
x_{2}+y_{1}y_{2}-z_{1}z_{2}\right)  v_{1}-\left(  x_{1}^{2}+y_{1}^{2}%
-z_{1}^{2}\right)  v_{2}\right]  .
\]
Interchanging the indices gives
\[
o_{2}\equiv\left(  a_{1}a_{2}\right)  a_{2}^{\bot}=\left[  \left(  x_{2}%
^{2}+y_{2}^{2}-z_{2}^{2}\right)  v_{1}-\left(  x_{1}x_{2}+y_{1}y_{2}%
-z_{1}z_{2}\right)  v_{2}\right]  .%
%TCIMACRO{\TeXButton{Proof box}{{\hspace{.1in} \rule{0.5em}{0.5em}}}}%
%BeginExpansion
{\hspace{.1in} \rule{0.5em}{0.5em}}%
%EndExpansion
\]

\end{proof}

The points $o_{1}$ and $o_{2}$ are the \textbf{opposite points }of the side
$\overline{a_{1}a_{2}}$. Since $o_{1}$ and $o_{2}$ are distinct, the side
$\overline{o_{1}o_{2}}$ is well-defined, and non-null since $o_{1}o_{2}%
=a_{1}a_{2}$. The side $\overline{o_{1}o_{2}}$ is \textbf{opposite} to the
side $\overline{a_{1}a_{2}}$. This relationship is involutory: $\overline
{a_{1}a_{2}}$ is also opposite to $\overline{o_{1}o_{2}}.$

\begin{theorem}
[Opposite lines]For any \textit{non-null} vertex $\overline{L_{1}L_{2}}$ there
is a unique line $O_{1}$ which passes through $L_{1}L_{2}$ and is
perpendicular to $L_{1},$ namely%
\[
O_{1}\equiv\left(  L_{1}L_{2}\right)  L_{1}^{\bot},
\]
and there is a unique line $O_{2}$ which passes through $L_{1}L_{2}$ and is
perpendicular to $L_{2},$ namely
\[%
\begin{tabular}
[c]{l}%
$O_{2}\equiv\left(  L_{1}L_{2}\right)  L_{2}^{\bot}.$%
\end{tabular}
\]
The lines $O_{1}$ and $O_{2}$ are distinct. If $L_{1}$ is a null line, then
$O_{1}=L_{1},$ and if $L_{2}$ is a null line, then $O_{2}=L_{2}.$ If
$L_{1}=\left(  v_{1}\right)  $ and $L_{2}=\left(  v_{2}\right)  $ with vectors
$v_{1}=\left(  l_{1},m_{1},n_{1}\right)  $ and $v_{2}=\left(  l_{2}%
,m_{2},n_{2}\right)  $, then
\begin{align*}
O_{1}  & =\left(  \left(  l_{1}l_{2}+m_{1}m_{2}-n_{1}n_{2}\right)
v_{1}-\left(  l_{1}^{2}+m_{1}^{2}-n_{1}^{2}\right)  v_{2}\right) \\
O_{2}  & =\left(  \left(  l_{2}^{2}+m_{2}^{2}-n_{2}^{2}\right)  v_{1}-\left(
l_{1}l_{2}+m_{1}m_{2}-n_{1}n_{2}\right)  v_{2}\right)  .
\end{align*}

\end{theorem}

\begin{proof}
This is dual to the Opposite points theorem.$%
%TCIMACRO{\TeXButton{Proof box}{{\hspace{.1in} \rule{0.5em}{0.5em}}}}%
%BeginExpansion
{\hspace{.1in} \rule{0.5em}{0.5em}}%
%EndExpansion
$
\end{proof}

The lines $O_{1}$ and $O_{2}$ are the \textbf{opposite lines }of the vertex
$\overline{L_{1}L_{2}}$. Since $O_{1}$ and $O_{2}$ are distinct, the vertex
$\overline{O_{1}O_{2}}$ is well-defined, and non-null since $O_{1}O_{2}%
=A_{1}A_{2}$. The vertex $\overline{O_{1}O_{2}}$ is \textbf{opposite} to the
vertex $\overline{A_{1}A_{2}}$. This relationship is involutory:
$\overline{L_{1}L_{2}}$ is also opposite to $\overline{O_{1}O_{2}}.$

If $L_{1}=a_{1}^{\bot}$ and $L_{2}=a_{2}^{\bot},$ then the opposite points
$o_{1}$ and $o_{2}$ of $\overline{a_{1}a_{2}}$ are dual respectively to the
opposite lines $O_{1}$ and $O_{2}$ of $\overline{L_{1}L_{2}}$.

\begin{example}
Consider the distinct points $a_{1}\equiv\left[  x:0:1\right]  $ and
$a_{2}\equiv\left[  0:y:1\right]  $. Then the opposite points of
$\overline{a_{1}a_{2}}$ are%
\[%
\begin{tabular}
[c]{lllll}%
$o_{1}=\left[  x:y\left(  x^{2}-1\right)  :x^{2}\right]  $ &  & \textit{and} &
& $o_{2}=\left[  x\left(  y^{2}-1\right)  :y:y^{2}\right]  .%
%TCIMACRO{\TeXButton{Math Diamond}{\hspace{.1in}\diamond}}%
%BeginExpansion
\hspace{.1in}\diamond
%EndExpansion
$%
\end{tabular}
\]

\end{example}

\begin{example}
Consider the distinct lines $L_{1}\equiv\left(  l_{1}:m_{1}:0\right)  $ and
$L_{2}\equiv\left(  l_{2}:m_{2}:0\right)  $. Then the opposite lines of
$\overline{L_{1}L_{2}}$ are%
\[%
\begin{tabular}
[c]{lllll}%
$O_{1}=\left(  -m_{1}:l_{1}:0\right)  $ &  & \textit{and} &  & $O_{2}=\left(
-m_{2}:l_{2}:0\right)  .%
%TCIMACRO{\TeXButton{Math Diamond}{\hspace{.1in}\diamond}}%
%BeginExpansion
\hspace{.1in}\diamond
%EndExpansion
$%
\end{tabular}
\]

\end{example}

\subsection{Couples}

A \textbf{couple} $\overline{aL}$ is a set $\left\{  a,L\right\}  $ consisting
of a point $a$ and a line $L$ such that $a$ does not lie on $L$. The
\textbf{dual }of the couple $\overline{aL}$ is the couple $\overline{L^{\bot
}a^{\bot}}.$ The couple $\overline{aL}$ is dual to itself precisely when $a$
is dual to $L,$ in which case we say $\overline{aL}$ is a \textbf{dual
couple}. A couple $\overline{aL}$ is \textbf{null }precisely when $a$ is a
null point or $L$ is a null line (or both).

\begin{theorem}
[Altitude line]For any non-dual couple $\overline{aL}$ there is a unique line
$N$ which passes through $a$ and is perpendicular to $L,$ namely%
\[
N\equiv aL^{\bot}.
\]

\end{theorem}

\begin{proof}
Any line $N$ which passes through $a$ and is perpendicular to $L$ must also
pass through $L^{\perp},$ and since $a$ and $L^{\perp}$ are by assumption
distinct there is exactly one such line, namely $N\equiv aL^{\perp}.%
%TCIMACRO{\TeXButton{Proof box}{{\hspace{.1in} \rule{0.5em}{0.5em}}}}%
%BeginExpansion
{\hspace{.1in} \rule{0.5em}{0.5em}}%
%EndExpansion
$
\end{proof}

The line $N\equiv aL^{\bot}$ is the \textbf{altitude line }of $\overline{aL},$
or the \textbf{altitude line to }$L$\textbf{\ through }$a.$

\begin{theorem}
[Altitude point]For any non-dual couple $\overline{aL}$ there is a unique
point $n$ which lies on $L$ and is perpendicular to $a,$ namely%
\[
n\equiv a^{\bot}L.
\]

\end{theorem}

\begin{proof}
This is dual to the Altitude line theorem. $%
%TCIMACRO{\TeXButton{Proof box}{{\hspace{.1in} \rule{0.5em}{0.5em}}}}%
%BeginExpansion
{\hspace{.1in} \rule{0.5em}{0.5em}}%
%EndExpansion
$
\end{proof}

The point $n\equiv a^{\bot}L$ is the \textbf{altitude point of }$\overline
{aL}$, or the \textbf{altitude point to }$a$\textbf{\ on }$L.$ The altitude
line $N$ and the altitude point $n$ of $\overline{aL}$ are dual.

If $a\equiv\left[  x:y:z\right]  $ and $L\equiv\left(  k:l:m\right)  $, then
the coefficients of $N$ and $n$ are both given by $J\left(
x,y,z;k,l,m\right)  $.

\begin{theorem}
[Parallel line]For any non-dual couple $\overline{aL}$ there is a unique line
$R$ which passes through $a$ and is perpendicular to the altitude line $N$ of
$\overline{aL}$, namely%
\[
R\equiv a\left(  a^{\bot}L\right)  .
\]

\end{theorem}

\begin{proof}
The line $R\equiv a\left(  a^{\bot}L\right)  $ is well-defined since $a$ does
not lie on $L.$ Furthermore $R$ passes through $a$ and is perpendicular to
$N\equiv aL^{\perp},$ and any such line must be $a\left(  a^{\bot}L\right)
$.$%
%TCIMACRO{\TeXButton{Proof box}{{\hspace{.1in} \rule{0.5em}{0.5em}}}}%
%BeginExpansion
{\hspace{.1in} \rule{0.5em}{0.5em}}%
%EndExpansion
$
\end{proof}

The line $R\equiv a\left(  a^{\bot}L\right)  $ is the \textbf{parallel line
}of\textbf{\ }$\overline{aL},$ or the \textbf{parallel line to }%
$L$\textbf{\ through }$a$. This is the only notion of parallel line that
appears in this treatment of the subject. Note that parallel lines are only
defined \textit{relative to a couple}, that is a point and a line, and
\textit{are unique}. Ironically, Euclid's parallel postulate is alive and well
in universal hyperbolic geometry!

\begin{theorem}
[Parallel point]For any non-dual couple $\overline{aL}$ there is a unique
point $r$ which lies on $a^{\perp}$ and is perpendicular to the altitude point
$n$ of $\overline{aL}$, namely%
\[
r\equiv a^{\bot}\left(  aL^{\bot}\right)  .
\]

\end{theorem}

\begin{proof}
This is dual to the Parallel line theorem.$%
%TCIMACRO{\TeXButton{Proof box}{{\hspace{.1in} \rule{0.5em}{0.5em}}}}%
%BeginExpansion
{\hspace{.1in} \rule{0.5em}{0.5em}}%
%EndExpansion
$
\end{proof}

The point $r\equiv a^{\bot}\left(  aL^{\bot}\right)  $ is the \textbf{parallel
point }of\textbf{\ }$\overline{aL},$ or the \textbf{parallel point to }%
$a$\textbf{\ on }$L$. The parallel line $R$ and the parallel point $r$ of
$\overline{aL}$ are dual.

\begin{theorem}
[Base point]For any non-dual couple $\overline{aL}$ there is a unique point
$b$ which lies on both $L$ and the altitude line $N$ of $\overline{aL}$,
namely%
\[
b\equiv\left(  aL^{\bot}\right)  L.
\]

\end{theorem}

\begin{proof}
The point $b\equiv\left(  aL^{\bot}\right)  L$ is well-defined since $L$ does
not pass through $a.$ Furthermore $b$ lies on both $L$ and $N\equiv aL^{\perp
}$, and any such point must be $\left(  aL^{\bot}\right)  L$.$%
%TCIMACRO{\TeXButton{Proof box}{{\hspace{.1in} \rule{0.5em}{0.5em}}}}%
%BeginExpansion
{\hspace{.1in} \rule{0.5em}{0.5em}}%
%EndExpansion
$
\end{proof}

The point $b\equiv\left(  aL^{\bot}\right)  L$ is the \textbf{base point
}of\textbf{\ }$\overline{aL},$ or the \textbf{base point of the altitude line
to }$L$\textbf{\ through }$a.$

\begin{theorem}
[Base line]For any non-dual couple $\overline{aL}$ there is a unique line $B$
which passes through both $L^{\bot}$ and the altitude point $n$ of
$\overline{aL}$, namely
\[
B\equiv\left(  a^{\bot}L\right)  L^{\bot}.
\]

\end{theorem}

\begin{proof}
This is dual to the Base point theorem. $%
%TCIMACRO{\TeXButton{Proof box}{{\hspace{.1in} \rule{0.5em}{0.5em}}}}%
%BeginExpansion
{\hspace{.1in} \rule{0.5em}{0.5em}}%
%EndExpansion
$
\end{proof}

The line $B\equiv\left(  a^{\bot}L\right)  L^{\bot}$ is the \textbf{base line
}of\textbf{\ }$\overline{aL},$ or the \textbf{base line of the altitude point
to }$a$\textbf{\ on }$L.$ The base line $B$ and the base point $b$ of
$\overline{aL}$ are dual.

Can one continue in this fashion to create more and more points and lines
associated to a couple $\overline{aL}?$ If one restricts to using only the
notions introduced so far---essentially perpendicularity---the answer is
generally \textit{no---the altitudes, parallels and bases} of a couple exhaust
such constructions. It is a good exercise to confirm this.

\subsection{Parametrizing points and lines}

\begin{theorem}
[Parametrizing a line]If a point $a$ lies on a line $L\equiv\left(
l:m:n\right)  $, then there are numbers $p,r$ and $s$ such that%
\[
a=\left[  np-ms:ls+nr:lp+mr\right]  .
\]

\end{theorem}

\begin{proof}
Suppose first that at least two of $l,m$ and $n$ are non-zero. The points
$\left[  n:0:l\right]  $, $\left[  0:n:m\right]  $ and $\left[  -m:l:0\right]
$ are then distinct, and all lie on $L.$ So the equation
\begin{equation}
lx+my-nz=0\label{ParaEqn}%
\end{equation}
has non-zero solutions $\left(  n,0,l\right)  $, $\left(  0,n,m\right)  $ and
$\left(  -m,l,0\right)  $. Since
\[%
\begin{pmatrix}
n & 0 & l\\
0 & n & m\\
-m & l & 0
\end{pmatrix}
\]
is a rank two matrix, every non-zero solution to (\ref{ParaEqn}) has the form%
\begin{align*}
\left(  x,y,z\right)   & =p\left(  n,0,l\right)  +r\left(  0,n,m\right)
+s\left(  -m,l,0\right) \\
& =\left(  np-ms,ls+nr,lp+mr\right)
\end{align*}
for some numbers $p,r$ and $s$. So if $a\equiv\left[  x:y:z\right]  $ lies on
$L,$ it can be written as $a=\left[  np-ms:ls+nr:lp+mr\right]  $.

If only one of $l,m$ and $n$ are non-zero, say $n\neq0,$ then every point $a $
lying on $L$ has the form $\left[  x:y:0\right]  $, which can be written as
$\left[  np:nr:0\right]  $, and this is also of the required form. $%
%TCIMACRO{\TeXButton{Proof box}{{\hspace{.1in} \rule{0.5em}{0.5em}}}}%
%BeginExpansion
{\hspace{.1in} \rule{0.5em}{0.5em}}%
%EndExpansion
$
\end{proof}

\begin{theorem}
[Parametrizing a point]If a line $L$ passes through a point $a\equiv\left[
x:y:z\right]  $, then there are numbers $p,r$ and $s$ such that%
\[
L=\left(  zp-ys:xs+zr:xp+yr\right)  .
\]

\end{theorem}

\begin{proof}
This is dual to the Parametrizing a line theorem.$%
%TCIMACRO{\TeXButton{Proof box}{{\hspace{.1in} \rule{0.5em}{0.5em}}}}%
%BeginExpansion
{\hspace{.1in} \rule{0.5em}{0.5em}}%
%EndExpansion
$
\end{proof}

If we know that a line is a join of two particular points, then it is often
more convenient to parametrize it in a different fashion. We frame this is in
the language of vectors.

\begin{theorem}
[Parametrizing a join]Suppose that $v_{1}$ and $v_{2}$ are linearly
independent vectors, with $a_{1}\equiv\left[  v_{1}\right]  $ and $a_{2}%
\equiv\left[  v_{2}\right]  $ the corresponding points. Then for any point $a$
lying on $a_{1}a_{2}$, there is a unique proportion $t:u$ such that
\[
a=\left[  tv_{1}+uv_{2}\right]  .
\]

\end{theorem}

\begin{proof}
Suppose that $v_{1}\equiv\left(  x_{1},y_{1},z_{1}\right)  $ and $v_{2}%
\equiv\left(  x_{2},y_{2},z_{2}\right)  $. Since $v_{1}$ and $v_{2}$ are
linearly independent, $a_{1}\equiv\left[  v_{1}\right]  $ and $a_{2}%
\equiv\left[  v_{2}\right]  $ are distinct points. By the Join of points
theorem,
\[
a_{1}a_{2}=\left(  y_{1}z_{2}-y_{2}z_{1}:z_{1}x_{2}-z_{2}x_{1}:x_{2}%
y_{1}-x_{1}y_{2}\right)
\]
so any point $a\equiv\left[  x:y:z\right]  $ lying on $a_{1}a_{2}$ satisfies%
\[
\left(  y_{1}z_{2}-y_{2}z_{1}\right)  x+\left(  z_{1}x_{2}-z_{2}x_{1}\right)
y-\left(  x_{2}y_{1}-x_{1}y_{2}\right)  z=0.
\]
As an equation for $v\equiv\left(  x,y,z\right)  $, this has independent
solutions $v_{1}$ and $v_{2},$ so any non-zero solution is a non-zero linear
combination $v=tv_{1}+uv_{2}$. Then
\[
a=\left[  v\right]  =\left[  tv_{1}+uv_{2}\right]
\]
where the proportion $t:u$ is unique.$%
%TCIMACRO{\TeXButton{Proof box}{{\hspace{.1in} \rule{0.5em}{0.5em}}}}%
%BeginExpansion
{\hspace{.1in} \rule{0.5em}{0.5em}}%
%EndExpansion
$
\end{proof}

\begin{theorem}
[Parametrizing a meet]Suppose that $v_{1}$ and $v_{2}$ are linearly
independent vectors, with $L_{1}\equiv\left(  v_{1}\right)  $ and $L_{2}%
\equiv\left(  v_{2}\right)  $ the corresponding lines. Then for any line $L$
passing through $L_{1}L_{2}$, there is a unique proportion $t:u$ such that
\[
L=\left(  tv_{1}+uv_{2}\right)  .
\]

\end{theorem}

\begin{proof}
This is dual to the Parametrizing a join theorem.$%
%TCIMACRO{\TeXButton{Proof box}{{\hspace{.1in} \rule{0.5em}{0.5em}}}}%
%BeginExpansion
{\hspace{.1in} \rule{0.5em}{0.5em}}%
%EndExpansion
$
\end{proof}

\subsection{Null points and null lines}

Null points and null lines are of central importance in universal hyperbolic
geometry. We first parametrize them, using Pythagorean triples.

\begin{theorem}
[Parametrization of null points]Any null point $\alpha$ is of the form
\[
\alpha=\alpha\left(  t:u\right)  \equiv\left[  t^{2}-u^{2}:2tu:t^{2}%
+u^{2}\right]
\]
for some unique proportion $t:u$.
\end{theorem}

\begin{proof}
The identity
\[
\left(  t^{2}-u^{2}\right)  ^{2}+\left(  2tu\right)  ^{2}-\left(  t^{2}%
+u^{2}\right)  ^{2}=0
\]
ensures that every point of the form $\alpha\left(  t:u\right)  $ is null.

Suppose that $\alpha\equiv\left[  x:y:z\right]  $ is a null point, so that
$x^{2}+y^{2}-z^{2}=0\allowbreak.$ If $x+z\neq0$, then set
\[%
\begin{tabular}
[c]{lllll}%
$t\equiv x+z$ &  & and &  & $u\equiv y$%
\end{tabular}
\]
so that
\begin{align*}
\left[  t^{2}-u^{2}:2tu:t^{2}+u^{2}\right]   & =\left[  x^{2}+2xz+z^{2}%
-y^{2}:2xy+2yz:x^{2}+2xz+z^{2}+y^{2}\right]  \allowbreak\\
& =\left[  2xz+2x^{2}:2xy+2yz:2xz+2z^{2}\right] \\
& =\left[  2x\left(  x+z\right)  :2y\left(  x+z\right)  :2z\left(  x+z\right)
\right] \\
& =\left[  x:y:z\right]  .
\end{align*}
This shows that $\alpha=\alpha\left(  t:u\right)  $. Note that our assumption
on the field not having characteristic two is used here. If$\allowbreak$
$x+z=0$, then necessarily $y=0,$ so that $\left[  x:y:z\right]  =\left[
-1:0:1\right]  $. In this case $\alpha=\alpha\left(  0:1\right)  $.

To show uniqueness, suppose that we have two proportions $t:u$ and $r:s$ with%
\[
\left[  t^{2}-u^{2}:2tu:t^{2}+u^{2}\right]  =\left[  r^{2}-s^{2}%
:2rs:r^{2}+s^{2}\right]  .
\]
Then using the condition for equality of $3$-proportions in
(\ref{ProportionEqu}),
\begin{align*}
\left(  t^{2}-u^{2}\right)  \left(  2rs\right)  -\left(  r^{2}-s^{2}\right)
\left(  2tu\right)   & =2\left(  ts-ru\right)  \left(  tr+us\right)  =0\\
\left(  t^{2}-u^{2}\right)  \left(  r^{2}+s^{2}\right)  -\left(  r^{2}%
-s^{2}\right)  \left(  t^{2}+u^{2}\right)   & =\allowbreak2\left(
ts-ru\right)  \left(  ru+ts\right)  =0\\
\left(  2tu\right)  \left(  r^{2}+s^{2}\right)  -\left(  2rs\right)  \left(
t^{2}+u^{2}\right)   & =2\left(  ts-ru\right)  \left(  us-tr\right)  =0.
\end{align*}
If $ts-ru\neq0$ then it follows that%
\begin{align*}
tr+us  & =0\\
ru+st  & =0\\
us-tr  & =0.
\end{align*}
Thus $tr=us=0$ from the first and third equations, so at least one of $r,t$
must be zero and at least one of $s,u$ must be zero, but since $r:s$ and $t:u
$ are both proportions, either $r=u=0$ or $t=s=0,$ but not both. But this
contradicts the second equation, so we conclude that $ts-ru=0,$ and the two
proportions $t:u$ and $r:s$ are indeed equal.$%
%TCIMACRO{\TeXButton{Proof box}{{\hspace{.1in} \rule{0.5em}{0.5em}}}}%
%BeginExpansion
{\hspace{.1in} \rule{0.5em}{0.5em}}%
%EndExpansion
$
\end{proof}

\begin{theorem}
[Parametrization of null lines]Any null line $\Lambda$ is of the form
\[
\Lambda=\Lambda\left(  t:u\right)  \equiv\left(  t^{2}-u^{2}:2tu:t^{2}%
+u^{2}\right)
\]
for some unique proportion $t:u$.
\end{theorem}

\begin{proof}
This is dual to the Parametrization of null points theorem.$%
%TCIMACRO{\TeXButton{Proof box}{{\hspace{.1in} \rule{0.5em}{0.5em}}}}%
%BeginExpansion
{\hspace{.1in} \rule{0.5em}{0.5em}}%
%EndExpansion
$
\end{proof}

It is worth pointing out that the three quadratic forms $t^{2}-u^{2},2tu$ and
$t^{2}+u^{2}$ play a key role in the new theory of\textit{\ chromogeometry}
(\cite{Wild3}). The corresponding bilinear forms appear also in the next result.

\begin{theorem}
[Join of null points]The join of the distinct null points $\alpha_{1}%
\equiv\alpha\left(  t_{1}:u_{1}\right)  $ and $\alpha_{2}\equiv\alpha\left(
t_{2}:u_{2}\right)  $ is the non-null line
\[
\alpha_{1}\alpha_{2}=L\left(  t_{1}:u_{1}|t_{2}:u_{2}\right)  \equiv\left(
t_{1}t_{2}-u_{1}u_{2}:t_{1}u_{2}+t_{2}u_{1}:t_{1}t_{2}+u_{1}u_{2}\right)  .
\]

\end{theorem}

\begin{proof}
Since
\begin{align*}
\alpha_{1}  & =\left[  t_{1}^{2}-u_{1}^{2}:2t_{1}u_{1}:t_{1}^{2}+u_{1}%
^{2}\right] \\
\alpha_{2}  & =\left[  t_{2}^{2}-u_{2}^{2}:2t_{2}u_{2}:t_{2}^{2}+u_{2}%
^{2}\right]
\end{align*}
the Join of points theorem shows that
\begin{align*}
\alpha_{1}\alpha_{2}  & =\left(
\begin{array}
[c]{c}%
2t_{1}u_{1}\left(  t_{2}^{2}+u_{2}^{2}\right)  -2t_{2}u_{2}\left(  t_{1}%
^{2}+u_{1}^{2}\right)  :\left(  t_{1}^{2}+u_{1}^{2}\right)  \left(  t_{2}%
^{2}-u_{2}^{2}\right)  -\left(  t_{2}^{2}+u_{2}^{2}\right)  \left(  t_{1}%
^{2}-u_{1}^{2}\right) \\
:2t_{1}u_{1}\left(  t_{2}^{2}-u_{2}^{2}\right)  -2t_{2}u_{2}\left(  t_{1}%
^{2}-u_{1}^{2}\right)
\end{array}
\right) \\
& =\left(  2\left(  t_{2}u_{1}-t_{1}u_{2}\right)  \left(  t_{1}t_{2}%
-u_{1}u_{2}\right)  :2\left(  t_{2}u_{1}-t_{1}u_{2}\right)  \left(  t_{1}%
u_{2}+t_{2}u_{1}\right)  :2\left(  t_{2}u_{1}-t_{1}u_{2}\right)  \left(
t_{1}t_{2}+u_{1}u_{2}\right)  \right) \\
& =\left(  t_{1}t_{2}-u_{1}u_{2}:t_{1}u_{2}+t_{2}u_{1}:t_{1}t_{2}+u_{1}%
u_{2}\right)  .
\end{align*}
We have divided by the factor $2\left(  t_{2}u_{1}-t_{1}u_{2}\right)  ,$ which
is non-zero since the field does not have characteristic two, and since
$t_{1}:u_{1}$ and $t_{2}:u_{2}$ are distinct proportions. The identity
\begin{equation}
\left(  t_{1}t_{2}-u_{1}u_{2}\right)  ^{2}+\left(  t_{1}u_{2}+t_{2}%
u_{1}\right)  ^{2}-\left(  t_{1}t_{2}+u_{1}u_{2}\right)  ^{2}=\allowbreak
\left(  t_{1}u_{2}-t_{2}u_{1}\right)  ^{2}.\label{Null line product}%
\end{equation}
shows that $L\left(  t_{1}:u_{1}|t_{2}:u_{2}\right)  $ is a non-null line.$%
%TCIMACRO{\TeXButton{Proof box}{{\hspace{.1in} \rule{0.5em}{0.5em}}}}%
%BeginExpansion
{\hspace{.1in} \rule{0.5em}{0.5em}}%
%EndExpansion
$
\end{proof}

\begin{theorem}
[Meet of null lines]The meet of the distinct null lines $\Lambda_{1}%
\equiv\Lambda\left(  t_{1}:u_{1}\right)  $ and $\Lambda_{2}\equiv
\Lambda\left(  t_{2}:u_{2}\right)  $ is the non-null point%
\[
\Lambda_{1}\Lambda_{2}=a\left(  t_{1}:u_{1}|t_{2}:u_{2}\right)  \equiv\left[
t_{1}t_{2}-u_{1}u_{2}:t_{1}u_{2}+t_{2}u_{1}:t_{1}t_{2}+u_{1}u_{2}\right]  .
\]

\end{theorem}

\begin{proof}
This is dual to the Join of null points theorem.$%
%TCIMACRO{\TeXButton{Proof box}{{\hspace{.1in} \rule{0.5em}{0.5em}}}}%
%BeginExpansion
{\hspace{.1in} \rule{0.5em}{0.5em}}%
%EndExpansion
$
\end{proof}

\begin{theorem}
[Null diagonal point]The meet of the disjoint lines $L\left(  t_{1}%
:u_{1}|t_{2}:u_{2}\right)  $ and $L\left(  t_{3}:u_{3}|t_{4}:u_{4}\right)  $
is the point
\[
a\left(  t_{1}:u_{1}|t_{2}:u_{2}||t_{3}:u_{3}|t_{4}:u_{4}\right)
\equiv\left[  x:y:z\right]
\]
where
\begin{align*}
x  & \equiv\left(  t_{1}u_{2}+t_{2}u_{1}\right)  \left(  t_{3}t_{4}+u_{3}%
u_{4}\right)  -\left(  t_{3}u_{4}+t_{4}u_{3}\right)  \left(  t_{1}t_{2}%
+u_{1}u_{2}\right) \\
y  & \equiv\left(  t_{1}t_{2}+u_{1}u_{2}\right)  \left(  t_{3}t_{4}-u_{3}%
u_{4}\right)  -\left(  t_{3}t_{4}+u_{3}u_{4}\right)  \left(  t_{1}t_{2}%
-u_{1}u_{2}\right) \\
z  & \equiv\left(  t_{1}u_{2}+t_{2}u_{1}\right)  \left(  t_{3}t_{4}-u_{3}%
u_{4}\right)  -\left(  t_{3}u_{4}+t_{4}u_{3}\right)  \left(  t_{1}t_{2}%
-u_{1}u_{2}\right)  .
\end{align*}
$\allowbreak$
\end{theorem}

\begin{proof}
Finding the coordinates $x,y$ and $z$ is essentially the computation of%
\[
J\left(  t_{1}t_{2}-u_{1}u_{2},t_{1}u_{2}+t_{2}u_{1},t_{1}t_{2}+u_{1}%
u_{2};t_{3}t_{4}-u_{3}u_{4},t_{3}u_{4}+t_{4}u_{3},t_{3}t_{4}+u_{3}%
u_{4}\right)  .%
%TCIMACRO{\TeXButton{Proof box}{{\hspace{.1in} \rule{0.5em}{0.5em}}}}%
%BeginExpansion
{\hspace{.1in} \rule{0.5em}{0.5em}}%
%EndExpansion
\]

\end{proof}

Note the pleasant identity
\begin{align*}
& \left(  \left(  t_{1}u_{2}+t_{2}u_{1}\right)  \left(  t_{3}t_{4}+u_{3}%
u_{4}\right)  -\left(  t_{3}u_{4}+t_{4}u_{3}\right)  \left(  t_{1}t_{2}%
+u_{1}u_{2}\right)  \right)  ^{2}\\
& +\left(  \left(  t_{1}t_{2}+u_{1}u_{2}\right)  \left(  t_{3}t_{4}-u_{3}%
u_{4}\right)  -\left(  t_{3}t_{4}+u_{3}u_{4}\right)  \left(  t_{1}t_{2}%
-u_{1}u_{2}\right)  \right)  ^{2}\\
& -\left(  \left(  t_{1}u_{2}+t_{2}u_{1}\right)  \left(  t_{3}t_{4}-u_{3}%
u_{4}\right)  -\left(  t_{3}u_{4}+t_{4}u_{3}\right)  \left(  t_{1}t_{2}%
-u_{1}u_{2}\right)  \right)  ^{2}\\
& =\allowbreak4\left(  t_{3}u_{2}-t_{2}u_{3}\right)  \left(  u_{1}t_{3}%
-t_{1}u_{3}\right)  \left(  u_{1}t_{4}-t_{1}u_{4}\right)  \allowbreak\left(
u_{2}t_{4}-t_{2}u_{4}\right)  .
\end{align*}

\begin{theorem}
[Null diagonal line]The join of the disjoint points $a\left(  t_{1}%
:u_{1}|t_{2}:u_{2}\right)  $ and $a\left(  t_{3}:u_{3}|t_{4}:u_{4}\right)  $
is the line%
\[
L\left(  t_{1}:u_{1}|t_{2}:u_{2}||t_{3}:u_{3}|t_{4}:u_{4}\right)
\equiv\left(  x:y:z\right)
\]
where $x,y$ and $z$ are as in the Null diagonal point theorem.
\end{theorem}

\begin{proof}
This is dual to the Null diagonal point theorem.$%
%TCIMACRO{\TeXButton{Proof box}{{\hspace{.1in} \rule{0.5em}{0.5em}}}}%
%BeginExpansion
{\hspace{.1in} \rule{0.5em}{0.5em}}%
%EndExpansion
$
\end{proof}

\begin{theorem}
[Perpendicular null line]If $a_{1}\equiv\left[  x_{1}:y_{1}:z_{1}\right]  $
and $a_{2}\equiv\left[  x_{2}:y_{2}:z_{2}\right]  $ are distinct points with
$a_{1}$ null, then $a_{1}a_{2}$ is a null line precisely when $a_{1}$ and
$a_{2}$ are perpendicular, in which case $a_{1}a_{2}=a_{1}^{\bot}.$
\end{theorem}

\begin{proof}
From the Join of points theorem
\[
a_{1}a_{2}\equiv\left(  y_{1}z_{2}-y_{2}z_{1}:z_{1}x_{2}-z_{2}x_{1}:x_{2}%
y_{1}-x_{1}y_{2}\right)  .
\]
The identity%
\begin{align*}
& \left(  y_{1}z_{2}-y_{2}z_{1}\right)  ^{2}+\left(  z_{1}x_{2}-z_{2}%
x_{1}\right)  ^{2}-\left(  x_{2}y_{1}-x_{1}y_{2}\right)  ^{2}\\
& =\left(  x_{1}x_{2}+y_{1}y_{2}-z_{1}z_{2}\right)  ^{2}-\allowbreak\left(
x_{1}^{2}+y_{1}^{2}-z_{1}^{2}\right)  \left(  x_{2}^{2}+y_{2}^{2}-z_{2}%
^{2}\right)
\end{align*}
shows that if $a_{1}$ is a null point, meaning that $x_{1}^{2}+y_{1}^{2}%
-z_{1}^{2}=0,$ then $a_{1}a_{2}$ is a null line precisely when $a_{1}$ and
$a_{2}$ are perpendicular. In this case both $a_{1}$ and $a_{2}$ are
perpendicular to $a_{1},$ so that $a_{1}a_{2}=a_{1}^{\bot}$.$%
%TCIMACRO{\TeXButton{Proof box}{{\hspace{.1in} \rule{0.5em}{0.5em}}}}%
%BeginExpansion
{\hspace{.1in} \rule{0.5em}{0.5em}}%
%EndExpansion
$
\end{proof}

\begin{theorem}
[Perpendicular null point]If $L_{1}\equiv\left(  l_{1}:m_{1}:n_{1}\right)  $
and $L_{2}\equiv\left(  l_{2}:m_{2}:n_{2}\right)  $ are distinct lines with
$L_{1}$ null, then $L_{1}L_{2}$ is a null point precisely when $L_{1}$ and
$L_{2}$ are perpendicular, in which case $L_{1}L_{2}=L_{1}^{\bot}.$
\end{theorem}

\begin{proof}
This is dual to the Perpendicular null line theorem.$%
%TCIMACRO{\TeXButton{Proof box}{{\hspace{.1in} \rule{0.5em}{0.5em}}}}%
%BeginExpansion
{\hspace{.1in} \rule{0.5em}{0.5em}}%
%EndExpansion
$
\end{proof}

\begin{theorem}
[Parametrizing a null line]Suppose that $\Lambda=\Lambda\left(  t:u\right)  $
is a null line. If $t^{2}+u^{2}\neq0,$ then any point $a$ lying on $\Lambda$
is of the form%
\[
\left[  r\left(  t^{2}-u^{2}\right)  -2stu:2rtu+s\left(  t^{2}-u^{2}\right)
:r\left(  t^{2}+u^{2}\right)  \right]
\]
for a unique proportion $r:s,$ while if $t^{2}+u^{2}=0,$ then every point $a$
lying on $\Lambda$ is of the form
\[
\left[  rt:ru:s\right]
\]
for a unique proportion $r:s.$
\end{theorem}

\begin{proof}
Two points lying on $\Lambda\left(  t:u\right)  \equiv\left(  t^{2}%
-u^{2}:2tu:t^{2}+u^{2}\right)  $ are $\left[  t^{2}-u^{2}:2tu:t^{2}%
+u^{2}\right]  $ and $\left[  -2tu:t^{2}-u^{2}:0\right]  $. If $t^{2}%
+u^{2}\neq0,$ these are distinct, and the Parametrization of a join theorem
shows that for any point $a$ lying on $\Lambda$ there is a unique proportion
$r:s$ such that
\[
a=\left[  r\left(  t^{2}-u^{2}\right)  -2stu:2rtu+s\left(  t^{2}-u^{2}\right)
:r\left(  t^{2}+u^{2}\right)  \right]  .
\]
On the other hand if $t^{2}+u^{2}=0$ then $\left[  t:u:0\right]  $ and
$\left[  0:0:1\right]  $ are two distinct points lying on $\Lambda,$ so again
by the Parametrizing a join theorem, for any point $a$ lying on $\Lambda$
there is a unique proportion $r:s$ such that
\[
a=\left[  rt:ru:s\right]  .%
%TCIMACRO{\TeXButton{Proof box}{{\hspace{.1in} \rule{0.5em}{0.5em}}}}%
%BeginExpansion
{\hspace{.1in} \rule{0.5em}{0.5em}}%
%EndExpansion
\]

\end{proof}

\begin{theorem}
[Parametrizing a null point]Suppose that $\alpha=\alpha\left(  t:u\right)  $
is a null point. If $t^{2}+u^{2}\neq0,$ then any line $L$ passing through
$\alpha$ is of the form%
\[
\left(  r\left(  t^{2}-u^{2}\right)  -2stu:2rtu+s\left(  t^{2}-u^{2}\right)
:r\left(  t^{2}+u^{2}\right)  \right)
\]
for a unique proportion $r:s,$ while if $t^{2}+u^{2}=0,$ then every line $L$
passing through $\alpha$ is of the form
\[
\left(  rt:ru:s\right)
\]
for a unique proportion $r:s.$
\end{theorem}

\begin{proof}
This is dual to the Parametrizing a null line theorem.$%
%TCIMACRO{\TeXButton{Proof box}{{\hspace{.1in} \rule{0.5em}{0.5em}}}}%
%BeginExpansion
{\hspace{.1in} \rule{0.5em}{0.5em}}%
%EndExpansion
$
\end{proof}

\subsection{Triangles and trilaterals}

A \textbf{triangle }$\overline{a_{1}a_{2}a_{3}}$ is a set $\left\{
a_{1},a_{2},a_{3}\right\}  $ of three non-collinear points. A
\textbf{trilateral }$\overline{L_{1}L_{2}L_{3}}$ is a set $\left\{
L_{1},L_{2},L_{3}\right\}  $ of three non-concurrent lines.

The triangle $\triangle\equiv\overline{a_{1}a_{2}a_{3}}$ has an
\textbf{associated trilateral }$\widetilde{\triangle}\equiv\overline
{L_{1}L_{2}L_{3}}$ consisting of the three \textbf{lines }of the triangle,
namely
\[%
\begin{tabular}
[c]{lllll}%
$L_{1}\equiv a_{2}a_{3}$ &  & $L_{2}\equiv a_{1}a_{3}$ &  & $L_{3}\equiv
a_{1}a_{2}.$%
\end{tabular}
\]
It also has a \textbf{dual trilateral} $\triangle^{\perp}\equiv\overline
{a_{1}^{\bot}a_{2}^{\bot}a_{3}^{\bot}}$ consisting of the three \textbf{dual
lines }of the triangle, namely\textbf{\ }$a_{1}^{\bot}$, $a_{2}^{\bot} $ and
$a_{3}^{\bot}$.

A trilateral $\triangledown\equiv\overline{L_{1}L_{2}L_{3}}$ has an
\textbf{associated triangle }$\widetilde{\triangledown}\equiv\overline
{a_{1}a_{2}a_{3}}$ consisting of the three \textbf{points }of the trilateral,
namely
\[%
\begin{tabular}
[c]{lllll}%
$a_{1}\equiv L_{2}L_{3}$ &  & $a_{2}\equiv L_{1}L_{3}$ &  & $a_{3}\equiv
L_{1}L_{2}.$%
\end{tabular}
\]
It also has a \textbf{dual triangle} $\triangledown^{\bot}\equiv
\overline{L_{1}^{\bot}L_{2}^{\bot}L_{3}^{\bot}}$ consisting of the three
\textbf{dual points }of the trilateral, namely $L_{1}^{\bot}$,\textbf{\ }%
$L_{2}^{\bot}$\textbf{\ }and $L_{3}^{\bot}$.

A triangle $\triangle\equiv\overline{a_{1}a_{2}a_{3}}$ and its associated
trilateral $\widetilde{\triangle}\equiv\overline{L_{1}L_{2}L_{3}}$ both have
\textbf{sides} $\overline{a_{1}a_{2}},$ $\overline{a_{2}a_{3}}$ and
$\overline{a_{1}a_{3}},$ \textbf{vertices} $\overline{L_{1}L_{2}},$
$\overline{L_{2}L_{3}}$ and $\overline{L_{1}L_{3}}$, and \textbf{couples}
$\overline{a_{1}L_{1}},$ $\overline{a_{2}L_{2}}$ and $\overline{a_{3}L_{3}}$.

\begin{theorem}
[Triangle trilateral duality]If the triangle $\triangle\equiv\overline
{a_{1}a_{2}a_{3}}$ is dual to the trilateral $\triangledown\equiv
\overline{A_{1}A_{2}A_{3}},$ then the \textit{points, lines, sides, vertices
}and\textit{\ couples }of $\overline{a_{1}a_{2}a_{3}}$ are dual respectively
to the \textit{lines, points, vertices, sides }and \textit{couples }of
$\overline{A_{1}A_{2}A_{3}}$.
\end{theorem}

\begin{proof}
These statements all follow directly from the definitions.$%
%TCIMACRO{\TeXButton{Proof box}{{\hspace{.1in} \rule{0.5em}{0.5em}}}}%
%BeginExpansion
{\hspace{.1in} \rule{0.5em}{0.5em}}%
%EndExpansion
$
\end{proof}

A triangle $\triangle\equiv\overline{a_{1}a_{2}a_{3}}$ is \textbf{null}
precisely when one or more of its lines $a_{1}a_{2},a_{2}a_{3}$ or $a_{1}%
a_{3}$ is null. More specifically, $\triangle$ is \textbf{singly-null
}precisely when exactly one of its lines is null, \textbf{doubly-null}
precisely when exactly two of its lines are null, and \textbf{triply-null} or
\textbf{fully-null }precisely when all three of its lines are null.

\begin{example}
Given a proportion $t:u,$ consider points $a_{1}\equiv\left[  t^{2}%
+u^{2}:0:t^{2}-u^{2}\right]  $, $a_{2}\equiv\left[  0:t^{2}+u^{2}:2tu\right]
\allowbreak$ and $a_{3}\equiv\left[  2tu:t^{2}-u^{2}:0\right]  $. Then
$\overline{a_{1}a_{2}a_{3}}$ is a fully-null triangle.$%
%TCIMACRO{\TeXButton{Math Diamond}{\hspace{.1in}\diamond}}%
%BeginExpansion
\hspace{.1in}\diamond
%EndExpansion
$
\end{example}

A trilateral $\triangledown\equiv\overline{L_{1}L_{2}L_{3}}$ is \textbf{null
}precisely when one or more of its points $L_{1}L_{2},L_{2}L_{3}$ or
$L_{1}L_{3}$ is null. More specifically, $\triangledown$ is
\textbf{singly-null }precisely when exactly one of its points is null,
\textbf{doubly-null} precisely when exactly two of its points are null, and
\textbf{triply-null} or \textbf{fully-null }precisely when all three of its
points are null.

A triangle $\triangle\equiv\overline{a_{1}a_{2}a_{3}}$ is \textbf{nil
}precisely when one or more of its points $a_{1},a_{2}$ or $a_{3}$ is null.
More specifically, $\triangle$ is \textbf{singly-nil }precisely when exactly
one of its points is null, \textbf{doubly-nil} precisely when exactly two of
its points are null, and \textbf{triply-nil} or \textbf{fully-nil }precisely
when all three of its points are null.

A trilateral $\triangledown\equiv\overline{L_{1}L_{2}L_{3}}$ is \textbf{nil
}precisely when one or more of its lines $L_{1},L_{2}$ or $L_{3}$ is null.
More specifically, $\triangledown$ is \textbf{singly-nil }precisely when
exactly one of its lines is null, \textbf{doubly-nil} precisely when exactly
two of its lines are null, and \textbf{triply-nil} or \textbf{fully-nil
}precisely when all three of its lines are null.

A triangle $\triangle\equiv\overline{a_{1}a_{2}a_{3}}$ is \textbf{right}
precisely when one or more of its vertices is right. More specifically,
$\triangle$ is \textbf{singly-right }precisely when it has exactly one right
vertex, \textbf{doubly-right} precisely when it has exactly two right
vertices, and \textbf{triply-right} or \textbf{fully-right }precisely when it
has three right vertices.

A trilateral $\triangledown\equiv\overline{L_{1}L_{2}L_{3}}$ is \textbf{right}
precisely when one or more of its sides is right. More specifically,
$\triangledown$ is \textbf{singly-right }precisely when it has exactly one
right side, \textbf{doubly-right} precisely when it has exactly two right
sides, and \textbf{triply-right} or \textbf{fully-right }precisely when it has
three right sides.

\section{Hyperbolic trigonometry}

In this section, which constitutes the heart of the paper, we introduce the
basic measurements of universal hyperbolic geometry: the \textit{quadrance}
$q\left(  a_{1},a_{2}\right)  $ between points $a_{1}$ and $a_{2},$ and the
\textit{spread }$S\left(  L_{1},L_{2}\right)  $ between lines $L_{1}$ and
$L_{2}$. These are \textit{dual notions} in universal hyperbolic geometry.
Recall that we work over a general field not of characteristic two: the values
of the quadrance and spread will be \textit{numbers in this field}.

After briefly describing how quadrance and spread may be interpreted using
cross ratios and perpendicularity, we introduce the main trigonometric laws:
the \textit{Triple quad formula}, the \textit{Triple spread formula},
\textit{Pythagoras' theorem}, the \textit{Spread law }and the \textit{Cross
law}, along with their duals. Along the way we also define the \textit{quadrea
}$\mathcal{A}$ of a triangle, and the \textit{quadreal} $\mathcal{L}$ of a
trilateral, and the secondary concepts of \textit{product} and \textit{cross
}between points and lines respectively. Then we derive some variants of the
main laws that assume the existence of \textit{midpoints} and
\textit{midlines}.

The notions of quadrance and spread in this hyperbolic setting are closely
related to the corresponding notions in planar Euclidean geometry, and the
main laws can be seen as deformations of the Euclidean ones. In the special
case of classical hyperbolic geometry in the Beltrami Klein model over the
real numbers, the quadrance $q\left(  a_{1},a_{2}\right)  $ between points
$a_{1}$ and $a_{2}$ is related to the usual hyperbolic distance $d\left(
a_{1},a_{2}\right)  $ between them by the relation%
\begin{equation}
q\left(  a_{1},a_{2}\right)  =-\sinh^{2}\left(  d\left(  a_{1},a_{2}\right)
\right)  .\label{QuadranceDistance}%
\end{equation}
The spread $S\left(  L_{1},L_{2}\right)  $ between lines $L_{1}$ and $L_{2}$
is related to the usual hyperbolic angle $\theta\left(  L_{1},L_{2}\right)  $
between them by the relation%
\begin{equation}
S\left(  L_{1},L_{2}\right)  =\sin^{2}\left(  \theta\left(  L_{1}%
,L_{2}\right)  \right)  .\label{SpreadAngle}%
\end{equation}
However our derivation of hyperbolic trigonometry is \textit{completely
independent of the classical treatment}, much more elementary, and extends far
beyond it. So we can take any formula from universal hyperbolic geometry, make
the substitutions (\ref{QuadranceDistance}) and (\ref{SpreadAngle}), restrict
to the interior of the null cone or unit disk, and obtain a valid formula for
classical hyperbolic geometry over the real numbers.

\subsection{Quadrance between points}

The \textbf{quadrance} between points $a_{1}\equiv\left[  x_{1}:y_{1}%
:z_{1}\right]  $ and $a_{2}\equiv\left[  x_{2}:y_{2}:z_{2}\right]  $ is the
number%
\begin{equation}
q\left(  a_{1},a_{2}\right)  \equiv1-\frac{\left(  x_{1}x_{2}+y_{1}y_{2}%
-z_{1}z_{2}\right)  ^{2}}{\left(  x_{1}^{2}+y_{1}^{2}-z_{1}^{2}\right)
\left(  x_{2}^{2}+y_{2}^{2}-z_{2}^{2}\right)  }.\label{QuadranceFormula}%
\end{equation}

This number is undefined if either of $a_{1},a_{2}$ are null points, and has
the value $1$ precisely when $a_{1}$ and $a_{2}$ are perpendicular. If
$a_{1}=a_{2}$ then $q\left(  a_{1},a_{2}\right)  =0,$ but we shall shortly see
that there are also other ways in which a quadrance of zero can occur.

\begin{example}
The quadrance between $a_{1}\equiv\left[  0:0:1\right]  $ and $a\equiv\left[
x:y:1\right]  $ is
\[
q\left(  a_{1},a\right)  =\allowbreak\frac{x^{2}+y^{2}}{x^{2}+y^{2}-1}.%
%TCIMACRO{\TeXButton{Math Diamond}{\hspace{.1in}\diamond}}%
%BeginExpansion
\hspace{.1in}\diamond
%EndExpansion
\]

\end{example}

\begin{example}
The quadrance between $a_{2}\equiv\left[  1:0:0\right]  $ and $a\equiv\left[
x:y:1\right]  $ is
\[
q\left(  a_{2},a\right)  =\allowbreak\frac{y^{2}-1}{x^{2}+y^{2}-1}.%
%TCIMACRO{\TeXButton{Math Diamond}{\hspace{.1in}\diamond}}%
%BeginExpansion
\hspace{.1in}\diamond
%EndExpansion
\]

\end{example}

\begin{example}
The quadrance between $c_{1}\equiv\left[  x_{1}:0:1\right]  $ and $c_{2}%
\equiv\left[  x_{2}:0:1\right]  $ is%
\[
q\left(  c_{1},c_{2}\right)  =\frac{-\left(  x_{1}-x_{2}\right)  ^{2}}{\left(
x_{2}^{2}-1\right)  \left(  x_{1}^{2}-1\right)  }.%
%TCIMACRO{\TeXButton{Math Diamond}{\hspace{.1in}\diamond}}%
%BeginExpansion
\hspace{.1in}\diamond
%EndExpansion
\]

\end{example}

\begin{theorem}
[Quadrance]For points $a_{1}\equiv\left[  x_{1}:y_{1}:z_{1}\right]  $ and
$a_{2}\equiv\left[  x_{2}:y_{2}:z_{2}\right]  $,%
\begin{equation}
q\left(  a_{1},a_{2}\right)  =-\frac{\left(  y_{1}z_{2}-y_{2}z_{1}\right)
^{2}+\left(  z_{1}x_{2}-z_{2}x_{1}\right)  ^{2}-\left(  x_{1}y_{2}-y_{1}%
x_{2}\right)  ^{2}}{\left(  x_{1}^{2}+y_{1}^{2}-z_{1}^{2}\right)  \left(
x_{2}^{2}+y_{2}^{2}-z_{2}^{2}\right)  }.\label{QuadranceForm2}%
\end{equation}

\end{theorem}

\begin{proof}
This is a consequence of the following extension of Fibonacci's identity:%
\begin{align*}
& \left(  x_{2}x_{3}+y_{2}y_{3}-z_{2}z_{3}\right)  ^{2}-\left(  y_{2}%
z_{3}-y_{3}z_{2}\right)  ^{2}-\left(  z_{2}x_{3}-z_{3}x_{2}\right)
^{2}+\left(  x_{3}y_{2}-x_{2}y_{3}\right)  ^{2}\\
& =\allowbreak\left(  x_{2}^{2}+y_{2}^{2}-z_{2}^{2}\right)  \left(  x_{3}%
^{2}+y_{3}^{2}-z_{3}^{2}\right)  .%
%TCIMACRO{\TeXButton{Proof box}{{\hspace{.1in} \rule{0.5em}{0.5em}}}}%
%BeginExpansion
{\hspace{.1in} \rule{0.5em}{0.5em}}%
%EndExpansion
\end{align*}

\end{proof}

\textit{Note carefully the minus sign} that begins the right hand side of
(\ref{QuadranceForm2}). This has the consequence that for points corresponding
to classical hyperbolic geometry, \textit{the quadrance between them is
negative, }in agreement with (\ref{QuadranceDistance}).

\begin{theorem}
[Zero quadrance]If $a_{1}$ and $a_{2}$ are distinct points, then $q\left(
a_{1},a_{2}\right)  =0$ precisely when $a_{1}a_{2}$ is a null line.
\end{theorem}

\begin{proof}
This follows from the Quadrance theorem together with the Join of points
theorem, since if $a_{1}\equiv\left[  x_{1}:y_{1}:z_{1}\right]  $ and
$a_{2}\equiv\left[  x_{2}:y_{2}:z_{2}\right]  $ then
\[
a_{1}a_{2}\equiv\left(  y_{1}z_{2}-y_{2}z_{1}:z_{1}x_{2}-z_{2}x_{1}:x_{2}%
y_{1}-x_{1}y_{2}\right)
\]
and this is a null line precisely when
\[
\left(  y_{1}z_{2}-y_{2}z_{1}\right)  ^{2}+\left(  z_{1}x_{2}-z_{2}%
x_{1}\right)  ^{2}-\left(  x_{2}y_{1}-x_{1}y_{2}\right)  ^{2}=0.%
%TCIMACRO{\TeXButton{Proof box}{{\hspace{.1in} \rule{0.5em}{0.5em}}}}%
%BeginExpansion
{\hspace{.1in} \rule{0.5em}{0.5em}}%
%EndExpansion
\]

\end{proof}

\subsection{Spread between lines}

The \textbf{spread} between lines $L_{1}\equiv\left(  l_{1}:m_{1}%
:n_{1}\right)  $ and $L_{2}\equiv\left(  l_{2}:m_{2}:n_{2}\right)  $ is the
number%
\begin{equation}
S\left(  L_{1},L_{2}\right)  \equiv1-\frac{\left(  l_{1}l_{2}+m_{1}m_{2}%
-n_{1}n_{2}\right)  ^{2}}{\left(  l_{1}^{2}+m_{1}^{2}-n_{1}^{2}\right)
\allowbreak\left(  l_{2}^{2}+m_{2}^{2}-n_{2}^{2}\right)  }%
.\label{Spread formula}%
\end{equation}

This number is undefined if either of $L_{1},L_{2}$ are null lines, and has
the value $1$ precisely when the lines $L_{1}$ and $L_{2}$ are perpendicular.
If $L_{1}=L_{2}$ then $S\left(  L_{1},L_{2}\right)  =0,$ but again there are
also other ways in which a spread of zero can occur.

\begin{theorem}
[Spread]For lines $L_{1}\equiv\left(  l_{1}:m_{1}:n_{1}\right)  $ and
$L_{2}\equiv\left(  l_{2}:m_{2}:n_{2}\right)  $,
\begin{equation}
S\left(  L_{1},L_{2}\right)  =-\frac{\left(  m_{1}n_{2}-m_{2}n_{1}\right)
^{2}+\left(  n_{1}l_{2}-n_{2}l_{1}\right)  ^{2}-\left(  l_{1}m_{2}-l_{2}%
m_{1}\right)  ^{2}}{\left(  l_{1}^{2}+m_{1}^{2}-n_{1}^{2}\right)  \left(
l_{2}^{2}+m_{2}^{2}-n_{2}^{2}\right)  \allowbreak}.\label{SpreadForm1}%
\end{equation}

\end{theorem}

\begin{proof}
This is dual to the Quadrance theorem.$%
%TCIMACRO{\TeXButton{Proof box}{{\hspace{.1in} \rule{0.5em}{0.5em}}}}%
%BeginExpansion
{\hspace{.1in} \rule{0.5em}{0.5em}}%
%EndExpansion
$
\end{proof}

\begin{example}
The spread between the non-null lines $L_{1}\equiv\left(  l_{1}:m_{1}%
:0\right)  $ and $L_{2}\equiv\left(  l_{2}:m_{2}:0\right)  $ passing through
the point $\left[  0:0:1\right]  $ is%
\[
S\left(  L_{1},L_{2}\right)  =\frac{\left(  l_{1}m_{2}-l_{2}m_{1}\right)
^{2}}{\left(  l_{1}^{2}+m_{1}^{2}\right)  \left(  l_{2}^{2}+m_{2}^{2}\right)
}.
\]
This is the same as the Euclidean planar spread (as studied in \cite{Wild1})
between the lines $\left\langle l_{1}:m_{1}:0\right\rangle $ and $\left\langle
l_{2}:m_{2}:0\right\rangle $ with respective equations $l_{1}X+m_{1}Y=0$ and
$l_{2}X+m_{1}Y=0$.$%
%TCIMACRO{\TeXButton{Math Diamond}{\hspace{.1in}\diamond}}%
%BeginExpansion
\hspace{.1in}\diamond
%EndExpansion
$
\end{example}

\begin{example}
The spread between the non-null lines $M_{1}\equiv\left(  l_{1}:0:1\right)  $
and $M_{2}\equiv\left(  l_{2}:0:1\right)  $ passing through the point $\left[
0:1:0\right]  $ is
\[
S\left(  M_{1},M_{2}\right)  =\allowbreak-\frac{\left(  l_{2}-l_{1}\right)
^{2}}{\left(  l_{1}^{2}-1\right)  \left(  l_{2}^{2}-1\right)  }.%
%TCIMACRO{\TeXButton{Math Diamond}{\hspace{.1in}\diamond}}%
%BeginExpansion
\hspace{.1in}\diamond
%EndExpansion
\]

\end{example}

\begin{theorem}
[Zero spread]If $L_{1}$ and $L_{2}$ are distinct non-null lines, then
$S\left(  L_{1},L_{2}\right)  =0$ precisely when $L_{1}L_{2}$ is a null point.
\end{theorem}

\begin{proof}
This is dual to the Zero quadrance theorem.$%
%TCIMACRO{\TeXButton{Proof box}{{\hspace{.1in} \rule{0.5em}{0.5em}}}}%
%BeginExpansion
{\hspace{.1in} \rule{0.5em}{0.5em}}%
%EndExpansion
$
\end{proof}

\begin{theorem}
[Quadrance spread duality]If $a_{1}$ and $a_{2}$ are non-null points with
$A_{1}\equiv a_{1}^{\bot}$ and $A_{2}\equiv a_{2}^{\bot}$ the dual lines, then%
\[
q\left(  a_{1},a_{2}\right)  =S\left(  A_{1},A_{2}\right)  .
\]

\end{theorem}

\begin{proof}
This is obvious from the definitions.$%
%TCIMACRO{\TeXButton{Proof box}{{\hspace{.1in} \rule{0.5em}{0.5em}}}}%
%BeginExpansion
{\hspace{.1in} \rule{0.5em}{0.5em}}%
%EndExpansion
$
\end{proof}

\subsection{Cross ratios}

We now show that the quadrance between points and the spread between lines may
be framed entirely within projective geometry with the additional notion of
\textit{perpendicularity} arising from the quadratic form $x^{2}+y^{2}-z^{2}%
$.\textit{\ }The crucial link is provided by the \textit{cross-ratio}. This
naturally extends to other quadratic forms, see for example \cite{Wild4}.
Since we are particularly interested here in explicit formulas in the
hyperbolic setting, this generalization is not treated here, but nevertheless
its existence should be kept in mind.

Note also that the following treatment relating quadrance to a certain cross
ratio is\textit{\ quite different} from the classical one over the real
numbers expressing a hyperbolic distance $d\left(  a_{1},a_{2}\right)  $ as
one half of the log of a different cross-ratio---which requires the line
$a_{1}a_{2}$ to pass through two null points. Here that assumption is
unnecessary. Furthermore our approach dualizes immediately to spreads between
lines also.

To connect with projective geometry, we work in the vector space
$\mathbb{F}^{3}$, with typical vector $v\equiv\left(  a,b,c\right)  $. Then
$\left[  v\right]  =\left[  a:b:c\right]  $ is the associated (hyperbolic)
point, as previously.

The cross-ratio may be viewed as an \textit{affine quantity} that extends to
be \textit{projectively invariant}. Suppose that four non-zero vectors
$v_{1},v_{2},u_{1},u_{2}$ lie in a two-dimensional subspace spanned by vectors
$p$ and $q.$ Then we can write $v_{1}=x_{1}p+y_{1}q,$ $v_{2}=x_{2}p+y_{2}q$,
$u_{1}=z_{1}p+w_{1}q,$ and $u_{2}=z_{2}p+w_{2}q$ uniquely. The
\textbf{cross-ratio} is the ratio of ratios:%
\[
\left(  v_{1},v_{2}:u_{1},u_{2}\right)  \equiv\frac{\left(  x_{1}w_{1}%
-y_{1}z_{1}\right)  }{\left(  x_{1}w_{2}-y_{1}z_{2}\right)  }/\frac{\left(
x_{2}w_{1}-y_{2}z_{1}\right)  }{\left(  x_{2}w_{2}-y_{2}z_{2}\right)  }=\frac{%
\begin{vmatrix}
x_{1} & y_{1}\\
z_{1} & w_{1}%
\end{vmatrix}
}{%
\begin{vmatrix}
x_{1} & y_{1}\\
z_{2} & w_{2}%
\end{vmatrix}
}/\frac{%
\begin{vmatrix}
x_{2} & y_{2}\\
z_{1} & w_{1}%
\end{vmatrix}
}{%
\begin{vmatrix}
x_{2} & y_{2}\\
z_{2} & w_{2}%
\end{vmatrix}
}.
\]

If we choose a new basis $p^{\prime}$ and $q^{\prime}$ with
\begin{align*}
p  & =ap^{\prime}+bq^{\prime}\\
q  & =cp^{\prime}+dq^{\prime}%
\end{align*}
then the cross-ratio is unchanged, as each of the $2\times2$ determinants is
multiplied by the determinant $ad-bc$ of the change of basis matrix. This
ensures that the cross-ratio is indeed well-defined.

Furthermore the cross-ratio only depends on the central lines or (hyperbolic)
points $\left[  v_{1}\right]  ,\left[  v_{2}\right]  ,\left[  u_{1}\right]  $
and $\left[  u_{2}\right]  $, since for example if we multiply $x_{1} $ and
$y_{1}$ by a non-zero factor $\lambda,$ then the cross-ratio is unchanged. So
we write
\[
\left(  \left[  v_{1}\right]  ,\left[  v_{2}\right]  :\left[  u_{1}\right]
,\left[  u_{2}\right]  \right)  \equiv\left(  v_{1},v_{2}:u_{1},u_{2}\right)
.
\]

A useful identity is
\begin{equation}
\left(  \left[  v_{1}\right]  ,\left[  v_{2}\right]  :\left[  u_{1}\right]
,\left[  u_{2}\right]  \right)  +\left(  \left[  v_{1}\right]  ,\left[
u_{1}\right]  :\left[  v_{2}\right]  ,\left[  u_{2}\right]  \right)
=1.\label{CrossRatioIdentity}%
\end{equation}

Now consider the special case when $p=v_{1}$ and $q=v_{2},\ $so that
$u_{1}=z_{1}v_{1}+w_{1}v_{2},$ and $u_{2}=z_{2}v_{1}+w_{2}v_{2}$. Then the
cross ratio becomes simply%
\begin{equation}
\left(  v_{1},v_{2}:u_{1},u_{2}\right)  \equiv\frac{%
\begin{vmatrix}
1 & 0\\
z_{1} & w_{1}%
\end{vmatrix}
}{%
\begin{vmatrix}
1 & 0\\
z_{2} & w_{2}%
\end{vmatrix}
}/\frac{%
\begin{vmatrix}
0 & 1\\
z_{1} & w_{1}%
\end{vmatrix}
}{%
\begin{vmatrix}
0 & 1\\
z_{2} & w_{2}%
\end{vmatrix}
}=\frac{w_{1}}{w_{2}}/\frac{z_{1}}{z_{2}}=\frac{w_{1}z_{2}}{w_{2}z_{1}}%
=\frac{w_{1}}{z_{1}}/\frac{w_{2}}{z_{2}}.\label{CrossRatioRatio}%
\end{equation}

\begin{theorem}
[Quadrance cross ratio]Suppose that $\overline{a_{1}a_{2}}$ is a non-null,
non-nil side with opposite points%
\[%
\begin{tabular}
[c]{lllll}%
$o_{1}\equiv\left(  a_{1}a_{2}\right)  a_{1}^{\bot}$ &  & \textit{and} &  &
$o_{2}\equiv\left(  a_{1}a_{2}\right)  a_{2}^{\bot}.$%
\end{tabular}
\]
Then%
\[
q\left(  a_{1},a_{2}\right)  =\left(  a_{1},o_{2}:a_{2},o_{1}\right)  .
\]

\end{theorem}

\begin{proof}
Write $a_{1}\equiv\left[  v_{1}\right]  $ and $a_{2}\equiv\left[
v_{2}\right]  $ where $v_{1}\equiv\left(  x_{1},y_{1},z_{1}\right)  $ and
$v_{2}\equiv\left(  x_{2},y_{2},z_{2}\right)  $. Then $v_{1}$ and $v_{2}$ are
linearly independent, since $a_{1}$ and $a_{2}$ are distinct. As
$\overline{a_{1}a_{2}}$ is a non-null side, the Opposite points theorem states
that the opposite points $o_{1}=\left(  a_{1}a_{2}\right)  a_{1}^{\bot}$ and
$o_{2}=\left(  a_{1}a_{2}\right)  a_{2}^{\bot}$ can be expressed as
\begin{align*}
o_{1}  & =\left[  \left(  x_{1}x_{2}+y_{1}y_{2}-z_{1}z_{2}\right)
v_{1}-\left(  x_{1}^{2}+y_{1}^{2}-z_{1}^{2}\right)  v_{2}\right] \\
o_{2}  & =\left[  \left(  x_{2}^{2}+y_{2}^{2}-z_{2}^{2}\right)  v_{1}-\left(
x_{1}x_{2}+y_{1}y_{2}-z_{1}z_{2}\right)  v_{2}\right]  .
\end{align*}
So $v_{1},v_{2},o_{1}$ and $o_{2}$ lie in the two-dimensional subspace spanned
by $v_{1}$ and $v_{2}.$ Using $v_{1}$ and $v_{2}$ as a basis,
(\ref{CrossRatioRatio}) shows that
\begin{align*}
\left(  a_{1},a_{2}:o_{2},o_{1}\right)   & =\frac{-\left(  x_{1}x_{2}%
+y_{1}y_{2}-z_{1}z_{2}\right)  }{\left(  x_{2}^{2}+y_{2}^{2}-z_{2}^{2}\right)
}/\frac{-\left(  x_{1}^{2}+y_{1}^{2}-z_{1}^{2}\right)  }{\left(  x_{1}%
x_{2}+y_{1}y_{2}-z_{1}z_{2}\right)  }\\
& =\frac{\left(  x_{1}x_{2}+y_{1}y_{2}-z_{1}z_{2}\right)  ^{2}}{\left(
x_{1}^{2}+y_{1}^{2}-z_{1}^{2}\right)  \left(  x_{2}^{2}+y_{2}^{2}-z_{2}%
^{2}\right)  }.
\end{align*}
So using (\ref{CrossRatioIdentity})
\[
\left(  a_{1},o_{2}:a_{2},o_{1}\right)  =1-\left(  a_{1},a_{2}:o_{2}%
,o_{1}\right)  =q\left(  a_{1},a_{2}\right)  .%
%TCIMACRO{\TeXButton{Proof box}{{\hspace{.1in} \rule{0.5em}{0.5em}}}}%
%BeginExpansion
{\hspace{.1in} \rule{0.5em}{0.5em}}%
%EndExpansion
\]

\end{proof}

It is also possible to discuss cross-ratios of lines, and there is a dual
result relating the spread between lines to such a cross-ratio.

\subsection{Triple quad formula}

We now come to the first two of the \textit{main\ laws of hyperbolic
trigonometry}: the \textit{Triple quad formula} and the \textit{Triple spread
formula}. Both involve the \textbf{Triple spread function}%
\begin{align}
S\left(  a,b,c\right)   & \equiv\allowbreak\left(  a+b+c\right)  ^{2}-2\left(
a^{2}+b^{2}+c^{2}\right)  -4abc\label{Triple Spread Fn}\\
& =\allowbreak2ab+2ac+2bc-a^{2}-b^{2}-c^{2}-4abc\\
& =-%
\begin{vmatrix}
0 & a & b & 1\\
a & 0 & c & 1\\
b & c & 0 & 1\\
1 & 1 & 1 & 2
\end{vmatrix}
\\
& =4\left(  1-a\right)  \left(  1-b\right)  \left(  1-c\right)  -\left(
a+b+c-2\right)  ^{2}\label{Triple Sp4}%
\end{align}
which also plays a major role in planar rational trigonometry (see
\cite{Wild1}). It should be noted that the Euclidean Triple quad formula%
\begin{equation}
\left(  Q_{1}+Q_{2}+Q_{3}\right)  ^{2}=2\left(  Q_{1}^{2}+Q_{2}^{2}+Q_{3}%
^{2}\right) \label{PlanarTripleQuad}%
\end{equation}
is the relation between the three Euclidean quadrances formed by three
collinear points. The hyperbolic version we now establish is a deformation of
(\ref{PlanarTripleQuad}).

\begin{theorem}
[Triple quad formula]If the points $a_{1},a_{2}$ and $a_{3}$ are collinear,
and $q_{1}\equiv q\left(  a_{2},a_{3}\right)  $, $q_{2}\equiv q\left(
a_{1},a_{3}\right)  $ and $q_{3}\equiv q\left(  a_{1},a_{2}\right)  $, then%
\[
\left(  q_{1}+q_{2}+q_{3}\right)  ^{2}=2\left(  q_{1}^{2}+q_{2}^{2}+q_{3}%
^{2}\right)  +4q_{1}\allowbreak q_{2}q_{3}.
\]

\end{theorem}

\begin{proof}
Suppose that the points are $a_{1}\equiv\left[  x_{1}:y_{1}:z_{1}\right]  $,
$a_{2}\equiv\left[  x_{2}:y_{2}:z_{2}\right]  $ and $a_{3}\equiv\left[
x_{3}:y_{3}:z_{3}\right]  $. Then
\begin{align*}
q_{1}  & \equiv q\left(  a_{2},a_{3}\right)  =1-\frac{\left(  x_{2}x_{3}%
+y_{2}y_{3}-z_{2}z_{3}\right)  ^{2}}{\left(  x_{2}^{2}+y_{2}^{2}-z_{2}%
^{2}\right)  \left(  x_{3}^{2}+y_{3}^{2}-z_{3}^{2}\right)  }\\
q_{2}  & \equiv q\left(  a_{1},a_{3}\right)  =1-\frac{\left(  x_{1}x_{3}%
+y_{1}y_{3}-z_{1}z_{3}\right)  ^{2}}{\left(  x_{1}^{2}+y_{1}^{2}-z_{1}%
^{2}\right)  \left(  x_{3}^{2}+y_{3}^{2}-z_{3}^{2}\right)  }\\
q_{3}  & \equiv q\left(  a_{1},a_{2}\right)  =1-\frac{\left(  x_{1}x_{2}%
+y_{1}y_{2}-z_{1}z_{2}\right)  ^{2}}{\left(  x_{1}^{2}+y_{1}^{2}-z_{1}%
^{2}\right)  \left(  x_{2}^{2}+y_{2}^{2}-z_{2}^{2}\right)  }.
\end{align*}

The expression%
\[
\left(  \left(  1-q_{1}\right)  +\left(  1-q_{2}\right)  +\left(
1-q_{3}\right)  -1\right)  ^{2}-4\left(  1-q_{1}\right)  \left(
1-q_{2}\right)  \left(  1-q_{3}\right)
\]
is a difference of squares. One of the factors is%
\begin{align*}
& \frac{\left(  x_{2}x_{3}+y_{2}y_{3}-z_{2}z_{3}\right)  ^{2}}{\left(
x_{2}^{2}+y_{2}^{2}-z_{2}^{2}\right)  \left(  x_{3}^{2}+y_{3}^{2}-z_{3}%
^{2}\right)  }+\frac{\left(  x_{1}x_{3}+y_{1}y_{3}-z_{1}z_{3}\right)  ^{2}%
}{\left(  x_{1}^{2}+y_{1}^{2}-z_{1}^{2}\right)  \left(  x_{3}^{2}+y_{3}%
^{2}-z_{3}^{2}\right)  }+\frac{\left(  x_{1}x_{2}+y_{1}y_{2}-z_{1}%
z_{2}\right)  ^{2}}{\left(  x_{1}^{2}+y_{1}^{2}-z_{1}^{2}\right)  \left(
x_{2}^{2}+y_{2}^{2}-z_{2}^{2}\right)  }-1\\
& -\frac{2\left(  x_{2}x_{3}+y_{2}y_{3}-z_{2}z_{3}\right)  \left(  x_{1}%
x_{3}+y_{1}y_{3}-z_{1}z_{3}\right)  \left(  x_{1}x_{2}+y_{1}y_{2}-z_{1}%
z_{2}\right)  }{\left(  x_{1}^{2}+y_{1}^{2}-z_{1}^{2}\right)  \left(
x_{2}^{2}+y_{2}^{2}-z_{2}^{2}\right)  \left(  x_{3}^{2}+y_{3}^{2}-z_{3}%
^{2}\right)  }.
\end{align*}
Somewhat remarkably, this expression is identically equally to%
\begin{equation}
\frac{\left(  x_{1}y_{2}z_{3}-x_{1}y_{3}z_{2}+x_{2}y_{3}z_{1}-x_{3}y_{2}%
z_{1}+x_{3}y_{1}z_{2}-x_{2}y_{1}z_{3}\right)  ^{2}}{\left(  x_{1}^{2}%
+y_{1}^{2}-z_{1}^{2}\right)  \left(  x_{2}^{2}+y_{2}^{2}-z_{2}^{2}\right)
\left(  x_{3}^{2}+y_{3}^{2}-z_{3}^{2}\right)  }.\label{Quadrea first}%
\end{equation}
If $a_{1},a_{2}$ and $a_{3}$ are collinear, then from the Collinear points
theorem the numerator of (\ref{Quadrea first}) is zero, so that
\begin{equation}
\left(  2-q_{1}-q_{2}-q_{3}\right)  ^{2}=4\left(  1-q_{1}\right)  \left(
1-q_{2}\right)  \left(  1-q_{3}\right)  .\label{TripleQuadAlt}%
\end{equation}
Comparing (\ref{Triple Spread Fn}) and (\ref{Triple Sp4}), this can be
rewritten as
\[
\left(  q_{1}+q_{2}+q_{3}\right)  ^{2}=2\left(  q_{1}^{2}+q_{2}^{2}+q_{3}%
^{2}\right)  +4q_{1}q_{2}q_{3}.%
%TCIMACRO{\TeXButton{Proof box}{{\hspace{.1in} \rule{0.5em}{0.5em}}}}%
%BeginExpansion
{\hspace{.1in} \rule{0.5em}{0.5em}}%
%EndExpansion
\]
The Triple quad formula can be rewritten more compactly, using the Triple
spread function, as
\[
S\left(  q_{1},q_{2},q_{3}\right)  =0.
\]

\end{proof}

\begin{example}
Consider the case when $a_{1}\equiv\left[  0:0:1\right]  $, $a_{2}%
\equiv\left[  x:y:1\right]  $ and $a_{3}\equiv\left[  x:-y:1\right]  $. The
corresponding lines are
\[
L_{1}=\left(  1:0:x\right)  \qquad L_{2}=\left(  y:x:0\right)  \qquad
L_{3}=\left(  -y:x:0\right)  .
\]
Then%
\[
q_{1}=\frac{4\left(  x^{2}-1\right)  y^{2}}{\left(  x^{2}+y^{2}-1\right)
^{2}}\qquad q_{2}=\frac{x^{2}+y^{2}}{x^{2}+y^{2}-1}\qquad q_{3}=\frac
{x^{2}+y^{2}}{x^{2}+y^{2}-1}%
\]
\newline and a computation shows that
\[
S\left(  q_{1},q_{2},q_{3}\right)  =\allowbreak\frac{16\left(  1+y^{2}\right)
\left(  1-x^{2}\right)  \allowbreak x^{2}y^{2}}{\left(  x^{2}+y^{2}-1\right)
^{4}}.
\]
This is zero if $x=0$ or $y=0,$ and also if $x^{2}=1$, in which case
$L_{1}\equiv a_{2}a_{3}$ is a null line, and if $y\neq0$ then the three points
$a_{1},a_{2}$ and $a_{3}$ are not collinear. But there is another
possibility---if we work in the field obtained by adjoining $i$ to the
rational numbers, where $i^{2}=-1$, then setting $y=i$ we find an example
involving \textit{non-null lines }for which the converse of the Triple quad
formula also \textit{does not hold}.$%
%TCIMACRO{\TeXButton{Math Diamond}{\hspace{.1in}\diamond}}%
%BeginExpansion
\hspace{.1in}\diamond
%EndExpansion
$
\end{example}

\subsection{Triple spread formula}

\begin{theorem}
[Triple spread formula]If the lines $L_{1},L_{2}$ and $L_{3}$ are concurrent,
and $S_{1}\equiv S\left(  L_{2},L_{3}\right)  $, $S_{2}\equiv S\left(
L_{1},L_{3}\right)  $ and $S_{3}\equiv S\left(  L_{1},L_{2}\right)  $, then
\[
\left(  S_{1}+S_{2}+S_{3}\right)  ^{2}=2\left(  S_{1}^{2}+S_{2}^{2}+S_{3}%
^{2}\right)  +4S_{1}S_{2}S_{3}.
\]

\end{theorem}

\begin{proof}
This is dual to the Triple quad formula.$%
%TCIMACRO{\TeXButton{Proof box}{{\hspace{.1in} \rule{0.5em}{0.5em}}}}%
%BeginExpansion
{\hspace{.1in} \rule{0.5em}{0.5em}}%
%EndExpansion
$
\end{proof}

The Triple spread formula can also be rewritten as
\[
S\left(  S_{1},S_{2},S_{3}\right)  =0.
\]

Here are some secondary results that involve the Triple spread function, and
apply to both collinear points and concurrent lines.

\begin{theorem}
[Complementary quadrances spreads]Suppose that $a,b$ and $c$ are numbers which
satisfy $S\left(  a,b,c\right)  =0 $. If $c=1$ then
\[
a+b=1.
\]

\end{theorem}

\begin{proof}
This is a consequence of the identity
\[
S\left(  a,b,c\right)  =\left(  1-c\right)  \left(  c+\left(  1-2a\right)
\left(  1-2b\right)  \right)  -\left(  a+b-1\right)  ^{2}.%
%TCIMACRO{\TeXButton{Proof box}{{\hspace{.1in} \rule{0.5em}{0.5em}}}}%
%BeginExpansion
{\hspace{.1in} \rule{0.5em}{0.5em}}%
%EndExpansion
\]

\end{proof}

The proof suggests that the converse of this result does not hold, since if
$S\left(  a,b,c\right)  =0$ and $a+b=1,$ then it is also possible that
$c=-\left(  1-2a\right)  \left(  1-2b\right)  $. In fact we will see an
example of just this situation in the Singly null singly nil theorem.

\begin{theorem}
[Equal quadrances spreads]Suppose that $a,b$ and $c$ are numbers which satisfy
$S\left(  a,b,c\right)  =0 $. If $a=b$ then either%
\[
c=0\qquad\mathit{or}\qquad c=4a\left(  1-a\right)  .
\]

\end{theorem}

\begin{proof}
This is a consequence of the identity
\[
S\left(  a,b,c\right)  =\allowbreak\left(  a-b\right)  \left(
b-a-2c+4ac\right)  -c\left(  c-4a+4a^{2}\right)  .%
%TCIMACRO{\TeXButton{Proof box}{{\hspace{.1in} \rule{0.5em}{0.5em}}}}%
%BeginExpansion
{\hspace{.1in} \rule{0.5em}{0.5em}}%
%EndExpansion
\]

\end{proof}

The polynomial%
\[
S_{2}\left(  x\right)  =4x\left(  1-x\right)
\]
which appears in this theorem is the \textbf{second spread polynomial}. It
plays a major role in hyperbolic geometry, and is also well-known in chaos
theory as the \textit{logistic map}.

\subsection{Pythagoras' theorem}

Pythagoras' theorem is of course very important. The planar Euclidean version
involves the formula
\[
Q_{3}=Q_{1}+Q_{2}.
\]
The hyperbolic version involves a deformation of this.

\begin{theorem}
[Pythagoras]Suppose that $a_{1},a_{2}$ and $a_{3}$ are distinct points with
quadrances $q_{1}\equiv q\left(  a_{2},a_{3}\right)  $, $q_{2}\equiv q\left(
a_{1},a_{3}\right)  $ and $q_{3}\equiv q\left(  a_{1},a_{2}\right)  $. If the
lines $a_{1}a_{3}$ and $a_{2}a_{3}$ are perpendicular, then%
\[
q_{3}=q_{1}+q_{2}-q_{1}q_{2}.
\]

\end{theorem}

\begin{proof}
Suppose that $a_{1}\equiv\left[  x_{1}:y_{1}:z_{1}\right]  $, $a_{2}%
\equiv\left[  x_{2}:y_{2}:z_{2}\right]  $ and $a_{3}\equiv\left[  x_{3}%
:y_{3}:z_{3}\right]  $. The lines $a_{1}a_{3}$ and $a_{2}a_{3}$ are
perpendicular precisely when$\allowbreak$%
\[
\left(  y_{2}z_{3}-y_{3}z_{2}\right)  \left(  y_{3}z_{1}-y_{1}z_{3}\right)
+\left(  z_{2}x_{3}-z_{3}x_{2}\right)  \left(  z_{3}x_{1}-z_{1}x_{3}\right)
-\left(  x_{3}y_{2}-x_{2}y_{3}\right)  \left(  x_{1}y_{3}-x_{3}y_{1}\right)
=0.
\]
This condition can be rewritten as%
\begin{equation}
\left(  x_{3}^{2}+y_{3}^{2}-z_{3}^{2}\right)  \left(  x_{1}x_{2}+y_{1}%
y_{2}-z_{1}z_{2}\right)  -\left(  x_{2}x_{3}+y_{2}y_{3}-z_{2}z_{3}\right)
\left(  x_{1}x_{3}+y_{1}y_{3}-z_{1}z_{3}\right)  =0.\label{Pythagoras3}%
\end{equation}
Now from the definition of quadrance,
\begin{align}
& \left(  1-q_{3}\right)  -\left(  1-q_{1}\right)  \left(  1-q_{2}\right)
\nonumber\\
& =\frac{\left(  x_{1}x_{2}+y_{1}y_{2}-z_{1}z_{2}\right)  ^{2}}{\left(
x_{1}^{2}+y_{1}^{2}-z_{1}^{2}\right)  \left(  x_{2}^{2}+y_{2}^{2}-z_{2}%
^{2}\right)  }-\frac{\left(  x_{2}x_{3}+y_{2}y_{3}-z_{2}z_{3}\right)  ^{2}%
}{\left(  x_{2}^{2}+y_{2}^{2}-z_{2}^{2}\right)  \left(  x_{3}^{2}+y_{3}%
^{2}-z_{3}^{2}\right)  }\frac{\left(  x_{1}x_{3}+y_{1}y_{3}-z_{1}z_{3}\right)
^{2}}{\left(  x_{1}^{2}+y_{1}^{2}-z_{1}^{2}\right)  \left(  x_{3}^{2}%
+y_{3}^{2}-z_{3}^{2}\right)  }\nonumber\\
& =\frac{\left(  x_{3}^{2}+y_{3}^{2}-z_{3}^{2}\right)  ^{2}\left(  x_{1}%
x_{2}+y_{1}y_{2}-z_{1}z_{2}\right)  ^{2}-\left(  x_{2}x_{3}+y_{2}y_{3}%
-z_{2}z_{3}\right)  ^{2}\left(  x_{1}x_{3}+y_{1}y_{3}-z_{1}z_{3}\right)  ^{2}%
}{\left(  x_{1}^{2}+y_{1}^{2}-z_{1}^{2}\right)  \left(  x_{2}^{2}+y_{2}%
^{2}-z_{2}^{2}\right)  \allowbreak\left(  x_{3}^{2}+y_{3}^{2}-z_{3}%
^{2}\right)  ^{2}}.\label{Pythagoras2}%
\end{align}
The numerator is a difference of squares, and one of the factors is the left
hand side of (\ref{Pythagoras3}), so if $a_{1}a_{3}$ and $a_{2}a_{3}$ are
perpendicular, then%
\[
1-q_{3}=\left(  1-q_{1}\right)  \left(  1-q_{2}\right)  .
\]
Now rewrite this as%
\[
q_{3}=q_{1}+q_{2}-q_{1}q_{2}.%
%TCIMACRO{\TeXButton{Proof box}{{\hspace{.1in} \rule{0.5em}{0.5em}}}}%
%BeginExpansion
{\hspace{.1in} \rule{0.5em}{0.5em}}%
%EndExpansion
\]

\end{proof}

Note that the converse does not follow from this argument, as the
\textit{other} factor of the numerator of (\ref{Pythagoras2}) might be zero.

\begin{theorem}
[Pythagoras' dual]Suppose that $L_{1},L_{2}$ and $L_{3}$ are distinct lines
with spreads $S_{1}\equiv S\left(  L_{2},L_{3}\right)  $, $S_{2}\equiv
S\left(  L_{1},L_{3}\right)  $ and $S_{3}\equiv S\left(  L_{1},L_{2}\right)
$. If the points $L_{1}L_{3}$ and $L_{2}L_{3}$ are perpendicular, then%
\[
S_{3}=S_{1}+S_{2}-S_{1}S_{2}.
\]

\end{theorem}

\begin{proof}
This is dual to Pythagoras' theorem.$%
%TCIMACRO{\TeXButton{Proof box}{{\hspace{.1in} \rule{0.5em}{0.5em}}}}%
%BeginExpansion
{\hspace{.1in} \rule{0.5em}{0.5em}}%
%EndExpansion
$
\end{proof}

\begin{example}
Consider the triangle $\overline{a_{1}a_{2}a_{3}}$ where $a_{1}\equiv\left[
x:0:1\right]  $, $a_{2}\equiv\left[  0:y:1\right]  $ and $a_{3}\equiv\left[
0:0:1\right]  $ as in Example 3. The lines are
\[
L_{1}\equiv a_{2}a_{3}=\left(  1:0:0\right)  \qquad L_{2}\equiv a_{1}%
a_{3}=\left(  0:1:0\right)  \qquad L_{3}\equiv a_{1}a_{2}=\left(
y:x:xy\right)  .
\]
The quadrances are%
\[
q_{1}=-\frac{y^{2}}{1-y^{2}}\qquad q_{2}=-\frac{x^{2}}{1-x^{2}}\qquad
q_{3}=\frac{x^{2}y^{2}-x^{2}-y^{2}}{\left(  1-x^{2}\right)  \left(
1-y^{2}\right)  }%
\]
and the spreads are%
\[
S_{1}=\frac{\left(  1-x^{2}\right)  y^{2}}{x^{2}+y^{2}-x^{2}y^{2}}\qquad
S_{2}=\frac{x^{2}\left(  1-y^{2}\right)  }{x^{2}+y^{2}-x^{2}y^{2}}\qquad
S_{3}=1.
\]
Since $S_{3}=1,$ the lines $L_{1}$ and $L_{2}$ are perpendicular, and
Pythagoras' theorem asserts that $q_{3}=q_{1}+q_{2}-q_{1}q_{2},$ which can be
verified directly.$%
%TCIMACRO{\TeXButton{Math Diamond}{\hspace{.1in}\diamond}}%
%BeginExpansion
\hspace{.1in}\diamond
%EndExpansion
$
\end{example}

\begin{example}
Consider the triangle $\overline{a_{1}a_{2}a_{3}}$ where $a_{1}\equiv\left[
0:0:1\right]  $, $a_{2}\equiv\left[  x:y:1\right]  $ and $a_{3}\equiv\left[
x:-y:1\right]  $ as in Example 11. The spread $S_{1}\equiv S\left(  a_{1}%
a_{2},a_{1}a_{3}\right)  $ is
\[
S_{1}=\frac{4x^{2}y^{2}}{\left(  x^{2}+y^{2}\right)  ^{2}}%
\]
\newline Using the quadrances computed in Example 11,
\[
q_{1}-q_{2}-q_{3}+q_{2}q_{3}=-\frac{\left(  x^{2}-y^{2}-2\right)  \left(
x^{2}-y^{2}\right)  }{\left(  x^{2}+y^{2}-1\right)  ^{2}}%
\]
while%
\[
1-S_{1}=1-\frac{4x^{2}y^{2}}{\left(  x^{2}+y^{2}\right)  ^{2}}=\frac{\left(
x^{2}-y^{2}\right)  ^{2}}{\left(  x^{2}+y^{2}\right)  ^{2}}.
\]
So we see that $S_{1}=1$ does imply $q_{1}=q_{2}+q_{3}-q_{1}q_{2}$, verifying
again Pythagoras' theorem. But if $x^{2}-y^{2}=2$ and $x\neq\pm y, $ then
$q_{1}=q_{2}+q_{3}-q_{1}q_{2}$ while $S_{1}\neq1$. This shows that the
\textit{converse of Pythagoras' theorem does not hold}.$%
%TCIMACRO{\TeXButton{Math Diamond}{\hspace{.1in}\diamond}}%
%BeginExpansion
\hspace{.1in}\diamond
%EndExpansion
$
\end{example}

\subsection{Product and cross}

In planar rational trigonometry, if $s\equiv s\left(  l_{1},l_{2}\right)  $ is
the spread between two lines $l_{1}$ and $l_{2}$, then $c\left(  l_{1}%
,l_{2}\right)  \equiv1-s$ is called the \textit{cross} between the lines. This
is a useful concept in its own right, and here we consider the hyperbolic
version, together with the corresponding notion between points, called
\textit{product}. The latter has no analogue in Euclidean geometry. Here are
the definitions.

The \textbf{product} between points $a_{1}\equiv\left[  x_{1}:y_{1}%
:z_{1}\right]  $ and $a_{2}\equiv\left[  x_{2}:y_{2}:z_{2}\right]  $ is the
number%
\[
p\left(  a_{1},a_{2}\right)  \equiv\frac{\left(  x_{1}x_{2}+y_{1}y_{2}%
-z_{1}z_{2}\right)  ^{2}}{\left(  x_{1}^{2}+y_{1}^{2}-z_{1}^{2}\right)
\left(  x_{2}^{2}+y_{2}^{2}-z_{2}^{2}\right)  }.
\]
The \textbf{cross} between lines $L_{1}\equiv\left(  l_{1}:m_{1}:n_{1}\right)
$ and $L_{2}\equiv\left(  l_{2}:m_{2}:n_{2}\right)  $ is the number%
\[
C\left(  L_{1},L_{2}\right)  \equiv\frac{\left(  l_{1}l_{2}+m_{1}m_{2}%
-n_{1}n_{2}\right)  ^{2}}{\left(  l_{1}^{2}+m_{1}^{2}-n_{1}^{2}\right)
\allowbreak\left(  l_{2}^{2}+m_{2}^{2}-n_{2}^{2}\right)  }.
\]
Then clearly
\[
q\left(  a_{1},a_{2}\right)  +p\left(  a_{1},a_{2}\right)  =1
\]
and%
\[
S\left(  L_{1},L_{2}\right)  +C\left(  L_{1},L_{2}\right)  =1.
\]

Pythagoras' theorem can be rewritten in terms of products as
\[
p_{3}=p_{1}p_{2}%
\]
while Pythagoras' dual theorem can be rewritten as
\[
C_{3}=C_{1}C_{2}.
\]
The Triple quad formula can be rewritten as
\[
\left(  p_{1}+p_{2}+p_{3}-1\right)  ^{2}=4p_{1}p_{2}p_{3}%
\]
while the Triple spread formula can be rewritten as%
\[
\left(  C_{1}+C_{2}+C_{3}-1\right)  ^{2}=4C_{1}C_{2}C_{3}.
\]

\subsection{The Spread law}

The Spread law is the analog in universal hyperbolic geometry of the
\textit{Sine law}. First we establish a useful, but somewhat lengthy, formula.

\begin{theorem}
[Spread formula]Suppose that $a_{1}\equiv\left[  x_{1}:y_{1}:z_{1}\right]  $,
$a_{2}\equiv\left[  x_{2}:y_{2}:z_{2}\right]  $ and $a_{3}\equiv\left[
x_{3}:y_{3}:z_{3}\right]  $ are distinct points with $a_{1}a_{2}$ and
$a_{1}a_{3}$ both non-null lines. Then $S_{1}\equiv S\left(  a_{1}a_{2}%
,a_{1}a_{3}\right)  $ is given by the expression%
\begin{equation}
S_{1}=-\frac{\left(  x_{1}^{2}+y_{1}^{2}-z_{1}^{2}\right)  \left(  x_{1}%
y_{2}z_{3}-x_{1}y_{3}z_{2}+x_{2}y_{3}z_{1}-x_{3}y_{2}z_{1}+x_{3}y_{1}%
z_{2}-x_{2}y_{1}z_{3}\right)  ^{2}\allowbreak}{\left(  l_{2}^{2}+m_{2}%
^{2}-n_{2}^{2}\right)  \allowbreak\left(  l_{3}^{2}+m_{3}^{2}-n_{3}%
^{2}\right)  }\label{SpreadFormula}%
\end{equation}
where%
\begin{align*}
&
\begin{tabular}
[c]{lllll}%
$l_{2}\equiv y_{3}z_{1}-y_{1}z_{3}$ &  & $m_{2}\equiv z_{3}x_{1}-z_{1}x_{3}$ &
& $n_{2}\equiv x_{1}y_{3}-x_{3}y_{1}$%
\end{tabular}
\\
&
\begin{tabular}
[c]{lllll}%
$l_{3}\equiv y_{1}z_{2}-y_{2}z_{1}$ &  & $m_{3}\equiv z_{1}x_{2}-z_{2}x_{1}$ &
& $n_{3}\equiv x_{2}y_{1}-x_{1}y_{2}.$%
\end{tabular}
\end{align*}

\end{theorem}

\begin{proof}
By the Join of points theorem%
\begin{align*}
a_{1}a_{3}  & =\left(  l_{2}:m_{2}:n_{2}\right) \\
a_{1}a_{2}  & =\left(  l_{3}:m_{3}:n_{3}\right)
\end{align*}
where $l_{2},m_{2},n_{2},l_{3},m_{3}$ and $n_{3}$ are as in the statement of
the theorem.

From the Spread theorem
\begin{equation}
S_{1}=-\frac{\left(  m_{2}n_{3}-m_{3}n_{2}\right)  ^{2}+\left(  n_{2}%
l_{3}-n_{3}l_{2}\right)  ^{2}-\left(  l_{2}m_{3}-l_{3}m_{2}\right)  ^{2}%
}{\left(  l_{2}^{2}+m_{2}^{2}-n_{2}^{2}\right)  \left(  l_{3}^{2}+m_{3}%
^{2}-n_{3}^{2}\right)  \allowbreak}.\label{SpreadForm3}%
\end{equation}
Now the polynomial identity%
\begin{align*}
& \left(  \left(  z_{3}x_{1}-z_{1}x_{3}\right)  \left(  x_{2}y_{1}-x_{1}%
y_{2}\right)  -\left(  z_{1}x_{2}-z_{2}x_{1}\right)  \left(  x_{1}y_{3}%
-x_{3}y_{1}\right)  \right)  ^{2}\\
& +\left(  \left(  x_{1}y_{3}-x_{3}y_{1}\right)  \left(  y_{1}z_{2}-y_{2}%
z_{1}\right)  -\left(  x_{2}y_{1}-x_{1}y_{2}\right)  \left(  y_{3}z_{1}%
-y_{1}z_{3}\right)  \right)  ^{2}\\
& -\left(  \left(  y_{3}z_{1}-y_{1}z_{3}\right)  \left(  z_{1}x_{2}-z_{2}%
x_{1}\right)  -\left(  y_{1}z_{2}-y_{2}z_{1}\right)  \left(  z_{3}x_{1}%
-z_{1}x_{3}\right)  \right)  ^{2}\\
& =\left(  x_{1}^{2}+y_{1}^{2}-z_{1}^{2}\right)  \left(  x_{1}y_{3}z_{2}%
-x_{1}y_{2}z_{3}+x_{2}y_{1}z_{3}-x_{2}z_{1}y_{3}-y_{1}x_{3}z_{2}+x_{3}%
y_{2}z_{1}\right)  ^{2}\allowbreak
\end{align*}
applied to the numerator of the right hand side of (\ref{SpreadForm3}) gives
\begin{equation}
S_{1}=-\frac{\left(  x_{1}^{2}+y_{1}^{2}-z_{1}^{2}\right)  \left(  x_{1}%
y_{2}z_{3}-x_{1}y_{3}z_{2}+x_{2}y_{3}z_{1}-x_{3}y_{2}z_{1}+x_{3}y_{1}%
z_{2}-x_{2}y_{1}z_{3}\right)  ^{2}\allowbreak}{\left(  l_{2}^{2}+m_{2}%
^{2}-n_{2}^{2}\right)  \allowbreak\left(  l_{3}^{2}+m_{3}^{2}-n_{3}%
^{2}\right)  }.%
%TCIMACRO{\TeXButton{Proof box}{{\hspace{.1in} \rule{0.5em}{0.5em}}}}%
%BeginExpansion
{\hspace{.1in} \rule{0.5em}{0.5em}}%
%EndExpansion
\label{SpreadForm2}%
\end{equation}

\end{proof}

The hyperbolic Spread law is essentially the same as the Euclidean one, which
has the form%
\[
\frac{s_{1}}{Q_{1}}=\frac{s_{2}}{Q_{2}}=\frac{s_{3}}{Q_{3}}.
\]

\begin{theorem}
[Spread law]Suppose that $a_{1},a_{2}$ and $a_{3}$ are distinct points with
quadrances $q_{1}\equiv q\left(  a_{2},a_{3}\right)  $, $q_{2}\equiv q\left(
a_{1},a_{3}\right)  $ and $q_{3}\equiv q\left(  a_{1},a_{2}\right)  $, and
spreads $S_{1}\equiv S\left(  a_{1}a_{2},a_{1}a_{3}\right)  $, $S_{2}\equiv
S\left(  a_{1}a_{2},a_{2}a_{3}\right)  $ and $S_{3}\equiv S\left(  a_{1}%
a_{3},a_{2}a_{3}\right)  $. Then
\[
\frac{S_{1}}{q_{1}}=\frac{S_{2}}{q_{2}}=\frac{S_{3}}{q_{3}}.
\]

\end{theorem}

\begin{proof}
If the spreads $S_{1},S_{2}$ and $S_{3}$ are defined, then each of the lines
$a_{1}a_{2},$ $a_{2}a_{3}$ and $a_{1}a_{3}$ is non-null and so by the Zero
quadrance theorem each of the quadrances $q_{1},q_{2}$ and $q_{3}$ is
non-zero. Using the notation of the previous theorem, as well as
\[%
\begin{tabular}
[c]{lllll}%
$l_{1}\equiv y_{2}z_{3}-y_{3}z_{2}$ &  & $m_{1}\equiv z_{2}x_{3}-z_{3}x_{2}$ &
& $n_{1}\equiv x_{3}y_{2}-x_{2}y_{3}$%
\end{tabular}
\]
the Quadrance theorem gives%
\begin{equation}
q_{1}=-\frac{l_{1}^{2}+m_{1}^{2}-n_{1}^{2}}{\left(  x_{2}^{2}+y_{2}^{2}%
-z_{2}^{2}\right)  \left(  x_{3}^{2}+y_{3}^{2}-z_{3}^{2}\right)
}.\label{Quadform2}%
\end{equation}
Combine the Spread formula (\ref{SpreadFormula}) with (\ref{Quadform2}) to get%
\begin{align}
\frac{S_{1}}{q_{1}}  & =\frac{\left(  x_{1}^{2}+y_{1}^{2}-z_{1}^{2}\right)
\left(  x_{2}^{2}+y_{2}^{2}-z_{2}^{2}\right)  \left(  x_{3}^{2}+y_{3}%
^{2}-z_{3}^{2}\right)  }{\left(  l_{1}^{2}+m_{1}^{2}-n_{1}^{2}\right)  \left(
l_{2}^{2}+m_{2}^{2}-n_{2}^{2}\right)  \allowbreak\left(  l_{3}^{2}+m_{3}%
^{2}-n_{3}^{2}\right)  }\nonumber\\
& \times\left(  x_{1}y_{2}z_{3}-x_{1}y_{3}z_{2}+x_{2}y_{3}z_{1}-x_{3}%
y_{2}z_{1}+x_{3}y_{1}z_{2}-x_{2}y_{1}z_{3}\right)  ^{2}\allowbreak
.\label{S over q}%
\end{align}
Since this a \textit{symmetric} rational expression in the three indices, we
conclude that
\[
\frac{S_{1}}{q_{1}}=\frac{S_{2}}{q_{2}}=\frac{S_{3}}{q_{3}}.%
%TCIMACRO{\TeXButton{Proof box}{{\hspace{.1in} \rule{0.5em}{0.5em}}}}%
%BeginExpansion
{\hspace{.1in} \rule{0.5em}{0.5em}}%
%EndExpansion
\]

\end{proof}

The dual of the Spread law is essentially the same, but to maintain the
duality principle we state it separately.

\begin{theorem}
[Spread dual law]Suppose that $L_{1},L_{2}$ and $L_{3}$ are distinct lines
with spreads $S_{1}\equiv S\left(  L_{2},L_{3}\right)  $, $S_{2}\equiv
S\left(  L_{1},L_{3}\right)  $ and $S_{3}\equiv S\left(  L_{1},L_{2}\right)
$, and quadrances $q_{1}\equiv q\left(  L_{1}L_{2},L_{1}L_{3}\right)  $,
$q_{2}\equiv q\left(  L_{1}L_{2},L_{2}L_{3}\right)  $ and $q_{3}\equiv
q\left(  L_{1}L_{3},L_{2}L_{3}\right)  $. Then%
\[
\frac{q_{1}}{S_{1}}=\frac{q_{2}}{S_{2}}=\frac{q_{3}}{S_{3}}.
\]

\end{theorem}

\begin{proof}
This is dual to the Spread law.$%
%TCIMACRO{\TeXButton{Proof box}{{\hspace{.1in} \rule{0.5em}{0.5em}}}}%
%BeginExpansion
{\hspace{.1in} \rule{0.5em}{0.5em}}%
%EndExpansion
$
\end{proof}

\subsection{Quadrea and quadreal}

If $a_{1}\equiv\left[  x_{1}:y_{1}:z_{1}\right]  $, $a_{2}\equiv\left[
x_{2}:y_{2}:z_{2}\right]  $ and $a_{3}\equiv\left[  x_{3}:y_{3}:z_{3}\right]
$ are non-null points, then the \textbf{quadrea }of $\left\{  a_{1}%
,a_{2},a_{3}\right\}  $ is the number%
\begin{equation}
\mathcal{A\equiv A}\left(  a_{1},a_{2},a_{3}\right)  \equiv-\frac{\left(
x_{1}y_{2}z_{3}-x_{1}y_{3}z_{2}+x_{2}y_{3}z_{1}-x_{3}y_{2}z_{1}+x_{3}%
y_{1}z_{2}-x_{2}y_{1}z_{3}\right)  ^{2}\allowbreak}{\left(  x_{1}^{2}%
+y_{1}^{2}-z_{1}^{2}\right)  \left(  x_{2}^{2}+y_{2}^{2}-z_{2}^{2}\right)
\left(  x_{3}^{2}+y_{3}^{2}-z_{3}^{2}\right)  }.\label{QuadreaFormula}%
\end{equation}
This is well-defined, and if the points are non-collinear and so form a
triangle $\overline{a_{1}a_{2}a_{3}},$ then we write $\mathcal{A\equiv
A}\left(  \overline{a_{1}a_{2}a_{3}}\right)  $. Note the minus sign. The
quadrea defined here in this essentially projective setting and the one
defined in the planar situation of \cite{Wild1} are analogous, with the latter
being $16 $ times the square of the area of the triangle.

\begin{theorem}
[Quadrea]Suppose that $a_{1},a_{2}$ and $a_{3}$ are distinct points with
quadrances $q_{2}\equiv q\left(  a_{1},a_{3}\right)  $ and $q_{3}\equiv
q\left(  a_{1},a_{2}\right)  $, spread $S_{1}\equiv S\left(  a_{1}a_{2}%
,a_{1}a_{3}\right)  $ and quadrea $\mathcal{A}$. Then%
\[
q_{2}q_{3}S_{1}=\mathcal{A}.
\]

\end{theorem}

\begin{proof}
Since the two quadrances $q_{2}$ and $q_{3}$ are defined, the points
$a_{1},a_{2}$ and $a_{3}$ are non-null. Using the notation of the Spread
formula,
\begin{align*}
q_{2}  & =-\frac{l_{2}^{2}+m_{2}^{2}-n_{2}^{2}}{\left(  x_{1}^{2}+y_{1}%
^{2}-z_{1}^{2}\right)  \left(  x_{3}^{2}+y_{3}^{2}-z_{3}^{2}\right)  }\\
q_{3}  & =-\frac{l_{3}^{2}+m_{3}^{2}-n_{3}^{2}}{\left(  x_{1}^{2}+y_{1}%
^{2}-z_{1}^{2}\right)  \left(  x_{2}^{2}+y_{2}^{2}-z_{2}^{2}\right)  }.
\end{align*}
Combining this with the Spread formula gives
\[
q_{2}q_{3}S_{1}=\mathcal{A}.%
%TCIMACRO{\TeXButton{Proof box}{{\hspace{.1in} \rule{0.5em}{0.5em}}}}%
%BeginExpansion
{\hspace{.1in} \rule{0.5em}{0.5em}}%
%EndExpansion
\]

\end{proof}

This argument works even if $q_{1}=q\left(  a_{2},a_{3}\right)  =0$, in which
case $S_{2}$ and $S_{3}$ are not defined. In case all three quadrances
$q_{1},q_{2}$ and $q_{3}$ are non-zero, and all three spreads $S_{1},S_{2}$
and $S_{3}$ are defined, symmetry gives
\[
q_{2}q_{3}S_{1}=q_{1}q_{3}S_{2}=q_{1}q_{2}S_{3}=\mathcal{A}\text{.}%
\]
It follows that we can rewrite the Spread law as
\[
\frac{S_{1}}{q_{1}}=\frac{S_{2}}{q_{2}}=\frac{S_{3}}{q_{3}}=\frac{\mathcal{A}%
}{q_{1}q_{2}q_{3}}.
\]

To compare, the Quadrea theorem in the planar situation has the form
\[
\mathcal{A}=4q_{2}q_{3}S_{1}.
\]

If $L_{1}\equiv\left(  l_{1}:m_{1}:n_{1}\right)  $, $L_{2}\equiv\left(
l_{2}:m_{2}:n_{2}\right)  $ and $L_{3}\equiv\left(  l_{3}:m_{3}:n_{3}\right)
$ are non-null lines, then the \textbf{quadreal} of $\left\{  L_{1}%
,L_{2},L_{3}\right\}  $ is the number
\begin{equation}
\mathcal{L}\equiv\mathcal{L}\left(  L_{1},L_{2},L_{3}\right)  \equiv
-\frac{\left(  l_{1}m_{2}n_{3}-l_{1}m_{3}n_{2}+l_{2}m_{3}n_{1}-l_{3}m_{2}%
n_{1}+l_{3}m_{1}n_{2}-l_{2}m_{1}n_{3}\right)  ^{2}\allowbreak}{\left(
l_{1}^{2}+m_{1}^{2}-n_{1}^{2}\right)  \left(  l_{2}^{2}+m_{2}^{2}-n_{2}%
^{2}\right)  \left(  l_{3}^{2}+m_{3}^{2}-n_{3}^{2}\right)  }%
.\label{QuadrealDef}%
\end{equation}
This is well-defined, and if the lines are non-concurrent and so form a
trilateral $\overline{L_{1}L_{2}L_{3}},$ then we write $\mathcal{L=L}\left(
\overline{L_{1}L_{2}L_{3}}\right)  $.

\begin{theorem}
[Quadreal]Suppose that $L_{1},L_{2}$ and $L_{3}$ are distinct lines with
spreads $S_{2}\equiv S\left(  L_{1},L_{3}\right)  $ and $S_{3}\equiv S\left(
L_{1},L_{2}\right)  $, quadrance $q_{1}\equiv q\left(  L_{1}L_{2},L_{1}%
L_{3}\right)  $ and quadreal $\mathcal{L}$. Then
\[
S_{2}S_{3}q_{1}=\mathcal{L}.
\]

\end{theorem}

\begin{proof}
This is dual to the Quadrea theorem.$%
%TCIMACRO{\TeXButton{Proof box}{{\hspace{.1in} \rule{0.5em}{0.5em}}}}%
%BeginExpansion
{\hspace{.1in} \rule{0.5em}{0.5em}}%
%EndExpansion
$
\end{proof}

It is convenient to also define the \textbf{quadreal }$\mathcal{L\equiv
L}\left(  \overline{a_{1}a_{2}a_{3}}\right)  $ \textbf{of a triangle} to be
the quadreal of its associated trilateral, and the \textbf{quadrea
}$\mathcal{A=A}\left(  \overline{L_{1}L_{2}L_{3}}\right)  $ \textbf{of a
trilateral} to be the quadrea of its associated triangle.

\begin{theorem}
[Quadrea quadreal product]Suppose that $a_{1},a_{2}$ and $a_{3}$ are distinct
points with quadrances $q_{1}\equiv q\left(  a_{2},a_{3}\right)  $,
$q_{2}\equiv q\left(  a_{1},a_{3}\right)  $ and $q_{3}\equiv q\left(
a_{1},a_{2}\right)  $, spreads $S_{1}\equiv S\left(  a_{1}a_{2},a_{1}%
a_{3}\right)  $, $S_{2}\equiv S\left(  a_{1}a_{2},a_{2}a_{3}\right)  $ and
$S_{3}\equiv S\left(  a_{1}a_{3},a_{2}a_{3}\right)  $, quadrea
$\mathcal{A\equiv A}\left(  a_{1},a_{2},a_{3}\right)  $ and quadreal
$\mathcal{L\equiv L}\left(  L_{1},L_{2},L_{3}\right)  $. Then%
\[
\mathcal{AL}=q_{1}q_{2}q_{3}S_{1}S_{2}S_{3}.
\]

\end{theorem}

\begin{proof}
Combine the Quadrea theorem and the Quadreal theorem.$%
%TCIMACRO{\TeXButton{Proof box}{{\hspace{.1in} \rule{0.5em}{0.5em}}}}%
%BeginExpansion
{\hspace{.1in} \rule{0.5em}{0.5em}}%
%EndExpansion
$
\end{proof}

\begin{example}
Consider the case when $a_{1}\equiv\left[  0:0:1\right]  $, $a_{2}%
\equiv\left[  x:y:1\right]  $ and $a_{3}\equiv\left[  x:-y:1\right]  $ as in
Example 12, with quadrances%
\[
q_{1}=\frac{4\left(  x^{2}-1\right)  y^{2}}{\left(  x^{2}+y^{2}-1\right)
^{2}}\qquad q_{2}=\frac{x^{2}+y^{2}}{x^{2}+y^{2}-1}\qquad q_{3}=\frac
{x^{2}+y^{2}}{x^{2}+y^{2}-1}%
\]
and spreads%
\[
S_{1}=\frac{4x^{2}y^{2}}{\left(  x^{2}+y^{2}\right)  ^{2}}\qquad S_{2}%
=\frac{\left(  x^{2}+y^{2}-1\right)  \allowbreak x^{2}}{\left(  x^{2}%
-1\right)  \left(  x^{2}+y^{2}\right)  }\qquad S_{3}=\frac{\left(  x^{2}%
+y^{2}-1\right)  \allowbreak x^{2}}{\left(  x^{2}-1\right)  \left(
x^{2}+y^{2}\right)  }.
\]
The quadrea and quadreal are
\[
\mathcal{A}=\frac{4x^{2}y^{2}}{\left(  x^{2}+y^{2}-1\right)  ^{2}}%
\qquad\text{\textit{and}}\qquad\mathcal{L}=\frac{4x^{4}y^{2}}{\left(
x^{2}-1\right)  \left(  x^{2}+y^{2}\right)  ^{2}}%
\]
and you may verify the Quadrea quadreal product theorem directly.$%
%TCIMACRO{\TeXButton{Math Diamond}{\hspace{.1in}\diamond}}%
%BeginExpansion
\hspace{.1in}\diamond
%EndExpansion
$
\end{example}

\subsection{The Cross law}

The Cross law is the analog of the \textit{Cosine law}. In planar rational
trigonometry the Cross law has the form
\begin{equation}
\left(  Q_{1}-Q_{2}-Q_{3}\right)  ^{2}=4Q_{2}Q_{3}\left(  1-s_{1}\right)
\label{Planar cross}%
\end{equation}
involving three quadrances and one spread, and it is hard to overstate the
importance of this most powerful formula. In the hyperbolic setting, the Cross
law is more complicated, but still very fundamental.

\begin{theorem}
[Cross law]Suppose that $a_{1},a_{2}$ and $a_{3}$ are distinct points with
quadrances $q_{1}\equiv q\left(  a_{2},a_{3}\right)  $, $q_{2}\equiv q\left(
a_{1},a_{3}\right)  $ and $q_{3}\equiv q\left(  a_{1},a_{2}\right)  $, and
spread $S_{1}\equiv S\left(  a_{1}a_{2},a_{1}a_{3}\right)  $. Then
\begin{equation}
\left(  q_{2}q_{3}S_{1}-q_{1}-q_{2}-q_{3}+2\right)  ^{2}=4\left(
1-q_{1}\right)  \left(  1-q_{2}\right)  \left(  1-q_{3}\right)
.\label{CrossLaw}%
\end{equation}

\end{theorem}

\begin{proof}
Suppose that $a_{1}\equiv\left[  x_{1}:y_{1}:z_{1}\right]  $, $a_{2}%
\equiv\left[  x_{2}:y_{2}:z_{2}\right]  $ and $a_{3}\equiv\left[  x_{3}%
:y_{3}:z_{3}\right]  $. The assumption that all three quadrances are defined
implies that the three points are non-null. Square both sides of the
polynomial identity%
\begin{align}
& -\left(  x_{1}y_{2}z_{3}-x_{1}y_{3}z_{2}+x_{2}y_{3}z_{1}-x_{3}y_{2}%
z_{1}+x_{3}y_{1}z_{2}-x_{2}y_{1}z_{3}\right)  ^{2}\nonumber\\
& +\left(  x_{1}^{2}+y_{1}^{2}-z_{1}^{2}\right)  \left(  x_{2}x_{3}+y_{2}%
y_{3}-z_{2}z_{3}\right)  ^{2}+\left(  x_{2}^{2}+y_{2}^{2}-z_{2}^{2}\right)
\left(  x_{1}x_{3}+y_{1}y_{3}-z_{1}z_{3}\right)  ^{2}\nonumber\\
& +\left(  x_{3}^{2}+y_{3}^{2}-z_{3}^{2}\right)  \left(  x_{1}x_{2}+y_{1}%
y_{2}-z_{1}z_{2}\right)  ^{2}\nonumber\\
& -\left(  x_{1}^{2}+y_{1}^{2}-z_{1}^{2}\right)  \left(  x_{2}^{2}+y_{2}%
^{2}-z_{2}^{2}\right)  \left(  x_{3}^{2}+y_{3}^{2}-z_{3}^{2}\right)
\label{QuadreaRelation}\\
& =2\left(  x_{2}x_{3}+y_{2}y_{3}-z_{2}z_{3}\right)  \left(  x_{1}x_{3}%
+y_{1}y_{3}-z_{1}z_{3}\right)  \left(  x_{1}x_{2}+y_{1}y_{2}-z_{1}z_{2}\right)
\nonumber
\end{align}
and divide by
\[
\left(  x_{1}^{2}+y_{1}^{2}-z_{1}^{2}\right)  ^{2}\left(  x_{2}^{2}+y_{2}%
^{2}-z_{2}^{2}\right)  ^{2}\left(  x_{3}^{2}+y_{3}^{2}-z_{3}^{2}\right)  ^{2}%
\]
to deduce that if $\mathcal{A}\equiv\mathcal{A}\left(  a_{1},a_{2}%
,a_{3}\right)  $ then
\[
\left(  \mathcal{A}+\left(  1-q_{1}\right)  +\left(  1-q_{2}\right)  +\left(
1-q_{3}\right)  -1\right)  ^{2}=4\left(  1-q_{1}\right)  \left(
1-q_{2}\right)  \left(  1-q_{3}\right)  .
\]
Rewrite this as
\begin{equation}
\left(  \mathcal{A}-q_{1}-q_{2}-q_{3}+2\right)  ^{2}=4\left(  1-q_{1}\right)
\left(  1-q_{2}\right)  \left(  1-q_{3}\right) \label{Crosslaw2}%
\end{equation}
and use the Quadrea theorem to replace $\mathcal{A}$ by $q_{2}q_{3}S_{1}$.$%
%TCIMACRO{\TeXButton{Proof box}{{\hspace{.1in} \rule{0.5em}{0.5em}}}}%
%BeginExpansion
{\hspace{.1in} \rule{0.5em}{0.5em}}%
%EndExpansion
$
\end{proof}

The Cross law gives a quadratic equation for the spreads of a triangle given
the quadrances. So the three quadrances of a triangle do\textit{\ }not quite
determine its spreads. As a quadratic equation in $\mathcal{A},$
(\ref{Crosslaw2}) can be rewritten using the Triple spread function as
\[
\mathcal{A}^{2}-2\left(  q_{1}+q_{2}+q_{3}-2\right)  \mathcal{A}=S\left(
q_{1},q_{2},q_{3}\right)  .
\]

Motivated by the Cross law, we define the\textbf{\ Cross function}
\begin{equation}
C\left(  \mathcal{A},q_{1},q_{2},q_{3}\right)  \equiv\left(  \mathcal{A}%
-q_{1}-q_{2}-q_{3}+2\right)  ^{2}-4\left(  1-q_{1}\right)  \left(
1-q_{2}\right)  \left(  1-q_{3}\right)  .\label{CrossFunction}%
\end{equation}

\begin{example}
Suppose that a triangle $\overline{a_{1}a_{2}a_{3}}$ has equal quadrances
$q_{1}=q_{2}=q_{3}\equiv-3.$ Then
\[
C\left(  \mathcal{A},-3,-3,-3\right)  \allowbreak=\left(  \mathcal{A}%
+11\right)  ^{2}-256=0
\]
has solutions $\mathcal{A}=-27$ and $\mathcal{A}=5,$ and from the Quadrea
theorem we deduce that
\[
S_{1}=S_{2}=S_{3}=-3\qquad\mathit{or}\qquad S_{1}=S_{2}=S_{3}=\frac{5}{9}.
\]
Two triangles $\overline{a_{1}a_{2}a_{3}}$ that have these quadrances and
spreads can be found over the respective fields $\mathbb{Q}\left(  \sqrt
{2},\sqrt{3}\right)  $ and $\mathbb{Q}\left(  \sqrt{2},\sqrt{3},\sqrt
{5}\right)  $, with
\[%
\begin{tabular}
[c]{lllll}%
$a_{1}\equiv\left[  \sqrt{2}:0:1\right]  $ &  & $a_{2}\equiv\left[
-1:\sqrt{3}:\sqrt{2}\right]  $ &  & $a_{3}\equiv\left[  -1:-\sqrt{3}:\sqrt
{2}\right]  $%
\end{tabular}
\]
and
\[%
\begin{tabular}
[c]{lllll}%
$a_{1}\equiv\left[  \sqrt{2}:0:\sqrt{5}\right]  $ &  & $a_{2}\equiv\left[
-1:\sqrt{3}:\sqrt{10}\right]  $ &  & $a_{3}\equiv\left[  -1:-\sqrt{3}%
:\sqrt{10}\right]  .%
%TCIMACRO{\TeXButton{Math Diamond}{\hspace{.1in}\diamond}}%
%BeginExpansion
\hspace{.1in}\diamond
%EndExpansion
$%
\end{tabular}
\]

\end{example}

It is an instructive exercise to verify that both of the classical hyperbolic
Cosine laws%
\begin{equation}
\cosh d_{1}=\cosh d_{2}\cosh d_{3}-\sinh d_{2}\sinh d_{3}\cos\theta
_{1}\label{Cos1}%
\end{equation}
and%
\begin{equation}
\cosh d_{1}=\frac{\cos\theta_{2}\cos\theta_{3}+\cos\theta_{1}}{\sin\theta
_{2}\sin\theta_{3}}\label{Cos2}%
\end{equation}
relating lengths $d_{1},d_{2},d_{3}$ and angles $\theta_{1},\theta_{2}%
,\theta_{3}$ in a classical hyperbolic triangle can be manipulated using
(\ref{QuadranceDistance}) and (\ref{SpreadAngle}) to obtain the Cross law.

\begin{theorem}
[Cross dual law]Suppose that $L_{1},L_{2}$ and $L_{3}$ are distinct lines with
spreads $S_{1}\equiv S\left(  L_{2},L_{3}\right)  $, $S_{2}\equiv S\left(
L_{1},L_{3}\right)  $ and $S_{3}\equiv S\left(  L_{1},L_{2}\right)  $, and
quadrance $q_{1}\equiv q\left(  L_{1}L_{2},L_{1}L_{3}\right)  $. Then
\[
\left(  S_{2}S_{3}q_{1}-S_{1}-S_{2}-S_{3}+2\right)  ^{2}=4\left(
1-S_{1}\right)  \left(  1-S_{2}\right)  \left(  1-S_{3}\right)  .
\]

\end{theorem}

\begin{proof}
This is dual to the Cross law.$%
%TCIMACRO{\TeXButton{Proof box}{{\hspace{.1in} \rule{0.5em}{0.5em}}}}%
%BeginExpansion
{\hspace{.1in} \rule{0.5em}{0.5em}}%
%EndExpansion
$
\end{proof}

This can also be restated in terms of the quadreal $\mathcal{L\equiv L}\left(
L_{1},L_{2},L_{3}\right)  $ as%
\[
\left(  \mathcal{L}-S_{1}-S_{2}-S_{3}+2\right)  ^{2}=4\left(  1-S_{1}\right)
\left(  1-S_{2}\right)  \left(  1-S_{3}\right)  .
\]

\subsection{Alternate formulations}

As in the Euclidean case, the most powerful of the trigonometric laws is the
Cross law
\begin{equation}
\left(  q_{2}q_{3}S_{1}-q_{1}-q_{2}-q_{3}+2\right)  ^{2}=4\left(
1-q_{1}\right)  \left(  1-q_{2}\right)  \left(  1-q_{3}\right)
.\label{CrossLawAlternate}%
\end{equation}
In the special case $S_{1}=0,$ we get%
\[
\left(  q_{1}+q_{2}+q_{3}-2\right)  ^{2}=4\left(  1-q_{1}\right)  \left(
1-q_{2}\right)  \left(  1-q_{3}\right)
\]
which is equivalent to the Triple quad formula%
\[
\left(  q_{1}+q_{2}+q_{3}\right)  ^{2}=2\left(  q_{1}^{2}+q_{2}^{2}+q_{3}%
^{2}\right)  +4q_{1}q_{2}q_{3}.
\]
If we rewrite (\ref{CrossLawAlternate}) in the form%
\begin{equation}
\left(  q_{1}-q_{2}-q_{3}+q_{2}q_{3}S_{1}\right)  ^{2}=4q_{2}q_{3}\left(
1-q_{1}\right)  \left(  1-S_{1}\right) \label{CrossLawAlt}%
\end{equation}
then in the special case $S_{1}=1$ we recover Pythagoras' theorem in the form%
\[
q_{1}=q_{2}+q_{3}-q_{2}q_{3}.
\]
Also (\ref{CrossLawAlt}) may be viewed as a deformation of the planar Cross
law (\ref{Planar cross}).

\subsection{Triple product and triple cross}

There are also secondary invariants associated to three points or three lines,
besides the quadrea and the quadreal. The \textbf{triple product }of the three
non-null points $a_{1},a_{2}$ and $a_{3}$ is the number%
\[
\mathcal{P\equiv P}\left(  a_{1},a_{2},a_{3}\right)  \equiv\frac{\left(
x_{2}x_{3}+y_{2}y_{3}-z_{2}z_{3}\right)  \left(  x_{1}x_{3}+y_{1}y_{3}%
-z_{1}z_{3}\right)  \left(  x_{1}x_{2}+y_{1}y_{2}-z_{1}z_{2}\right)  }{\left(
x_{1}^{2}+y_{1}^{2}-z_{1}^{2}\right)  \left(  x_{2}^{2}+y_{2}^{2}-z_{2}%
^{2}\right)  \left(  x_{3}^{2}+y_{3}^{2}-z_{3}^{2}\right)  }.
\]
The \textbf{triple cross }of the three non-null lines $L_{1},L_{2}$ and
$L_{3}$ is the number%
\[
\mathcal{C=C}\left(  L_{1},L_{2},L_{3}\right)  \equiv\frac{\left(  l_{2}%
l_{3}+m_{2}m_{3}-n_{2}n_{3}\right)  \left(  l_{1}l_{3}+m_{1}m_{3}-n_{1}%
n_{3}\right)  \left(  l_{1}l_{2}+m_{1}m_{2}-n_{1}n_{2}\right)  }{\left(
l_{1}^{2}+m_{1}^{2}-n_{1}^{2}\right)  \left(  l_{2}^{2}+m_{2}^{2}-n_{2}%
^{2}\right)  \left(  l_{3}^{2}+m_{3}^{2}-n_{3}^{2}\right)  }.
\]

\begin{theorem}
[Triple product relation]Suppose that the three points $a_{1},a_{2},a_{3}$
have quadrea $\mathcal{A}\equiv\mathcal{A}\left(  a_{1},a_{2},a_{3}\right)  $,
triple product $\mathcal{P\equiv P}\left(  a_{1},a_{2},a_{3}\right)  $, and
products $p_{1}\equiv p\left(  a_{2},a_{3}\right)  $, $p_{2}\equiv p\left(
a_{1},a_{3}\right)  $ and $p_{3}\equiv p\left(  a_{1},a_{2}\right)  $. Then
\[
\mathcal{A}+p_{1}+p_{2}+p_{3}-1=2\mathcal{P}.
\]

\end{theorem}

\begin{proof}
This is a consequence of the algebraic identity (\ref{QuadreaRelation}).$%
%TCIMACRO{\TeXButton{Proof box}{{\hspace{.1in} \rule{0.5em}{0.5em}}}}%
%BeginExpansion
{\hspace{.1in} \rule{0.5em}{0.5em}}%
%EndExpansion
$
\end{proof}

\begin{theorem}
[Triple cross relation]Suppose that the three lines $L_{1},L_{2},L_{3}$ have
quadreal $\mathcal{L}\equiv\mathcal{L}\left(  L_{1},L_{2},L_{3}\right)  $,
triple cross $\mathcal{C\equiv C}\left(  L_{1},L_{2},L_{3}\right)  $, and
crosses $C_{1}\equiv C\left(  L_{2},L_{3}\right)  $, $C_{2}\equiv C\left(
L_{1},L_{3}\right)  $ and $C_{3}\equiv C\left(  L_{1},L_{2}\right)  $. Then%
\[
\mathcal{L}+C_{1}+C_{2}+C_{3}-1=2\mathcal{C}.
\]

\end{theorem}

\begin{proof}
This is dual to the Triple product relation.$%
%TCIMACRO{\TeXButton{Proof box}{{\hspace{.1in} \rule{0.5em}{0.5em}}}}%
%BeginExpansion
{\hspace{.1in} \rule{0.5em}{0.5em}}%
%EndExpansion
$
\end{proof}

The Cross law in the form (\ref{Crosslaw2}) may be seen to be the result of
squaring the Triple product relation and replacing $\mathcal{P}^{2}$ with
$p_{1}p_{2}p_{3},$ and similarly for the Cross dual law.

\subsection{Midpoints and midlines}

A\textbf{\ midpoint }of the side $\overline{a_{1}a_{2}}$ is a point $m$ which
lies on $a_{1}a_{2}$ and satisfies
\begin{equation}
q\left(  a_{1},m\right)  =q\left(  a_{2},m\right)  .\label{MidDef}%
\end{equation}

\begin{theorem}
[Midpoints]The side $\overline{a_{1}a_{2}}$ has a midpoint precisely when
$p\left(  a_{1},a_{2}\right)  =1-q\left(  a_{1},a_{2}\right)  $ is a square.
In this case if $a_{1}\equiv\left[  x_{1}:y_{1}:z_{1}\right]  $ and
$a_{2}\equiv\left[  x_{2}:y_{2}:z_{2}\right]  $, then we can renormalize so
that
\[
x_{1}^{2}+y_{1}^{2}-z_{1}^{2}=x_{2}^{2}+y_{2}^{2}-z_{2}^{2}%
\]
and then there are exactly two midpoints, namely%
\[
m_{1}\equiv\left[  x_{1}+x_{2}:y_{1}+y_{2}:z_{1}+z_{2}\right]  \qquad
\mathrm{and}\qquad m_{2}\equiv\left[  x_{1}-x_{2}:y_{1}-y_{2}:z_{1}%
-z_{2}\right]  .
\]
Furthermore $m_{1}\bot~m_{2}.$
\end{theorem}

\begin{proof}
If $a_{1}\equiv\left[  x_{1}:y_{1}:z_{1}\right]  $, $a_{2}\equiv\left[
x_{2}:y_{2}:z_{2}\right]  $ and $m\equiv\left[  x:y:z\right]  $, then
(\ref{MidDef}) becomes
\[
1-\frac{\left(  x_{1}x+y_{1}y-z_{1}z\right)  ^{2}}{\left(  x_{1}^{2}+y_{1}%
^{2}-z_{1}^{2}\right)  \left(  x^{2}+y^{2}-z^{2}\right)  }=1-\frac{\left(
x_{2}x+y_{2}y-z_{2}z\right)  ^{2}}{\left(  x_{2}^{2}+y_{2}^{2}-z_{2}%
^{2}\right)  \left(  x^{2}+y^{2}-z^{2}\right)  }.
\]
In order for this to have a solution, $\left(  x_{1}^{2}+y_{1}^{2}-z_{1}%
^{2}\right)  \left(  x_{2}^{2}+y_{2}^{2}-z_{2}^{2}\right)  $ must be a square,
or equivalently%
\[
p\left(  a_{1},a_{2}\right)  \equiv\frac{\left(  x_{1}x_{2}+y_{1}y_{2}%
-z_{1}z_{2}\right)  ^{2}}{\left(  x_{1}^{2}+y_{1}^{2}-z_{1}^{2}\right)
\left(  x_{2}^{2}+y_{2}^{2}-z_{2}^{2}\right)  }=1-q\left(  a_{1},a_{2}\right)
\]
must be a square. In this case, after renormalization we may assume that
\[
x_{1}^{2}+y_{1}^{2}-z_{1}^{2}=x_{2}^{2}+y_{2}^{2}-z_{2}^{2}.
\]
Then $q\left(  a_{1},m\right)  =q\left(  a_{2},m\right)  $ is equivalent to
\begin{align*}
0  & =\left(  x_{1}x+y_{1}y-z_{1}z\right)  ^{2}-\left(  x_{2}x+y_{2}%
y-z_{2}z\right)  ^{2}\\
& =\allowbreak\left(  x\left(  x_{1}-x_{2}\right)  +y\left(  y_{1}%
-y_{2}\right)  -z\left(  z_{1}-z_{2}\right)  \right)  \left(  x\left(
x_{1}+x_{2}\right)  +y\left(  y_{1}+y_{2}\right)  -z\left(  z_{1}%
+z_{2}\right)  \right)  .\allowbreak
\end{align*}
The solutions to this which lie on $a_{1}a_{2}$ are
\[
m_{1}\equiv\left[  x_{1}+x_{2}:y_{1}+y_{2}:z_{1}+z_{2}\right]  \qquad
\mathrm{and}\qquad m_{2}\equiv\left[  x_{1}-x_{2}:y_{1}-y_{2}:z_{1}%
-z_{2}\right]  .
\]
These points are distinct since
\begin{align*}
\left(  x_{1}+x_{2}\right)  \left(  y_{1}-y_{2}\right)  -\left(  x_{1}%
-x_{2}\right)  \left(  y_{1}+y_{2}\right)   & =\allowbreak2\left(  x_{2}%
y_{1}-x_{1}y_{2}\right) \\
\left(  x_{1}+x_{2}\right)  \left(  z_{1}-z_{2}\right)  -\left(  x_{1}%
-x_{2}\right)  \left(  z_{1}+z_{2}\right)   & =\allowbreak2\left(  x_{2}%
z_{1}-x_{1}z_{2}\right) \\
\left(  y_{1}+y_{2}\right)  \left(  z_{1}-z_{2}\right)  -\left(  y_{1}%
-y_{2}\right)  \left(  z_{1}+z_{2}\right)   & =\allowbreak2\left(  y_{2}%
z_{1}-y_{1}z_{2}\right)
\end{align*}
and by assumption $a_{1}$ and $a_{2}$ are distinct, and $2\neq0,$ so at least
one of these expressions is non-zero.

Also $m_{1}\bot~m_{2}$ since
\begin{align*}
& \left(  x_{1}+x_{2}\right)  \left(  x_{1}-x_{2}\right)  +\left(  y_{1}%
+y_{2}\right)  \left(  y_{1}-y_{2}\right)  -\left(  z_{1}+z_{2}\right)
\left(  z_{1}-z_{2}\right) \\
& =\left(  x_{1}^{2}+y_{1}^{2}-z_{1}^{2}\right)  -\left(  x_{2}^{2}+y_{2}%
^{2}-z_{2}^{2}\right)  =0.%
%TCIMACRO{\TeXButton{Proof box}{{\hspace{.1in} \rule{0.5em}{0.5em}}}}%
%BeginExpansion
{\hspace{.1in} \rule{0.5em}{0.5em}}%
%EndExpansion
\end{align*}
A\textbf{\ midline }of the vertex $\overline{L_{1}L_{2}}$ is a line $M$ which
passes through $L_{1}L_{2}$ and satisfies%
\[
S\left(  L_{1},M\right)  =S\left(  L_{2},M\right)  .
\]

\end{proof}

\begin{theorem}
[Midlines]The vertex $\overline{L_{1}L_{2}}$ has a midline precisely when
$C\left(  L_{1},L_{2}\right)  =1-S\left(  L_{1},L_{2}\right)  $ is a square.
In this case if $L_{1}\equiv\left(  l_{1}:m_{1}:n_{1}\right)  $ and
$L_{2}\equiv\left(  l_{2}:m_{2}:n_{2}\right)  $, then we can renormalize so
that
\[
l_{1}^{2}+m_{1}^{2}-n_{1}^{2}=l_{2}^{2}+m_{2}^{2}-n_{2}^{2}%
\]
and then there are exactly two midlines, namely%
\[
M_{1}\equiv\left[  l_{1}+l_{2}:m_{1}+m_{2}:n_{1}+n_{2}\right]  \qquad
\mathrm{and}\qquad M_{2}\equiv\left[  l_{1}-l_{2}:m_{1}-m_{2}:n_{1}%
-n_{2}\right]  .
\]
Furthermore $M_{1}\bot~M_{2}$.
\end{theorem}

\begin{proof}
This is dual to the Midpoints theorem.$%
%TCIMACRO{\TeXButton{Proof box}{{\hspace{.1in} \rule{0.5em}{0.5em}}}}%
%BeginExpansion
{\hspace{.1in} \rule{0.5em}{0.5em}}%
%EndExpansion
$
\end{proof}

In classical hyperbolic geometry over the real numbers, there is always
exactly one midpoint, while there are two midlines, usually called
\textit{angle bisectors}. So the symmetry between midpoints and midlines is
missing. In a future paper, we will see that this symmetry gives rise to many
rich aspects of triangle geometry in the universal hyperbolic setting. This
symmetry is also shared by elliptic geometry.

\subsection{Mid formulas}

In classical trigonometry, both planar, elliptic and hyperbolic, there are
many formulas in which midpoints, midlines and the associated half-distances
and half-angles figure prominently. The analog of a relation such as
$\beta=2\alpha$ between angles $\alpha$ and $\beta$ in classical hyperbolic
geometry is a relation such as%
\begin{equation}
S=S_{2}\left(  R\right)  \equiv4R\left(  1-R\right) \label{DoubleQuad}%
\end{equation}
between the corresponding spreads $R$ and $S$. The same kind of relation holds
between quadrances $q$ and $r$ if $q=q\left(  a_{1},a_{2}\right)  $ and
$r=q\left(  a_{1},m\right)  =q\left(  a_{2},m\right)  $ where $m$ is a
midpoint of $\overline{a_{1}a_{2}},$ so this is also the analog between the
relation $d=2c$ between distances. Because there are two midpoints, or
midlines, we should not be surprised that (\ref{DoubleQuad}) is a quadratic equation.

Just as bisecting an angle generally involves moving to an extension field, so
finding $R,$ given $S$ in (\ref{DoubleQuad}), requires solving the quadratic
equation
\[
4\left(  R-\frac{1}{2}\right)  ^{2}=1-S.
\]
So such an $R$ exists precisely when $1-S$ is a \textit{square}, as we have
also seen in the Midline theorem. If we assume the existence of midpoints or
midlines, in other words that certain quadrances $q$ or certain spreads $S$
have the property that they are one minus a square, then a \textit{number of
the basic formulas of the subject have alternate formulations, sometimes
simpler}. We call such relations \textit{mid formulas}.

Recall the definition of the \textit{Triple spread function }in
(\ref{Triple Spread Fn}):%
\[
S\left(  a,b,c\right)  =\allowbreak\left(  a+b+c\right)  ^{2}-2\left(
a^{2}+b^{2}+c^{2}\right)  -4abc.
\]
The identity%
\[
S\left(  t-a,b,c\right)  -S\left(  a,t-b,c\right)  =\allowbreak4c\left(
b-a\right)  \left(  1-t\right)
\]
suggests that something special happens when $t=1,$ and in fact
\[
S\left(  1-a,b,c\right)  =S\left(  a,1-b,c\right)  =S\left(  a,b,1-c\right)
=S\left(  1-a,1-b,1-c\right)  .
\]
Also
\[
S\left(  1-a,1-b,1-c\right)  =S\left(  a,b,c\right)  +\left(  2a-1\right)
\left(  2b-1\right)  \left(  2c-1\right)  .
\]

The next results similarly rely on pleasant identities.

\begin{theorem}
[Triple quad mid]Suppose that
\[
q_{1}\equiv4p_{1}\left(  1-p_{1}\right)  \qquad q_{2}\equiv4p_{2}\left(
1-p_{2}\right)  \qquad\mathit{and}\qquad q_{3}\equiv4p_{3}\left(
1-p_{3}\right)  .
\]
Then
\[
S\left(  q_{1},q_{2},q_{3}\right)  =0
\]
precisely when either
\[
S\left(  p_{1},p_{2},p_{3}\right)  =0\qquad\mathit{or}\qquad S\left(
1-p_{1},1-p_{2},1-p_{3}\right)  =0.
\]

\end{theorem}

\begin{proof}
This follows from the identity%
\[
S\left(  q_{1},q_{2},q_{3}\right)  =-16S\left(  p_{1},p_{2},p_{3}\right)
S\left(  1-p_{1},1-p_{2},1-p_{3}\right)  .%
%TCIMACRO{\TeXButton{Proof box}{{\hspace{.1in} \rule{0.5em}{0.5em}}}}%
%BeginExpansion
{\hspace{.1in} \rule{0.5em}{0.5em}}%
%EndExpansion
\]

\end{proof}

\begin{theorem}
[Pythagoras mid]Suppose that
\[
q_{1}\equiv4p_{1}\left(  1-p_{1}\right)  \qquad q_{2}\equiv4p_{2}\left(
1-p_{2}\right)  \qquad\mathit{and}\qquad q_{3}\equiv4p_{3}\left(
1-p_{3}\right)  .
\]
Then
\[
q_{3}=q_{1}+q_{2}-q_{1}q_{2}%
\]
precisely when either
\[
p_{3}=p_{1}+p_{2}-p_{1}p_{2}\qquad\mathit{or}\qquad p_{3}=2p_{1}p_{2}%
-p_{1}-p_{2}+1.
\]

\end{theorem}

\begin{proof}
This follows from the identity%
\[
q_{3}-q_{1}-q_{2}+q_{1}q_{2}=-4\left(  p_{3}-p_{1}-p_{2}+2p_{1}p_{2}\right)
\left(  p_{3}-2p_{1}p_{2}+p_{1}+p_{2}-1\right)  .%
%TCIMACRO{\TeXButton{Proof box}{{\hspace{.1in} \rule{0.5em}{0.5em}}}}%
%BeginExpansion
{\hspace{.1in} \rule{0.5em}{0.5em}}%
%EndExpansion
\]

The next result shows that the existence of midpoints converts the Cross law
from a quadratic equation to two linear ones. Recall the definition of the
Cross function of (\ref{CrossFunction}):%
\[
C\left(  \mathcal{A},q_{1},q_{2},q_{3}\right)  \equiv\left(  \mathcal{A}%
-q_{1}-q_{2}-q_{3}+2\right)  ^{2}-4\left(  1-q_{1}\right)  \left(
1-q_{2}\right)  \left(  1-q_{3}\right)  .
\]

\end{proof}

\begin{theorem}
[Cross mid]Suppose that
\[
q_{1}\equiv4p_{1}\left(  1-p_{1}\right)  \qquad q_{2}\equiv4p_{2}\left(
1-p_{2}\right)  \qquad\mathit{and}\qquad q_{3}\equiv4p_{3}\left(
1-p_{3}\right)  .
\]
Then for any number $\mathcal{A}$,
\[
C\left(  \mathcal{A},q_{1},q_{2},q_{3}\right)  =0
\]
precisely when either
\[
\mathcal{A}=4S\left(  p_{1},p_{2},p_{3}\right)  \qquad\mathit{or}%
\qquad\mathcal{A}=4S\left(  1-p_{1},1-p_{2},1-p_{3}\right)  .
\]

\end{theorem}

\begin{proof}
This follows from the identity%
\begin{align*}
& \left(  \mathcal{A}-q_{1}-q_{2}-q_{3}+2\right)  ^{2}-4\left(  1-q_{1}%
\right)  \left(  1-q_{2}\right)  \left(  1-q_{3}\right) \\
& =\left(  \mathcal{A}-4S\left(  p_{1},p_{2},p_{3}\right)  \right)  \left(
\mathcal{A}-4S\left(  1-p_{1},1-p_{2},1-p_{3}\right)  \right)  .%
%TCIMACRO{\TeXButton{Proof box}{{\hspace{.1in} \rule{0.5em}{0.5em}}}}%
%BeginExpansion
{\hspace{.1in} \rule{0.5em}{0.5em}}%
%EndExpansion
\end{align*}

\end{proof}

\subsection{Quadrance and spread of a couple}

The \textbf{quadrance }$q\left(  \overline{aL}\right)  $\textbf{\ of a
non-null couple }$\overline{aL}$ is
\[
q\left(  \overline{aL}\right)  =1-q\left(  a,L^{\bot}\right)  .
\]
In case the couple is non-dual, this is by the Complementary quadrances
theorem equal to $q\left(  a,b\right)  $ where $b$ is the base point of
$\overline{aL}.$

The \textbf{spread }$S\left(  \overline{aL}\right)  $\textbf{\ of a couple
}$\overline{aL}$ is
\[
S\left(  \overline{aL}\right)  =1-S\left(  L,a^{\bot}\right)  .
\]
In case the couple is non-dual, this is by the Complementary spreads theorem
equal to $S\left(  L,R\right)  $ where $R$ is the parallel line of
$\overline{aL}.$ The Quadrance/spread theorem shows that
\[
q\left(  \overline{aL}\right)  =S\left(  \overline{aL}\right)  .
\]

\begin{theorem}
[Couple quadrance spread]Suppose that $\overline{aL}$ is a non-null couple
with $a\equiv\left[  x:y:z\right]  $ and $L\equiv\left(  l:m:n\right)  $. Then%
\[
q\left(  \overline{aL}\right)  =S\left(  \overline{aL}\right)  =\frac{\left(
lx+my-nz\right)  ^{2}\allowbreak}{\left(  x^{2}+y^{2}-z^{2}\right)  \left(
l^{2}+m^{2}-n^{2}\right)  }.
\]

\end{theorem}

\begin{proof}
By definition $q\left(  \overline{aL}\right)  $ is the product between $a$ and
$L^{\bot},$ and also $S\left(  \overline{aL}\right)  $ is the cross between $L
$ and $a^{\bot}.$ These are equal to the given expression.$%
%TCIMACRO{\TeXButton{Proof box}{{\hspace{.1in} \rule{0.5em}{0.5em}}}}%
%BeginExpansion
{\hspace{.1in} \rule{0.5em}{0.5em}}%
%EndExpansion
$
\end{proof}

\subsection{Spread polynomials}

The spread polynomials $S_{n}\left(  x\right)  $ are a remarkable family of
polynomials that replace the Chebyshev polynomials $T_{n}\left(  x\right)  $
of the first kind in rational trigonometry, and their importance extends also
to hyperbolic geometry. To motivate their introduction, we investigate the
effect of combining three equal quadrances. Recall that the Equal quadrances
spreads theorem states that if $S\left(  a,b,c\right)  =0$ and $a=b, $ then
$c=0$ or $c=S_{2}\left(  a\right)  \equiv4a\left(  1-a\right)  $.

\begin{theorem}
[Three equal quadrances]Suppose that $a,b$ and $c$ are numbers which satisfy
$S\left(  a,b,c\right)  =0 $ and $b=S_{2}\left(  a\right)  $. Then
\[
c=a\qquad\mathit{or}\qquad c=a\left(  3-4a\right)  ^{2}.
\]

\end{theorem}

\begin{proof}
The identity
\[
S\left(  a,4a\left(  1-a\right)  ,c\right)  =\allowbreak\left(  a-c\right)
\left(  c-9a+24a^{2}-16a^{3}\right)
\]
shows that either $c=a$ or
\[
c=9a-24a^{2}+16a^{3}=\allowbreak a\left(  3-4a\right)  ^{2}.%
%TCIMACRO{\TeXButton{Proof box}{{\hspace{.1in} \rule{0.5em}{0.5em}}}}%
%BeginExpansion
{\hspace{.1in} \rule{0.5em}{0.5em}}%
%EndExpansion
\]

\end{proof}

If we continue in this way we generate the \textbf{spread polynomials}
$S_{n}\left(  x\right)  ,$ which were defined in \cite{Wild1} recursively over
a general field, by
\begin{align}
S_{0}\left(  x\right)   & \equiv0\nonumber\\
S_{1}\left(  x\right)   & \equiv x\nonumber\\
S_{n}\left(  x\right)   & \equiv2\left(  1-2x\right)  S_{n-1}\left(  x\right)
-S_{n-2}\left(  x\right)  +2x.\label{Recursive}%
\end{align}

The next theorem is taken directly from \cite{Wild1}.

\begin{theorem}
[Recursive spreads]The spread polynomials $S_{n}\left(  x\right)  $ have the
property that for any number $x$ and any $n=1,2,3,\cdots,$%
\[
S\left(  x,S_{n-1}\left(  x\right)  ,S_{n}\left(  x\right)  \right)  =0.
\]

\end{theorem}

\begin{proof}
Fix a number $x$ and use induction on $n.$ For $n=1$ the statement follows
from the Equal quadrances spreads theorem. For a general $n\geq1$,
\[
S\left(  x,S_{n-1}\left(  x\right)  ,S_{n}\left(  x\right)  \right)  =0
\]
precisely when%
\begin{equation}
\left(  x+S_{n-1}\left(  x\right)  +S_{n}\left(  x\right)  \right)
^{2}=2\left(  x^{2}+S_{n-1}^{2}\left(  x\right)  +S_{n}^{2}\left(  x\right)
\right)  +4xS_{n-1}\left(  x\right)  S_{n}\left(  x\right) \label{Spread1}%
\end{equation}
while
\[
S\left(  x,S_{n}\left(  x\right)  ,S_{n+1}\left(  x\right)  \right)  =0
\]
precisely when
\begin{equation}
\left(  x+S_{n}\left(  x\right)  +S_{n+1}\left(  x\right)  \right)
^{2}=2\left(  x^{2}+S_{n}^{2}\left(  x\right)  +S_{n+1}^{2}\left(  x\right)
\right)  +4xS_{n}\left(  x\right)  S_{n+1}\left(  x\right)  .\label{Spread2}%
\end{equation}
Rearrange and factor the difference between equations (\ref{Spread1}) and
(\ref{Spread2}) to get%
\[
\left(  S_{n+1}\left(  x\right)  -S_{n-1}\left(  x\right)  \right)  \left(
S_{n+1}\left(  x\right)  -2\left(  1-2x\right)  S_{n}\left(  x\right)
+S_{n-1}\left(  x\right)  -2x\right)  =0
\]
Thus (\ref{Spread2}) follows from (\ref{Spread1}) if%
\[
S_{n+1}\left(  x\right)  =2\left(  1-2x\right)  S_{n}\left(  x\right)
-S_{n-1}\left(  x\right)  +2x.
\]
Since this agrees with the recursive definition of the spread polynomials, the
induction is complete.$%
%TCIMACRO{\TeXButton{Proof box}{{\hspace{.1in} \rule{0.5em}{0.5em}}}}%
%BeginExpansion
{\hspace{.1in} \rule{0.5em}{0.5em}}%
%EndExpansion
$
\end{proof}

The coefficients of the spread polynomials are integers, with the coefficient
of $x^{n}$ in $S_{n}\left(  x\right)  $ a power of four. It follows that the
degree of $S_{n}\left(  x\right)  $ is $n$ over any field not of
characteristic two. Here are the first few spread polynomials.%
\begin{align*}
S_{0}\left(  x\right)   &  =0\\
S_{1}\left(  x\right)   &  =x\\
S_{2}\left(  x\right)   &  =4x-4x^{2}=4x\left(  1-x\right) \\
S_{3}\left(  x\right)   &  =9x-24x^{2}+16x^{3}=x\left(  3-4x\right)  ^{2}\\
S_{4}\left(  x\right)   &  =16x-80x^{2}+128x^{3}-64x^{4}=16x\left(
1-x\right)  \left(  1-2x\right)  ^{2}\\
S_{5}\left(  x\right)   &  =25x-200x^{2}+560x^{3}-640x^{4}+256x^{5}=x\left(
5-20x+16x^{2}\right)  ^{2}\\
S_{6}\left(  x\right)   &  =4x\left(  1-x\right)  \left(  3-4x\right)
^{2}\left(  1-4x\right)  ^{2}\\
S_{7}\left(  x\right)   &  =x\left(  7-56x+112x^{2}-64x^{3}\right)  ^{2}.
\end{align*}
More details about these new polynomials can be found in \cite{GohWildberger}.
They play an important role in regular stars and polygons which we hope to
discuss in a future paper. They also have remarkable number-theoretic
properties, some of which are even more interesting than those of the
Chebyshev polynomials $T_{n}\left(  x.\right)  $

\section{Special triangles and trilaterals}

In this section we study a triangle $\overline{a_{1}a_{2}a_{3}}$ and its dual
trilateral $\overline{L_{1}L_{2}L_{3}},$ so that
\[%
\begin{tabular}
[c]{lllll}%
$L_{1}=a_{2}a_{3}$ &  & $L_{2}=a_{1}a_{3}$ &  & $L_{3}=a_{1}a_{2}$%
\end{tabular}
\]
and%
\[%
\begin{tabular}
[c]{lllll}%
$a_{1}=L_{2}L_{3}$ &  & $a_{2}=L_{1}L_{3}$ &  & $a_{3}=L_{1}L_{2}.$%
\end{tabular}
\]
The standard conventions throughout are that the quadrances are
\[
q_{1}\equiv q\left(  a_{2},a_{3}\right)  \qquad q_{2}\equiv q\left(
a_{1},a_{3}\right)  \qquad q_{3}\equiv q\left(  a_{1},a_{2}\right)
\]
and the spreads are
\[
S_{1}\equiv S\left(  L_{2},L_{3}\right)  \qquad S_{2}\equiv S\left(
L_{1},L_{3}\right)  \qquad S_{3}\equiv S\left(  L_{1},L_{2}\right)  .
\]

If the triangle $\overline{a_{1}a_{2}a_{3}}$ is nil, or equivalently the
trilateral $\overline{L_{1}L_{2}L_{3}}$ is null, then at least one of the
points $a_{1},a_{2},a_{3}$ is null, so that at least two of the quadrances
$q_{1},q_{2},q_{3}$ are undefined. If the triangle $\overline{a_{1}a_{2}a_{3}}
$ is null, or equivalently the trilateral $\overline{L_{1}L_{2}L_{3}}$ is nil,
then at least one of the lines $L_{1},L_{2},L_{3}$ is null, so that at least
two of the spreads $S_{1},S_{2},S_{3}$ are undefined. So the existence of the
three quadrances is equivalent to the triangle being non-nil, while the
existence of the three spreads is equivalent to the trilateral being non-nil.

We begin by studying \textit{right triangles}, in particular the phenomenon of
parallax, and Napier's Rules. Then formulas for \textit{isosceles} and
\textit{equilateral} \textit{triangles }are derived, including the
\textit{Equilateral relation}. Then we establish the main \textit{theorems of
proportion}, such as Menelaus' theorem and Ceva's theorem and their duals.

\subsection{Right triangles, parallax and Napier's rules}

\begin{theorem}
[Thales]Suppose that $\overline{a_{1}a_{2}a_{3}}$ is a non-null non-nil right
triangle with $S_{3}=1.$ Then
\[
S_{1}=\frac{q_{1}}{q_{3}}\qquad\mathit{and}\qquad S_{2}=\frac{q_{2}}{q_{3}}.
\]

\end{theorem}

\begin{proof}
This follows directly from the Spread law.$%
%TCIMACRO{\TeXButton{Proof box}{{\hspace{.1in} \rule{0.5em}{0.5em}}}}%
%BeginExpansion
{\hspace{.1in} \rule{0.5em}{0.5em}}%
%EndExpansion
$
\end{proof}

\begin{theorem}
[Thales' dual]Suppose that $\overline{L_{1}L_{2}L_{3}}$ is a non-null non-nil
right trilateral with $q_{3}=1.$ Then
\[
q_{1}=\frac{S_{1}}{S_{3}}\qquad\mathit{and}\qquad q_{2}=\frac{S_{2}}{S_{3}}.
\]

\end{theorem}

\begin{proof}
This is dual to Thales' theorem.$%
%TCIMACRO{\TeXButton{Proof box}{{\hspace{.1in} \rule{0.5em}{0.5em}}}}%
%BeginExpansion
{\hspace{.1in} \rule{0.5em}{0.5em}}%
%EndExpansion
$
\end{proof}

\begin{theorem}
[Right parallax]If a right triangle $\overline{a_{1}a_{2}a_{3}}$ has spreads
$S_{1}\equiv0$, $S_{2}\equiv S\neq0$ and $S_{3}\equiv1,$ then it will have
only one defined quadrance, namely%
\[
q_{1}=\frac{S-1}{S}.
\]

\end{theorem}

\begin{proof}
If $S_{1}=0$ then by the Zero spread theorem $a_{1}$ is a null point, so
$q_{2}$ and $q_{3}$ are undefined. Since $S_{2}$ and $S_{3}$ are by assumption
non-zero, $a_{2}$ and $a_{3}$ are non-null points. The Cross dual law%
\[
\left(  S_{2}S_{3}q_{1}-S_{1}-S_{2}-S_{3}+2\right)  ^{2}=4\left(
1-S_{1}\right)  \left(  1-S_{2}\right)  \left(  1-S_{3}\right)
\]
applies, and becomes%
\[
\left(  Sq_{1}+1-S\right)  ^{2}=0.
\]
Thus
\[
q_{1}=\frac{S-1}{S}.%
%TCIMACRO{\TeXButton{Proof box}{{\hspace{.1in} \rule{0.5em}{0.5em}}}}%
%BeginExpansion
{\hspace{.1in} \rule{0.5em}{0.5em}}%
%EndExpansion
\]

\end{proof}

Reciprocally, we may restate the conclusion as
\[
S=\frac{1}{1-q_{1}}.
\]

\begin{theorem}
[Right parallax dual]If a right trilateral $\overline{A_{1}A_{2}A_{3}}$ has
quadrances $q_{1}\equiv0$, $q_{2}\equiv q\neq0$ and $q_{3}\equiv1,$ then it
will have only one defined spread, namely%
\[
S_{1}=\frac{q-1}{q}.
\]

\end{theorem}

\begin{proof}
This is dual to the Right parallax theorem.$%
%TCIMACRO{\TeXButton{Proof box}{{\hspace{.1in} \rule{0.5em}{0.5em}}}}%
%BeginExpansion
{\hspace{.1in} \rule{0.5em}{0.5em}}%
%EndExpansion
$
\end{proof}

The following theorem is of considerable practical importance.

\begin{theorem}
[Napier's rules]Suppose that a right triangle $\overline{a_{1}a_{2}a_{3}}$ has
quadrances $q_{1},q_{2}$ and $q_{3},$ and spreads $S_{1},S_{2}$ and
$S_{3}\equiv1.$ Then any two of the quantities $S_{1},S_{2},q_{1},q_{2},q_{3}$
determine the other three, solely by the three \textbf{basic equations} from
Thales' theorem and Pythagoras' theorem:
\[
S_{1}=\frac{q_{1}}{q_{3}}\qquad S_{2}=\frac{q_{2}}{q_{3}}\qquad\mathit{and}%
\qquad q_{3}=q_{1}+q_{2}-q_{1}q_{2}.
\]

\end{theorem}

\begin{proof}
Given two of the quadrances, determine the third via Pythagoras' theorem
$q_{3}=q_{1}+q_{2}-q_{1}q_{2}$, and then Thales' theorem gives the spreads.

Given two spreads $S_{1}$ and $S_{2}$, use Pythagoras' theorem and the
relations $q_{1}=S_{1}q_{3}$ and $q_{2}=S_{2}q_{3}$ to obtain%
\[
1=S_{1}+S_{2}-S_{1}S_{2}q_{3}.
\]
Thus%
\begin{align*}
q_{3}  & =\frac{S_{1}+S_{2}-1}{S_{1}S_{2}}\\
q_{1}  & =S_{1}q_{3}=\frac{S_{1}+S_{2}-1}{S_{2}}\\
q_{2}  & =S_{2}q_{3}=\frac{S_{1}+S_{2}-1}{S_{1}}.
\end{align*}
Given a spread, say $S_{1},$ and one of the quadrances, then there are three
possibilities. If the given quadrance is $q_{3},$ then $q_{1}=S_{1}q_{3}$ and%
\begin{align*}
q_{2}  & =\frac{q_{3}-q_{1}}{1-q_{1}}=\frac{q_{3}\left(  1-S_{1}\right)
}{1-S_{1}q_{3}}\\
S_{2}  & =\frac{q_{2}}{q_{3}}=\frac{1-S_{1}}{1-S_{1}q_{3}}.
\end{align*}
If the given quadrance is $q_{1},$ then
\[
q_{3}=\frac{q_{1}}{S_{1}}%
\]
and%
\begin{align*}
q_{2}  & =\frac{q_{3}-q_{1}}{1-q_{1}}=\frac{q_{1}\left(  1-S_{1}\right)
}{S_{1}\left(  1-q_{1}\right)  }\\
S_{2}  & =\frac{q_{2}}{q_{3}}=\frac{1-S_{1}}{1-q_{1}}.
\end{align*}
If the given quadrance is $q_{2},$ then substitute $q_{1}=S_{1}q_{3}$ into
Pythagoras' theorem to get
\[
q_{3}=S_{1}q_{3}+q_{2}-S_{1}q_{2}q_{3}.
\]
So%
\begin{align*}
q_{3}  & =\frac{q_{2}}{1-S_{1}\left(  1-q_{2}\right)  }\\
q_{1}  & =\frac{S_{1}q_{2}}{1-S_{1}\left(  1-q_{2}\right)  }\\
S_{2}  & =\frac{q_{2}}{q_{3}}=1-S_{1}\left(  1-q_{2}\right)  .%
%TCIMACRO{\TeXButton{Proof box}{{\hspace{.1in} \rule{0.5em}{0.5em}}}}%
%BeginExpansion
{\hspace{.1in} \rule{0.5em}{0.5em}}%
%EndExpansion
\end{align*}

\end{proof}

The various equations derived in this proof are the analogs of Napier's rules,
and are \textit{fundamental for hyperbolic trigonometry}. It is perhaps best
to remember that all follow from the basic equations by elementary algebraic manipulations.

\begin{theorem}
[Napier's dual rules]Suppose that a right trilateral $\overline{A_{1}%
A_{2}A_{3}}$ has quadrances $q_{1},q_{2}$ and $q_{3}\equiv1,$ and spreads
$S_{1},S_{2}$ and $S_{3}.$ Then any two of the quantities $q_{1},q_{2}%
,S_{1},S_{2},S_{3}$ determine the other three, solely by the three
\textbf{basic equations} from Thales' dual theorem and Pythagoras' dual
theorem:
\[
q_{1}=\frac{S_{1}}{S_{3}}\qquad q_{2}=\frac{S_{2}}{S_{3}}\qquad\mathit{and}%
\qquad S_{3}=S_{1}+S_{2}-S_{1}S_{2}.
\]

\end{theorem}

\begin{proof}
This is dual to Napier's rules.$%
%TCIMACRO{\TeXButton{Proof box}{{\hspace{.1in} \rule{0.5em}{0.5em}}}}%
%BeginExpansion
{\hspace{.1in} \rule{0.5em}{0.5em}}%
%EndExpansion
$
\end{proof}

\begin{example}
Let's investigate the existence of a triangle $\overline{a_{1}a_{2}a_{3}}$
with spreads $S_{1}\equiv1/4,$ $S_{2}\equiv1/2$ and $S_{3}\equiv1,$ which
correspond respectively to angles of $\pi/6,$ $\pi/4$ and $\pi/2$ over
suitable extension fields of $%
%TCIMACRO{\U{211a} }%
%BeginExpansion
\mathbb{Q}
%EndExpansion
$. Since $\overline{a_{1}a_{2}a_{3}}$ is a right triangle, Napier's rules
hold, so that
\[
q_{3}=\frac{S_{1}+S_{2}-1}{S_{1}S_{2}}=-2\qquad q_{1}=S_{1}q_{3}=-\frac{1}%
{2}\qquad q_{2}=S_{2}q_{3}=-1.
\]
The quadrea is then $\mathcal{A}\left(  \overline{a_{1}a_{2}a_{3}}\right)
=1/2.$ To construct such a triangle let's assume that $a_{3}\equiv\left[
0:0:1\right]  $ and that $a_{1}\equiv\left[  x:0:1\right]  $ for some $x$ and
$a_{2}\equiv\left[  0:y:1\right]  $ for some $y.$ Then we must have
\[
q_{2}=\frac{x^{2}}{x^{2}-1}=-1
\]
so that $x=\pm1/\sqrt{2}$ and%
\[
q_{1}=\frac{y^{2}}{y^{2}-1}=-\frac{1}{2}%
\]
so that $y=\pm1/\sqrt{3}.$ Thus over $\mathbb{Q}\left(  \sqrt{2},\sqrt
{3}\right)  $ we can construct such a triangle.$%
%TCIMACRO{\TeXButton{Math Diamond}{\hspace{.1in}\diamond}}%
%BeginExpansion
\hspace{.1in}\diamond
%EndExpansion
$
\end{example}

\begin{example}
The same spreads as in the previous example can be achieved also over a finite
field. Since $5^{2}=2$ and $7^{2}=3$ in $\mathbb{F}_{23},$ over this field
there is a triangle $\overline{a_{1}a_{2}a_{3}}$ with spreads $S_{1}%
\equiv1/4=6,$ $S_{2}\equiv1/2=12$ and $S_{3}\equiv1$: choose
\[
a_{1}\equiv\left[  14:0:1\right]  \qquad a_{2}\equiv\left[  0:10:1\right]
\qquad a_{3}\equiv\left[  0:0:1\right]  .%
%TCIMACRO{\TeXButton{Math Diamond}{\hspace{.1in}\diamond}}%
%BeginExpansion
\hspace{.1in}\diamond
%EndExpansion
\]

\end{example}

\begin{example}
Suppose that $\overline{a_{1}a_{2}a_{3}}$ is a non-null right triangle with
$S_{3}\equiv1.$ Let $b$ denote the base point of the couple $\overline
{a_{3}\left(  a_{1}a_{2}\right)  }$. Define the quadrances
\[
r_{1}\equiv q\left(  a_{1},b\right)  \qquad r_{2}\equiv q\left(
a_{2},b\right)  \qquad\mathit{and}\qquad r_{3}\equiv q\left(  a_{3},b\right)
.
\]
By the Complimentary spreads theorem $S\left(  a_{3}b,a_{3}a_{2}\right)
=1-S_{2}=S_{1}.$ Then by Thales' theorem applied to the triangles
$\overline{a_{1}a_{2}a_{3}},\overline{a_{1}ba_{3}}$ and $\overline{ba_{2}%
a_{3}},$%
\[
S_{1}=\frac{q_{1}}{q_{3}}=\frac{r_{3}}{q_{2}}=\frac{r_{2}}{q_{1}}%
\]
so that
\[
r_{3}=\frac{q_{1}q_{2}}{q_{3}}\qquad r_{2}=\frac{q_{1}^{2}}{q_{3}}%
\qquad\mathit{and}\qquad r_{1}=\frac{q_{2}^{2}}{q_{3}}%
\]
the last by symmetry.$%
%TCIMACRO{\TeXButton{Math Diamond}{\hspace{.1in}\diamond}}%
%BeginExpansion
\hspace{.1in}\diamond
%EndExpansion
$
\end{example}

\subsection{Isosceles triangles}

The results of this section have obvious duals which we leave the reader to formulate.

A triangle $\overline{a_{1}a_{2}a_{3}}$ is \textbf{isosceles }precisely when
at least two of its quadrances are equal or at least two of its spreads are
equal. If all are defined, then the two conditions are equivalent.

\begin{theorem}
[Pons Asinorum]Suppose that a triangle $\overline{a_{1}a_{2}a_{3}}$ has
quadrances $q_{1},q_{2}$ and $q_{3},$ and spreads $S_{1},S_{2}$ and $S_{3}.$
Then $q_{1}=q_{2}$ precisely when $S_{1}=S_{2}.$
\end{theorem}

\begin{proof}
This follows from the Spread law%
\[
\frac{S_{1}}{q_{1}}=\frac{S_{2}}{q_{2}}=\frac{S_{3}}{q_{3}}.%
%TCIMACRO{\TeXButton{Proof box}{{\hspace{.1in} \rule{0.5em}{0.5em}}}}%
%BeginExpansion
{\hspace{.1in} \rule{0.5em}{0.5em}}%
%EndExpansion
\]

\end{proof}

\begin{theorem}
[Isosceles right]If $\overline{a_{1}a_{2}a_{3}}$ is a non-nil isosceles
triangle with two right spreads $S_{1}=S_{2}\equiv1,$ then also $q_{1}%
=q_{2}=1,$ and furthermore $S_{3}=q_{3}.$
\end{theorem}

\begin{proof}
If $S_{1}$ and $S_{2}$ are defined then $\overline{a_{1}a_{2}a_{3}}$ is a
non-null triangle, so by the Zero quadrances theorem all the quadrances are
non-zero. Thales' law shows that
\[
1=\frac{q_{2}}{q_{1}}%
\]
so $q_{1}=q_{2}.$ Pythagoras' theorem then gives%
\[
q_{1}=q_{2}+q_{3}-q_{2}q_{3}=q_{1}+q_{3}-q_{1}q_{3}%
\]
from which $q_{1}=q_{2}=1.$ Then the Spread law shows that $S_{3}=q_{3}.%
%TCIMACRO{\TeXButton{Proof box}{{\hspace{.1in} \rule{0.5em}{0.5em}}}}%
%BeginExpansion
{\hspace{.1in} \rule{0.5em}{0.5em}}%
%EndExpansion
$
\end{proof}

\begin{theorem}
[Isosceles mid]Suppose that an isosceles triangle $\overline{a_{1}a_{2}a_{3}}$
has quadrances $q_{1}=q_{2}\equiv q$ and $q_{3},$ and corresponding spreads
$S_{1}=S_{2}\equiv S$ and $S_{3},$ and that the couple $\overline{a_{3}\left(
a_{1}a_{2}\right)  }$ is non-dual, with base point $b.$ If
\[
r_{1}\equiv q\left(  a_{1},b\right)  \qquad r_{2}\equiv q\left(
a_{2},b\right)  \qquad\mathit{and}\qquad r_{3}\equiv q\left(  a_{3},b\right)
\]
then%
\[
r_{3}=Sq
\]
and%
\[
r_{1}=r_{2}=\frac{q\left(  1-S\right)  }{1-Sq}.
\]

\end{theorem}

\begin{proof}
Since the couple $\overline{a_{3}\left(  a_{1}a_{2}\right)  }$ is non-dual by
assumption, there is by the Base point theorem a unique point $b$ which lies
on $a_{1}a_{2}$ and for which $ba_{3}$ is perpendicular to $a_{1}a_{2}.$ By
Pythagoras' theorem%
\[
q_{2}=r_{1}+r_{3}-r_{1}r_{3}%
\]
while Thales' theorem shows that%
\[
S=\frac{r_{3}}{q}%
\]
so that
\[
r_{3}=Sq.
\]
If $r_{3}=Sq=1$ then $q_{2}=1,$ so that in $\overline{a_{1}a_{3}b}$ the Spread
law gives $S_{1}=1$. But then by symmetry $S_{2}=1,$ and so both $a_{1}$ and
$a_{2}$ are base points of $\overline{a_{3}\left(  a_{1}a_{2}\right)  },$
which is impossible. So $r_{3}\neq1$ and
\[
r_{1}=\frac{q-r_{3}}{1-r_{3}}=\frac{q\left(  1-S\right)  }{1-Sq}.
\]
By symmetry $r_{1}=r_{2}.%
%TCIMACRO{\TeXButton{Proof box}{{\hspace{.1in} \rule{0.5em}{0.5em}}}}%
%BeginExpansion
{\hspace{.1in} \rule{0.5em}{0.5em}}%
%EndExpansion
$
\end{proof}

\begin{theorem}
[Isosceles triangle]Suppose that an isosceles triangle $\overline{a_{1}%
a_{2}a_{3}}$ has quadrances $q_{1}=q_{2}\equiv q$ and $q_{3},$ and
corresponding spreads $S_{1}=S_{2}\equiv S$ and $S_{3},$ and that the couple
$\overline{a_{3}\left(  a_{1}a_{2}\right)  }$ is non-dual. Then
\[
q_{3}=\frac{4\left(  1-S\right)  q\left(  1-q\right)  }{\left(  1-Sq\right)
^{2}}\qquad\mathit{and}\qquad S_{3}=\frac{4S\left(  1-S\right)  \left(
1-q\right)  }{\left(  1-Sq\right)  ^{2}}.
\]
Furthermore $1-q_{3}$ is a square.
\end{theorem}

\begin{proof}
Using the notation of the Isosceles triangle mid theorem,
\[
q_{3}=S_{2}\left(  r_{1}\right)  =4r_{1}\left(  1-r_{1}\right)
\]
which is
\[
q_{3}=4\times\frac{q\left(  1-S\right)  }{1-Sq}\times\frac{\left(  1-q\right)
}{1-Sq}=\frac{4\left(  1-S\right)  q\left(  1-q\right)  }{\left(  1-Sq\right)
^{2}}.
\]
Use the Spread law%
\[
\frac{S_{3}}{q_{3}}=\frac{S}{q}%
\]
to get%
\[
S_{3}=\frac{4S\left(  1-S\right)  \left(  1-q\right)  }{\left(  1-Sq\right)
^{2}}.
\]

Then
\[
1-q_{3}=1-\frac{4q\left(  1-q\right)  \left(  1-S\right)  }{\left(
1-Sq\right)  ^{2}}=\allowbreak\frac{\left(  Sq-2q+1\right)  ^{2}}{\left(
1-Sq\right)  ^{2}}%
\]
so that $1-q_{3}$ is indeed a square. Alternatively, since $b\,$is a midpoint
of $\overline{a_{1}a_{2}},$ the Midpoints theorem also shows that $1-q_{3}$ is
a square.$%
%TCIMACRO{\TeXButton{Proof box}{{\hspace{.1in} \rule{0.5em}{0.5em}}}}%
%BeginExpansion
{\hspace{.1in} \rule{0.5em}{0.5em}}%
%EndExpansion
$
\end{proof}

\begin{example}
Suppose that $a_{1}\equiv\left[  a:0:1\right]  $, $a_{2}\equiv\left[
-a:0:1\right]  $ and $a_{3}\equiv\left[  0:b:1\right]  $, so the triangle
$\overline{a_{1}a_{2}a_{3}}$ has quadrances%
\[
q_{1}=q_{2}\equiv q=-\frac{a^{2}+b^{2}-a^{2}b^{2}}{\left(  1-a^{2}\right)
\left(  1-b^{2}\right)  }\allowbreak\qquad\mathit{and}\qquad q_{3}%
=-\frac{4a^{2}}{\left(  a^{2}-1\right)  ^{2}}%
\]
and spreads
\[
S_{1}=S_{2}\equiv S=\frac{b^{2}\left(  1-a^{2}\right)  }{a^{2}+b^{2}%
-a^{2}b^{2}}\qquad\mathit{and}\qquad S_{3}=\frac{4a^{2}b^{2}\left(
1-b^{2}\right)  }{\left(  a^{2}+b^{2}-a^{2}b^{2}\right)  ^{2}}.
\]
You may check that the relations in the Isosceles triangle theorem are
satisfied.$%
%TCIMACRO{\TeXButton{Math Diamond}{\hspace{.1in}\diamond}}%
%BeginExpansion
\hspace{.1in}\diamond
%EndExpansion
$
\end{example}

\begin{theorem}
[Isosceles parallax]If $\overline{a_{1}a_{2}a_{3}}$ is a non-null isosceles
triangle with $a_{1}$ a null point, $q_{1}\equiv q$ and $S_{2}=S_{3}\equiv S,$
then
\[
q=\frac{4\left(  S-1\right)  }{S^{2}}.
\]

\end{theorem}

\begin{proof}
The quadrance $q_{1}$ is non-zero since $\overline{a_{1}a_{2}a_{3}}$ is by
assumption non-null, and so $L_{1}\equiv a_{2}a_{3}$ is a non-null line. Since
$a_{1}$ is a null point, the couple $\overline{a_{1}L_{1}}$ is non-dual, and
so by the Altitude line and Base point theorems has an altitude line $N$ and a
base point $b.$ Apply the Right parallax theorem to both $\overline{a_{1}%
a_{2}b}$ and $\overline{a_{1}a_{3}b}$ to get
\[
q\left(  a_{1},b\right)  =q\left(  a_{2},p\right)  =\frac{S-1}{S}\equiv r.
\]
By the Equal quadrances theorem,
\[
q=4r\left(  1-r\right)  =4\left(  \frac{S-1}{S}\right)  \left(  \frac{1}%
{S}\right)  =\frac{4\left(  S-1\right)  }{S^{2}}.%
%TCIMACRO{\TeXButton{Proof box}{{\hspace{.1in} \rule{0.5em}{0.5em}}}}%
%BeginExpansion
{\hspace{.1in} \rule{0.5em}{0.5em}}%
%EndExpansion
\]

\end{proof}

\subsection{Equilateral triangles}

A triangle is \textbf{equilateral }precisely when either all its quadrances
are equal or all its spreads are equal. In case these are all defined, Pons
Asinorum implies that these two conditions are equivalent. The following
formula appeared in the Euclidean spherical case as Exercise 24.1 in
\cite{Wild1}.

\begin{theorem}
[Equilateral]Suppose that a triangle $\overline{a_{1}a_{2}a_{3}}$ is
equilateral with common quadrance $q_{1}=q_{2}=q_{3}\equiv q,$ and with common
spread $S_{1}=S_{2}=S_{3}\equiv S$. Then%
\begin{equation}
\left(  1-Sq\right)  ^{2}=4\left(  1-S\right)  \left(  1-q\right)
.\label{Equilateral Relation}%
\end{equation}

\end{theorem}

\begin{proof}
If $Sq=1$ then any point of the triangle is the dual of the opposite line, so
that all the quadrances are equal to $1,$ so both sides of the Equilateral
relation are zero. Otherwise the result is a consequence of the Isosceles
triangle theorem, which gives
\[
q_{3}=q=\frac{4q\left(  1-q\right)  \left(  1-S\right)  }{\left(  1-Sq\right)
^{2}}%
\]
so that
\[
\left(  1-Sq\right)  ^{2}=4\left(  1-S\right)  \left(  1-q\right)  .%
%TCIMACRO{\TeXButton{Proof box}{{\hspace{.1in} \rule{0.5em}{0.5em}}}}%
%BeginExpansion
{\hspace{.1in} \rule{0.5em}{0.5em}}%
%EndExpansion
\]

\end{proof}

The \textbf{equilateral relation} (\ref{Equilateral Relation}) is symmetric in
$S$ and $q$. Note that the point $\left[  -3,-3\right]  $ satisfies the
relation, and that the $q$-intercept and $S$-intercept are both $3/4.$ 

\begin{theorem}
[Equilateral mid]Suppose that
\[
S\equiv4R\left(  1-R\right)  \qquad\mathit{and}\qquad q\equiv4p\left(
1-p\right)  .
\]
Then
\[
\left(  1-Sq\right)  ^{2}=4\left(  1-S\right)  \left(  1-q\right)
\]
precisely when either%
\[
4Rp=1\qquad\mathit{or}\qquad4R\left(  1-p\right)  =1\qquad\mathit{or}%
\qquad4p\left(  1-R\right)  =1\qquad\mathit{or}\qquad4\left(  1-R\right)
\left(  1-p\right)  =1.
\]

\end{theorem}

\begin{proof}
This follows from the identity%
\begin{align*}
& \left(  1-Sq\right)  ^{2}-4\left(  1-S\right)  \left(  1-q\right) \\
& =\left(  4Rp-1\right)  \left(  4R\left(  1-p\right)  -1\right)  \left(
4p\left(  1-R\right)  -1\right)  \left(  4\left(  1-R\right)  \left(
1-p\right)  -1\right)  .%
%TCIMACRO{\TeXButton{Proof box}{{\hspace{.1in} \rule{0.5em}{0.5em}}}}%
%BeginExpansion
{\hspace{.1in} \rule{0.5em}{0.5em}}%
%EndExpansion
\end{align*}

\end{proof}

\subsection{Triangle proportions}

The following results are direct analogs of planar theorems in \cite{Wild1},
and the proofs are similar. The Triangle proportions theorem is self-dual,
while Menelaus' theorem and Ceva's theorem have separate duals.

\begin{theorem}
[Triangle proportions]Suppose that $\overline{a_{1}a_{2}a_{3}}$ is a triangle
with quadrances $q_{1}$, $q_{2}$ and $q_{3},$ spreads $S_{1},S_{2}$ and
$S_{3}$, and that $d$ is a non-null point lying on the line $a_{1}a_{2},$
distinct from $a_{1}$ and $a_{2}$. Define the quadrances $r_{1}\equiv q\left(
a_{1},d\right)  $ and $r_{2}\equiv q\left(  a_{2},d\right)  $, and the spreads
$R_{1}\equiv S\left(  a_{3}a_{1},a_{3}d\right)  $ and $R_{2}\equiv S\left(
a_{3}a_{2},a_{3}d\right)  $. Then%
\[
\frac{R_{1}}{R_{2}}=\frac{S_{1}}{S_{2}}\frac{r_{1}}{r_{2}}=\frac{q_{1}}{q_{2}%
}\frac{r_{1}}{r_{2}}.
\]

\end{theorem}

\begin{proof}
Define also $r_{3}\equiv q\left(  a_{3},d\right)  $. The assumptions imply
that all of the quadrances and spread defined are non-zero. In $\overline
{da_{2}a_{3}}$ use the Spread law to get%
\[
\frac{S_{2}}{r_{3}}=\frac{R_{2}}{r_{2}}.
\]
In $\overline{da_{1}a_{3}}$ use the Spread law to get%
\[
\frac{S_{1}}{r_{3}}=\frac{R_{1}}{s_{1}}.
\]
Thus%
\[
r_{3}=\frac{S_{2}r_{2}}{R_{2}}=\frac{S_{1}r_{1}}{R_{1}}%
\]
and rearrange to obtain%
\[
\frac{R_{1}}{R_{2}}=\frac{S_{1}}{S_{2}}\frac{r_{1}}{r_{2}}.
\]
Since
\[
\frac{S_{1}}{S_{2}}=\frac{q_{1}}{q_{2}}%
\]
this can be rewritten as%
\[
\frac{R_{1}}{R_{2}}=\frac{q_{1}}{q_{2}}\frac{r_{1}}{r_{2}}.%
%TCIMACRO{\TeXButton{Proof box}{{\hspace{.1in} \rule{0.5em}{0.5em}}}}%
%BeginExpansion
{\hspace{.1in} \rule{0.5em}{0.5em}}%
%EndExpansion
\]

\end{proof}

\begin{theorem}
[Menelaus]Suppose that $\overline{a_{1}a_{2}a_{3}}$ is a non-null triangle,
and that $L $ is a non-null line meeting $a_{2}a_{3}$, $a_{1}a_{3}$ and
$a_{1}a_{2}$ at the non-null points $d_{1}$, $d_{2}$ and $d_{3}$ respectively.
Define the quadrances%
\[%
\begin{array}
[c]{ccc}%
r_{1}\equiv q\left(  a_{2},d_{1}\right)  &  & t_{1}\equiv q\left(  d_{1}%
,a_{3}\right) \\
r_{2}\equiv q\left(  a_{3},d_{2}\right)  &  & t_{2}\equiv q\left(  d_{2}%
,a_{1}\right) \\
r_{3}\equiv q\left(  a_{1},d_{3}\right)  &  & t_{3}\equiv q\left(  d_{3}%
,a_{2}\right)  .
\end{array}
\]
Then%
\[
r_{1}r_{2}r_{3}=t_{1}t_{2}t_{3}.
\]

\end{theorem}

\begin{proof}
If one of the points $d_{1},d_{2},d_{3}$ is a point of the triangle, then both
sides of the required equation are zero and we are done. So suppose this is
not the case. Define the spreads between $L$ and the non-null lines
$a_{2}a_{3}$, $a_{1}a_{3}$ and $a_{1}a_{2}$ to be respectively $R_{1}$, $R_{2}
$ and $R_{3}$. These are all non-zero since $d_{1},d_{2}$ and $d_{3}$ are by
assumption non-null, while the assumption on $\overline{a_{1}a_{2}a_{3}}$
ensures that all the quadrances involved in the theorem are also non-zero. So
we can use the Spread law in the triangles $\overline{d_{1}d_{2}a_{3}}$,
$\overline{d_{2}d_{3}a_{1}}$ and $\overline{d_{3}d_{1}a_{2}}$ to get%
\[
\frac{R_{1}}{R_{2}}=\frac{r_{2}}{t_{1}}\qquad\frac{R_{2}}{R_{3}}=\frac{r_{3}%
}{t_{2}}\qquad\frac{R_{3}}{R_{1}}=\frac{r_{1}}{t_{3}}.
\]
Multiply these three equations to obtain
\[
r_{1}r_{2}r_{3}=t_{1}t_{2}t_{3}.%
%TCIMACRO{\TeXButton{Proof box}{{\hspace{.1in} \rule{0.5em}{0.5em}}}}%
%BeginExpansion
{\hspace{.1in} \rule{0.5em}{0.5em}}%
%EndExpansion
\]

\end{proof}

The following result in the planar case was called the \textit{Alternate
spreads theorem} in \cite{Wild1}.

\begin{theorem}
[Menelaus' dual]Suppose that $\overline{A_{1}A_{2}A_{3}}$ is a non-null
trilateral, and that $a$ is a non-null point joining $A_{2}A_{3}$, $A_{1}%
A_{3}$ and $A_{1}A_{2}$ to form the non-null lines $D_{1}$, $D_{2}$ and
$D_{3}$ respectively. Define the spreads%
\[%
\begin{array}
[c]{ccc}%
R_{1}\equiv S\left(  A_{2},D_{1}\right)  &  & T_{1}\equiv S\left(  D_{1}%
,A_{3}\right) \\
R_{2}\equiv S\left(  A_{3},D_{2}\right)  &  & T_{2}\equiv S\left(  D_{2}%
,A_{1}\right) \\
R_{3}\equiv S\left(  A_{1},D_{3}\right)  &  & T_{3}\equiv S\left(  D_{3}%
,A_{2}\right)  .
\end{array}
\]
Then%
\[
R_{1}R_{2}R_{3}=T_{1}T_{2}T_{3}.
\]

\end{theorem}

\begin{proof}
This is dual to Menelaus' theorem.$%
%TCIMACRO{\TeXButton{Proof box}{{\hspace{.1in} \rule{0.5em}{0.5em}}}}%
%BeginExpansion
{\hspace{.1in} \rule{0.5em}{0.5em}}%
%EndExpansion
$
\end{proof}

\begin{theorem}
[Ceva]Suppose that $\overline{a_{1}a_{2}a_{3}}$ is a non-nil triangle, and
that $a_{0}$ is a non-null point distinct from $a_{1},a_{2}$ and $a_{3}$, and
that the lines $a_{0}a_{1}$, $a_{0}a_{2}$ and $a_{0}a_{3}$ are non-null and
meet the lines $a_{2}a_{3}$, $a_{1}a_{3}$ and $a_{1}a_{2}$ respectively at the
points $d_{1}$, $d_{2}$ and $d_{3}$. Define the quadrances%
\[%
\begin{array}
[c]{ccc}%
r_{1}\equiv q\left(  a_{2},d_{1}\right)  &  & t_{1}\equiv q\left(  d_{1}%
,a_{3}\right) \\
r_{2}\equiv q\left(  a_{3},d_{2}\right)  &  & t_{2}\equiv q\left(  d_{2}%
,a_{1}\right) \\
r_{3}\equiv q\left(  a_{1},d_{3}\right)  &  & t_{3}\equiv q\left(  d_{3}%
,a_{2}\right)  .
\end{array}
\]
Then%
\[
r_{1}r_{2}r_{3}=t_{1}t_{2}t_{3}.
\]

\end{theorem}

\begin{proof}
If one of the lines of $\overline{a_{1}a_{2}a_{3}}$ is null, then both sides
of the required equation are zero. Otherwise we may assume that $\overline
{a_{1}a_{2}a_{3}}$ is non-null, with $S_{1},S_{2}$ and $S_{3}$ the usual
spreads. Define the spreads%
\[%
\begin{array}
[c]{ccc}%
R_{1}\equiv S\left(  a_{1}a_{2},a_{1}a_{0}\right)  &  & T_{1}\equiv S\left(
a_{1}a_{3},a_{1}a_{0}\right) \\
R_{2}\equiv S\left(  a_{2}a_{3},a_{2}a_{0}\right)  &  & T_{2}\equiv S\left(
a_{2}a_{1},a_{2}a_{0}\right) \\
R_{3}\equiv S\left(  a_{3}a_{1},a_{3}a_{0}\right)  &  & T_{3}\equiv S\left(
a_{3}a_{2},a_{3}a_{0}\right)  .
\end{array}
\]
as in the Alternate spreads theorem. Since $\overline{a_{1}a_{2}a_{3}}$ is
non-nil, these are all non-zero. Then use the Triangle proportions theorem
with the triangle $\overline{a_{1}a_{2}a_{3}}$ and the respective lines
$a_{1}d_{1}$, $a_{2}d_{2}$ and $a_{3}d_{3}$ to obtain%
\[
\frac{R_{1}}{P_{1}}=\frac{S_{2}}{S_{3}}\frac{r_{1}}{t_{1}}\qquad\frac{R_{2}%
}{P_{2}}=\frac{S_{3}}{S_{1}}\frac{r_{2}}{t_{2}}\qquad\frac{R_{3}}{P_{3}}%
=\frac{S_{1}}{S_{2}}\frac{r_{3}}{t_{3}}.
\]
Multiply these three equations and use the Alternate spreads theorem to get%
\[
\frac{R_{1}R_{2}R_{3}}{P_{1}P_{2}P_{3}}=\frac{r_{1}r_{2}r_{3}}{t_{1}t_{2}%
t_{3}}=1.%
%TCIMACRO{\TeXButton{Proof box}{{\hspace{.1in} \rule{0.5em}{0.5em}}}}%
%BeginExpansion
{\hspace{.1in} \rule{0.5em}{0.5em}}%
%EndExpansion
\]

\end{proof}

The dual of Ceva's theorem is not familiar in the Euclidean situation, where
it holds also.

\begin{theorem}
[Ceva's dual]Suppose that $\overline{A_{1}A_{2}A_{3}}$ is a non-nil trilateral
and that $A_{0}$ is a non-null line distinct from $A_{1},A_{2}$ and $A_{3}$,
and that the points $A_{0}A_{1}$, $A_{0}A_{2}$ and $A_{0}A_{3}$ are non-null
and join the points $A_{2}A_{3}$, $A_{1}A_{3}$ and $A_{1}A_{2}$ respectively
to get the lines $D_{1}$, $D_{2}$ and $D_{3}$. Define the spreads%
\[%
\begin{array}
[c]{ccc}%
R_{1}\equiv S\left(  A_{2},D_{1}\right)  &  & T_{1}\equiv S\left(  D_{1}%
,A_{3}\right) \\
R_{2}\equiv S\left(  A_{3},D_{2}\right)  &  & T_{2}\equiv S\left(  D_{2}%
,A_{1}\right) \\
R_{3}\equiv S\left(  A_{1},D_{3}\right)  &  & T_{3}\equiv S\left(  D_{3}%
,A_{2}\right)  .
\end{array}
\]
Then%
\[
R_{1}R_{2}R_{3}=T_{1}T_{2}T_{3}.
\]

\end{theorem}

\begin{proof}
This is dual to Ceva's theorem.$%
%TCIMACRO{\TeXButton{Proof box}{{\hspace{.1in} \rule{0.5em}{0.5em}}}}%
%BeginExpansion
{\hspace{.1in} \rule{0.5em}{0.5em}}%
%EndExpansion
$
\end{proof}

The converses of these theorems are not generally valid.

\section{Null trigonometry}

One of the significant differences between universal hyperbolic geometry and
classical hyperbolic geometry is that the rich theory of \textit{null
trigonometry} plays a larger role. What is presented here is just an
introduction to this fascinating subject, which has no parallel in Euclidean
geometry. Formulations of dual results are left to the reader.

\subsection{Singly nil triangles}

Recall that a triangle is \textit{singly nil} precisely when exactly one of
its points is null.

\begin{theorem}
[Nil cross law]Suppose the triangle $\overline{a_{1}a_{2}a_{3}}$ is singly
nil, with $a_{3}$ a null point, spreads $S_{1}\equiv S\left(  a_{1}a_{2}%
,a_{1}a_{3}\right)  $, $S_{2}\equiv S\left(  a_{1}a_{2},a_{2}a_{3}\right)  $
and $S_{3}\equiv S\left(  a_{1}a_{3},a_{2}a_{3}\right)  $, and quadrance
$q_{3}\equiv q\left(  a_{1},a_{2}\right)  $. Then $S_{3}=0$ and%
\[
q_{3}^{2}-2q_{3}\left(  \frac{1}{S_{1}}+\frac{1}{S_{2}}-\frac{1}{S_{1}S_{2}%
}\right)  +\left(  \frac{1}{S_{1}}-\frac{1}{S_{2}}\right)  ^{2}=0.
\]
Furthermore $\left(  1-2S_{1}\right)  \left(  1-2S_{2}\right)  $ is a square.
\end{theorem}

\begin{proof}
The Zero spread theorem shows that since $a_{3}$ is a null point, $S_{3}=0,$
while since $a_{1}$ and $a_{2}$ are non-null points, $S_{1}$ and $S_{2}$ are
non-zero. In this case the Cross dual law%
\[
\left(  S_{1}S_{2}q_{3}-\left(  S_{1}+S_{2}+S_{3}\right)  +2\right)
^{2}=4\left(  1-S_{1}\right)  \left(  1-S_{2}\right)  \left(  1-S_{3}\right)
\]
still applies, and after rearrangement%
\[
S_{1}^{2}S_{2}^{2}q_{3}^{2}-2q_{3}S_{2}S_{1}\left(  S_{1}+S_{2}-2\right)
+\left(  S_{2}-S_{1}\right)  ^{2}=0.
\]
This can be rewritten as
\[
q_{3}^{2}-2q_{3}\left(  \frac{1}{S_{1}}+\frac{1}{S_{2}}-\frac{1}{S_{1}S_{2}%
}\right)  +\left(  \frac{1}{S_{1}}-\frac{1}{S_{2}}\right)  ^{2}=0.
\]
After completing the square, it becomes%
\[
\left(  q_{3}-\left(  \frac{1}{S_{1}}+\frac{1}{S_{2}}-\frac{1}{S_{1}S_{2}%
}\right)  \right)  ^{2}=\frac{\left(  1-2S_{1}\right)  \left(  1-2S_{2}%
\right)  }{S_{1}^{2}S_{2}^{2}}%
\]
so $\left(  1-2S_{1}\right)  \left(  1-2S_{2}\right)  $ must be a square.$%
%TCIMACRO{\TeXButton{Proof box}{{\hspace{.1in} \rule{0.5em}{0.5em}}}}%
%BeginExpansion
{\hspace{.1in} \rule{0.5em}{0.5em}}%
%EndExpansion
$
\end{proof}

\begin{example}
In the special case $S_{2}\equiv1,$ the Nil cross law above becomes
\[
\left(  S_{1}q_{3}-S_{1}+1\right)  ^{2}=0
\]
so that%
\[
q_{3}=\frac{S_{1}-1}{S_{1}}%
\]
as also given by the Right parallax theorem.
%TCIMACRO{\TeXButton{diamond}{\hspace{.1in}{$\diamond$}}}%
%BeginExpansion
\hspace{.1in}{$\diamond$}%
%EndExpansion

\end{example}

\begin{example}
In the special case $S_{1}=S_{2}\equiv S,$ the Nil cross law becomes
\[
S^{2}q_{3}\left(  S^{2}q_{3}-4S+4\right)  =0
\]
so that $q_{3}=0$ or
\[
q_{3}=\frac{4\left(  S-1\right)  }{S^{2}}%
\]
as also given by the Isosceles parallax theorem.%
%TCIMACRO{\TeXButton{diamond}{\hspace{.1in}{$\diamond$}}}%
%BeginExpansion
\hspace{.1in}{$\diamond$}%
%EndExpansion

\end{example}

\begin{example}
Suppose that $a_{3}\equiv\left[  1:0:0\right]  $ and that $a_{1}\equiv\left[
x:0:1\right]  $ and $a_{2}\equiv\left[  0:y:1\right]  $. Then
\[
q_{3}=-\frac{\left(  x^{2}+y^{2}-x^{2}y^{2}\right)  }{\left(  x^{2}-1\right)
\left(  y^{2}-1\right)  }%
\]
while%
\[
S_{1}=\frac{\left(  1-x^{2}\right)  y^{2}}{x^{2}+y^{2}-x^{2}y^{2}}%
\qquad\mathit{and}\qquad S_{2}=\frac{\left(  1-x\right)  ^{2}y^{2}\left(
1-y^{2}\right)  }{x^{2}+y^{2}-x^{2}y^{2}}.
\]
You may check that these expressions satisfy the Nil cross law.%
%TCIMACRO{\TeXButton{diamond}{\hspace{.1in}{$\diamond$}}}%
%BeginExpansion
\hspace{.1in}{$\diamond$}%
%EndExpansion

\end{example}

\subsection{Doubly nil triangles}

\begin{theorem}
[Doubly nil triangle]Suppose the triangle $\overline{a_{1}a_{2}a_{3}}$ is
doubly nil, with $a_{1}$ and $a_{2}$ null points. Let $h$ be the quadrance of
the couple $\overline{a_{3}\left(  a_{1}a_{2}\right)  }.$ Then $S_{3}\equiv
S\left(  a_{3}a_{1},a_{3}a_{2}\right)  $ and $h$ satisfy the relation
\[
S_{3}=-\frac{4h}{\left(  1-h\right)  ^{2}}.
\]

\end{theorem}

\begin{proof}
The Parametrization of null points theorem shows that since $a_{1}$ and
$a_{2}$ are null points, we can write
\begin{align*}
a_{1}  & =\alpha\left(  t_{1}:u_{1}\right)  \equiv\left[  t_{1}^{2}-u_{1}%
^{2}:2t_{1}u_{1}:t_{1}^{2}+u_{1}^{2}\right] \\
a_{2}  & =\alpha\left(  t_{2}:u_{2}\right)  \equiv\left[  t_{2}^{2}-u_{2}%
^{2}:2t_{2}u_{2}:t_{2}^{2}+u_{2}^{2}\right]  .
\end{align*}
By the Join of null points theorem,%
\[
L_{3}\equiv a_{1}a_{2}=\left(  t_{1}t_{2}-u_{1}u_{2}:t_{1}u_{2}+t_{2}%
u_{1}:t_{1}t_{2}+u_{1}u_{2}\right)  .
\]
Now suppose that $a_{3}=\left[  x:y:z\right]  $ for some numbers $x,y$ and
$z.$ Then from the Couple quadrance spread theorem and the identity
(\ref{Null line product})
\[
h=\frac{\left(  \left(  t_{1}t_{2}-u_{1}u_{2}\right)  x+\left(  t_{1}%
u_{2}+t_{2}u_{1}\right)  y-\left(  t_{1}t_{2}+u_{1}u_{2}\right)  z\right)
^{2}}{\left(  x^{2}+y^{2}-z^{2}\right)  \left(  t_{1}u_{2}-t_{2}u_{1}\right)
^{2}}.
\]
Also the spread $S_{3}$ is computed to be%
\[
S_{3}=-\frac{4\left(  x\left(  t_{1}t_{2}-u_{1}u_{2}\right)  +y\left(
t_{1}u_{2}+t_{2}u_{1}\right)  -z\left(  t_{1}t_{2}+u_{1}u_{2}\right)  \right)
^{2}\allowbreak\left(  t_{1}u_{2}-t_{2}u_{1}\right)  ^{2}\left(  x^{2}%
+y^{2}-z^{2}\right)  }{\left(  x\left(  t_{1}^{2}-u_{1}^{2}\right)
+2yt_{1}u_{1}-z\left(  t_{1}^{2}+u_{1}^{2}\right)  \right)  ^{2}\left(
x\left(  t_{2}^{2}-u_{2}^{2}\right)  +2yt_{2}u_{2}-z\left(  t_{2}^{2}%
+u_{2}^{2}\right)  \right)  ^{2}\allowbreak}.
\]
It is then an algebraic identity that
\[
S_{3}=-\frac{4h}{\left(  1-h\right)  ^{2}}.%
%TCIMACRO{\TeXButton{Proof box}{{\hspace{.1in} \rule{0.5em}{0.5em}}}}%
%BeginExpansion
{\hspace{.1in} \rule{0.5em}{0.5em}}%
%EndExpansion
\]

\end{proof}

In a future paper we will see that this result has a natural interpretation in
terms of spreads subtended by certain circles.

\subsection{Triply nil triangles}

\begin{theorem}
[Triply nil quadreal]Suppose that $\overline{\alpha_{1}\alpha_{2}\alpha_{3}}$
is a triply nil triangle. Then
\[
\mathcal{L}\left(  \overline{\alpha_{1}\alpha_{2}\alpha_{3}}\right)  =-4.
\]

\end{theorem}

\begin{proof}
Suppose that
\begin{align*}
\alpha_{1}  & =\alpha\left(  t_{1}:u_{1}\right)  \equiv\left[  t_{1}^{2}%
-u_{1}^{2}:2t_{1}u_{1}:t_{1}^{2}+u_{1}^{2}\right] \\
\alpha_{2}  & =\alpha\left(  t_{2}:u_{2}\right)  \equiv\left[  t_{2}^{2}%
-u_{2}^{2}:2t_{2}u_{2}:t_{2}^{2}+u_{2}^{2}\right] \\
\alpha_{3}  & =\alpha\left(  t_{3}:u_{3}\right)  \equiv\left[  t_{3}^{2}%
-u_{3}^{2}:2t_{3}u_{3}:t_{3}^{2}+u_{3}^{2}\right]  .
\end{align*}
Then by the Join of null points theorem, the lines of $\overline{\alpha
_{1}\alpha_{2}\alpha_{3}}$ are%
\begin{align*}
L_{1}  & \equiv\alpha_{2}\alpha_{3}=L\left(  t_{2}:u_{2}|t_{3}:u_{3}\right)
\equiv\left(  t_{2}t_{3}-u_{2}u_{3}:t_{2}u_{3}+t_{3}u_{2}:t_{2}t_{3}%
+u_{2}u_{3}\right) \\
L_{2}  & \equiv\alpha_{1}\alpha_{3}=L\left(  t_{1}:u_{1}|t_{3}:u_{3}\right)
\equiv\left(  t_{1}t_{3}-u_{1}u_{3}:t_{1}u_{3}+t_{3}u_{1}:t_{1}t_{3}%
+u_{1}u_{3}\right) \\
L_{3}  & \equiv\alpha_{1}\alpha_{2}=L\left(  t_{1}:u_{1}|t_{2}:u_{2}\right)
\equiv\left(  t_{1}t_{2}-u_{1}u_{2}:t_{1}u_{2}+t_{2}u_{1}:t_{1}t_{2}%
+u_{1}u_{2}\right)  .
\end{align*}

A computer calculation shows that substituting these values into the
expression (\ref{QuadrealDef}) for $\mathcal{L}\left(  L_{1},L_{2}%
,L_{3}\right)  $ gives identically the number $-4.%
%TCIMACRO{\TeXButton{Proof box}{{\hspace{.1in} \rule{0.5em}{0.5em}}}}%
%BeginExpansion
{\hspace{.1in} \rule{0.5em}{0.5em}}%
%EndExpansion
$
\end{proof}

\begin{theorem}
[Triply nil balance]Suppose that $\overline{\alpha_{1}\alpha_{2}\alpha_{3}}$
is a triply nil triangle, and that $d$ is any point lying on $\alpha_{1}%
\alpha_{2}$ distinct from $\alpha_{1}$ and $\alpha_{2}.$ If $b_{1}$ is the
base of the couple $\overline{d\left(  \alpha_{1}\alpha_{3}\right)  },$ and
$b_{2}$ is the base of the couple $\overline{d\left(  \alpha_{2}\alpha
_{3}\right)  },$ then
\[
q\left(  d,b_{1}\right)  q\left(  d,b_{2}\right)  =1.
\]
Furthermore $db_{1}$ is perpendicular to $db_{2}.$
\end{theorem}

\begin{proof}
From the Parametrizing a line theorem we can find a proportion $r:s$ so that%
\begin{equation}
d=\left[  r\left(  t_{1}^{2}-u_{1}^{2}\right)  +s\left(  t_{2}^{2}-u_{2}%
^{2}\right)  :2rt_{1}u_{1}+2st_{2}u_{2}:r\left(  t_{1}^{2}+u_{1}^{2}\right)
+s\left(  t_{2}^{2}+u_{2}^{2}\right)  \right]  .\label{Param1}%
\end{equation}
Since $d$ is distinct from $\alpha_{1}$ and $\alpha_{2},$ both $r$ and $s$ are
non-zero. The Join of null points theorem gives%
\begin{equation}
\alpha_{2}\alpha_{3}=\left(  t_{2}t_{3}-u_{2}u_{3},t_{2}u_{3}+t_{3}u_{2}%
,t_{2}t_{3}+u_{2}u_{3}\right)  .\label{Param2}%
\end{equation}
The quadrance $q\left(  d,b_{1}\right)  $ by the Couple quadrance spread
theorem can be calculated from these two expressions, and with the aid of a
computer we get%
\[
q\left(  d,b_{1}\right)  =-\frac{r\left(  t_{1}u_{3}-u_{1}t_{3}\right)
^{2}\allowbreak}{s\left(  t_{2}u_{3}-t_{3}u_{2}\right)  ^{2}}.
\]
Similarly
\[
q\left(  d,b_{2}\right)  =-\frac{s\left(  t_{2}u_{3}-t_{3}u_{2}\right)  ^{2}%
}{r\left(  t_{1}u_{3}-u_{1}t_{3}\right)  ^{2}\allowbreak}.
\]
So
\[
q\left(  d,b_{1}\right)  q\left(  d,b_{2}\right)  =1.
\]

\end{proof}

Applying the hyperbolic cross function $J$ to (\ref{Param1}) and
(\ref{Param2}) we get an expression for the altitude line $db_{1}$ from $d$ to
$\alpha_{2}\alpha_{3},$ and similarly an expression for $db_{2}.$ Then
applying the formula for the spread between two lines, a computer calculation
shows that
\[
S\left(  db_{1},db_{2}\right)  =1.%
%TCIMACRO{\TeXButton{Proof box}{{\hspace{.1in} \rule{0.5em}{0.5em}}}}%
%BeginExpansion
{\hspace{.1in} \rule{0.5em}{0.5em}}%
%EndExpansion
\]

As a consequence, Pythagoras' theorem gives%
\[
q\left(  b_{1},b_{2}\right)  =q\left(  d,b_{1}\right)  +q\left(
d,b_{2}\right)  -1.
\]

The next theorem is but a brief introduction to a wealth of intricate
relations that exist in a triply nil triangle. It suggests that there is a
rich family of numbers that play a universal role in hyperbolic geometry, and
raises the question of cataloguing such numbers.

\begin{theorem}
[Triply nil orthocenter]Suppose that we work over a field with characteristic
neither two, three or five, and that $\overline{\alpha_{1}\alpha_{2}\alpha
_{3}}$ is a triply nil triangle, so that each of $\alpha_{1},\alpha_{2}$ and
$\alpha_{3}$ is a null point. Then the three sides of $\overline{\alpha
_{1}\alpha_{2}\alpha_{3}}$ are non-null, and the three altitudes of the three
couples of $\overline{\alpha_{1}\alpha_{2}\alpha_{3}}$ meet in a point $o$
called the \textbf{orthocenter }of $\overline{\alpha_{1}\alpha_{2}\alpha_{3}}%
$. If the bases of the altitudes are respectively $b_{1},b_{2}$ and $b_{3},$
then $\overline{b_{1}b_{2}b_{3}}$ is an equilateral triangle with common
quadrance $q=-5/4$ and common spread $S=16/25.$ The orthocenter of
$\overline{b_{1}b_{2}b_{3}}$ is also $o,$ and
\[
q\left(  o,b_{1}\right)  =q\left(  o,b_{2}\right)  =q\left(  o,b_{3}\right)
=-1/3.
\]

\end{theorem}

\begin{proof}
Suppose that $\alpha_{1}=\alpha\left(  t_{1}:u_{1}\right)  $, $\alpha
_{2}=\alpha\left(  t_{2}:u_{2}\right)  $ and $\alpha_{3}=\alpha\left(
t_{3}:u_{3}\right)  $. Then from the Join of null points theorem
\begin{align*}
\alpha_{1}\alpha_{2}  & =\left(  t_{1}t_{2}-u_{1}u_{2}:t_{1}u_{2}+t_{2}%
u_{1}:t_{1}t_{2}+u_{1}u_{2}\right) \\
\alpha_{1}\alpha_{3}  & =\left(  t_{1}t_{3}-u_{1}u_{3}:t_{1}u_{3}+u_{1}%
t_{3}:t_{1}t_{3}+u_{1}u_{3}\right) \\
\alpha_{2}\alpha_{3}  & =\left(  t_{2}t_{3}-u_{2}u_{3}:t_{2}u_{3}+u_{2}%
t_{3}:t_{2}t_{3}+u_{2}u_{3}\right)  .
\end{align*}
Use the Join of points and Meet of lines theorems to determine that the
altitudes of $\overline{\alpha_{1}\alpha_{2}\alpha_{3}}$ meet at a point $o,$
and to give explicit expressions for the points $b_{1},b_{2}$ and $b_{3}$.
While the exact formulas are somewhat lengthy to write down, the expressions
for quadrance and spread applied to them, together with some pleasant computer
simplifications, give the results.$%
%TCIMACRO{\TeXButton{Proof box}{{\hspace{.1in} \rule{0.5em}{0.5em}}}}%
%BeginExpansion
{\hspace{.1in} \rule{0.5em}{0.5em}}%
%EndExpansion
$
\end{proof}

Note that $q=-5/4$ and $S=16/25$ satisfy the Equilateral relation $\left(
1-Sq\right)  ^{2}=4\left(  1-S\right)  \left(  1-q\right)  $.

\subsection{Triangle thinness}

One of the defining aspects of classical hyperbolic geometry is
\textit{thinness of triangles}. Here are two results that give universal
approaches to this phenomenon.

\begin{theorem}
[Triply nil Cevian thinness]Suppose that $\overline{\alpha_{1}\alpha_{2}%
\alpha_{3}}$ is a triply nil triangle, and that $a$ is a point distinct from
$\alpha_{1},\alpha_{2}$ and $\alpha_{3}.$ Define the cevian points
$c_{1}\equiv\left(  a\alpha_{1}\right)  \left(  \alpha_{2}\alpha_{3}\right)
$, $c_{2}\equiv\left(  a\alpha_{2}\right)  \left(  \alpha_{1}\alpha
_{3}\right)  $ and $c_{3}\equiv\left(  a\alpha_{3}\right)  \left(  \alpha
_{1}\alpha_{2}\right)  .$ Then
\[
\mathcal{A}\left(  c_{1},c_{2},c_{3}\right)  =1.
\]

\end{theorem}

\begin{proof}
Using the notation of the proofs of the previous theorems, suppose that
$a\equiv\left[  x:y:z\right]  $ is an arbitrary point distinct from
$\alpha_{1},\alpha_{2}$ and $\alpha_{3}$. Then in terms of $t_{1},u_{1}%
,t_{2},u_{2},t_{3},u_{3}$ and $x,y$ and $z,$ we may use the Joint of points
and Meet of lines theorems to find expressions for $c_{1},c_{2}$ and $c_{3},$
and then use a computer to evaluate the quadrea $\mathcal{A}\left(
c_{1},c_{2},c_{3}\right)  $. In terms of the expression (\ref{QuadreaFormula})
the numerator becomes the square of $\allowbreak$%
\begin{align*}
& 8\left(  t_{1}u_{2}-t_{2}u_{1}\right)  \left(  t_{2}u_{3}-t_{3}u_{2}\right)
\left(  t_{3}u_{1}-t_{1}u_{3}\right)  \allowbreak\left(  x\left(  u_{2}%
u_{3}-t_{2}t_{3}\right)  -y\left(  t_{2}u_{3}+yt_{3}u_{2}\right)  +z\left(
t_{2}t_{3}+u_{2}u_{3}\right)  \right)  \allowbreak\\
& \times\left(  x\left(  t_{1}t_{3}-u_{1}u_{3}\right)  +y\left(  t_{1}%
u_{3}+u_{1}t_{3}\right)  -z\left(  t_{1}t_{3}+u_{1}u_{3}\right)  \right)
\allowbreak\left(  x\left(  t_{1}t_{2}-u_{1}u_{2}\right)  +y\left(  t_{1}%
u_{2}+t_{2}u_{1}\right)  -z\left(  t_{1}t_{2}+u_{1}u_{2}\right)  \right)
\end{align*}
while the denominator has three factors which are all of the form%
\begin{align*}
& 4\left(  t_{3}u_{1}-t_{1}u_{3}\right)  \left(  t_{1}u_{2}-t_{2}u_{1}\right)
\\
& \times\left(  x\left(  t_{1}t_{3}-u_{1}u_{3}\right)  +y\left(  t_{1}%
u_{3}+u_{1}t_{3}\right)  -z\left(  t_{1}t_{3}+u_{1}u_{3}\right)  \right)
\allowbreak\left(  x\left(  t_{1}t_{2}-u_{1}u_{2}\right)  +y\left(  t_{1}%
u_{2}+t_{2}u_{1}\right)  -z\left(  t_{1}t_{2}+u_{1}u_{2}\right)  \right)  .
\end{align*}
The result then follows from pleasant cancellation.$%
%TCIMACRO{\TeXButton{Proof box}{{\hspace{.1in} \rule{0.5em}{0.5em}}}}%
%BeginExpansion
{\hspace{.1in} \rule{0.5em}{0.5em}}%
%EndExpansion
$
\end{proof}

\begin{theorem}
[Triply nil altitude thinness]Suppose that $\overline{\alpha_{1}\alpha
_{2}\alpha_{3}}$ is a triply nil triangle and that $a$ is a point distinct
from the duals of the lines. If the altitudes to the lines of this triangle
from $a$ meet the lines respectively at base points $b_{1},b_{2}$ and $b_{3}$,
then%
\[
\mathcal{A}\left(  b_{1},b_{2},b_{3}\right)  =1.
\]

\end{theorem}

\begin{proof}
Using the notation of the proof of the previous theorem, the altitude line
from $a$ to $\alpha_{1}\alpha_{2}$ is
\[
N_{3}\equiv\left[  -y\left(  t_{1}t_{2}+u_{1}u_{2}\right)  +z\left(
t_{1}u_{2}+t_{2}u_{1}\right)  :x\left(  t_{1}t_{2}+u_{1}u_{2}\right)
-z\left(  t_{1}t_{2}-u_{1}u_{2}\right)  :x\left(  t_{1}u_{2}+t_{2}%
u_{1}\right)  -y\left(  t_{1}t_{2}-u_{1}u_{2}\right)  \right]
\]
and its meet with $\alpha_{1}\alpha_{2}$ is the base point%
\[
b_{3}=\left[
\begin{array}
[c]{c}%
x\left(  t_{1}^{2}-u_{1}^{2}\right)  \left(  t_{2}^{2}-u_{2}^{2}\right)
+y\left(  t_{1}u_{2}+t_{2}u_{1}\right)  \left(  t_{1}t_{2}-u_{1}u_{2}\right)
-z\left(  t_{1}^{2}t_{2}^{2}-u_{1}^{2}u_{2}^{2}\right) \\
:x\left(  t_{1}u_{2}+t_{2}u_{1}\right)  \left(  t_{1}t_{2}-u_{1}u_{2}\right)
+4yt_{1}t_{2}u_{1}u_{2}-z\left(  t_{1}t_{2}+u_{1}u_{2}\right)  \left(
t_{1}u_{2}+t_{2}u_{1}\right) \\
:x\left(  t_{1}^{2}t_{2}^{2}-u_{1}^{2}u_{2}^{2}\right)  +y\left(  t_{1}%
t_{2}+u_{1}u_{2}\right)  \left(  t_{1}u_{2}+t_{2}u_{1}\right)  -z\left(
t_{1}^{2}+u_{1}^{2}\right)  \left(  t_{2}^{2}+u_{2}^{2}\right)
\end{array}
\right]
\]

\end{proof}

With similar expressions for $b_{1}$ and $b_{2},$ a computer calculation then
shows that identically%
\[
\mathcal{A}\left(  b_{1},b_{2},b_{3}\right)  =1.%
%TCIMACRO{\TeXButton{Proof box}{{\hspace{.1in} \rule{0.5em}{0.5em}}}}%
%BeginExpansion
{\hspace{.1in} \rule{0.5em}{0.5em}}%
%EndExpansion
\]

It is also worth noting that if the coordinates of $b_{3}$ in
(\ref{ThreeBases3}) are respectively $b_{31},b_{32}$ and $b_{33}$, and
similarly for $b_{1}$ and $b_{2},$ then
\[
\det%
\begin{pmatrix}
b_{11} & b_{12} & b_{13}\\
b_{21} & b_{22} & b_{23}\\
b_{31} & b_{32} & b_{33}%
\end{pmatrix}
=\allowbreak-\left(  t_{1}u_{2}-t_{2}u_{1}\right)  \left(  t_{2}u_{3}%
-t_{3}u_{2}\right)  \left(  t_{3}u_{1}-t_{1}u_{3}\right)  .
\]

\subsection{Singly null singly nil triangles}

A triangle $\overline{a_{1}a_{2}a_{3}}$ is \textbf{singly null and singly nil}
when it has exactly one null point, and exactly one null line. There are two
types of such triangles, depending on whether or not the null point lies on
the null line.

\begin{theorem}
[Singly null singly nil Thales]Suppose that $\overline{a_{1}a_{2}a_{3}}$ is a
singly null and singly nil triangle in which $a_{2}a_{3}$ is a null line and
$a_{3}$ is a null point. If $q_{3}\equiv q\left(  a_{1},a_{2}\right)  $ and
$S_{1}\equiv S\left(  a_{1}a_{2},a_{1}a_{3}\right)  $ then
\[
q_{3}S_{1}=1.
\]

\end{theorem}

\begin{proof}
Suppose that $a_{3}=\alpha\left(  t:u\right)  \equiv\left[  t^{2}%
-u^{2}:2tu:t^{2}+u^{2}\right]  $ and $a_{1}=\left[  x:y:z\right]  $, with
$a_{1}$ non-null.

If $t^{2}+u^{2}=0$ then we can write $a_{3}=\left[  t:u:0\right]  $, and by
the Parametrizing a null line theorem $a_{2}=\left[  rt:ru:s\right]  $ for
some proportion $r:s$ where $s\neq0.$ In this case the definition of quadrance
and the condition $t^{2}+u^{2}=0$ gives
\[
q_{3}=\frac{2r^{2}tuxy-2rsuyz-2rstxz+s^{2}x^{2}+s^{2}y^{2}-r^{2}t^{2}%
y^{2}-r^{2}u^{2}x^{2}+r^{2}t^{2}z^{2}+r^{2}u^{2}z^{2}}{\left(  x^{2}%
+y^{2}-z^{2}\right)  s^{2}}%
\]
while the Spread formula gives
\[
S_{1}=\allowbreak\frac{\left(  x^{2}+y^{2}-z^{2}\right)  s^{2}}{2r^{2}%
tuxy-2rsuyz-2rstxz+s^{2}x^{2}+s^{2}y^{2}-r^{2}t^{2}y^{2}-r^{2}u^{2}x^{2}%
+r^{2}t^{2}z^{2}+r^{2}u^{2}z^{2}}%
\]
so that
\[
q_{1}S_{1}=1.
\]

If $t^{2}+u^{2}\neq0$ then also by the Parametrizing a null line theorem we
can write
\[
a_{2}\equiv\left(  r\left(  t^{2}-u^{2}\right)  -2stu:2rtu+s\left(
t^{2}-u^{2}\right)  :r\left(  t^{2}+u^{2}\right)  \right)
\]
for some proportion $r:s$ also with $s\neq0.$ Then a computer calculation
shows also that
\[
q_{1}S_{1}=1.%
%TCIMACRO{\TeXButton{Proof box}{{\hspace{.1in} \rule{0.5em}{0.5em}}}}%
%BeginExpansion
{\hspace{.1in} \rule{0.5em}{0.5em}}%
%EndExpansion
\]

\end{proof}

This result suggests that if $a_{2}a_{3}$ is a null line and $a_{3}$ is a null
point, then although $q\left(  a_{2},a_{3}\right)  $ is not defined, it
behaves in some respects like the number $1.$

\begin{theorem}
[Singly null singly nil orthocenter]Suppose that $\overline{a_{1}a_{2}a_{3}}$
is a triangle in which $a_{1}a_{2}$ is a null line, and that $a_{3}$ is a null
point. Suppose that the base of the couple $\overline{a_{1}\left(  a_{2}%
a_{3}\right)  }$ is $b_{1},$ and the base of the couple $\overline
{a_{2}\left(  a_{1}a_{3}\right)  }$ is $b_{2}$. Then the lines $a_{1}b_{1},$
$a_{2}b_{2}$ and $a_{3}\left(  a_{1}a_{2}\right)  ^{\perp}$ intersect in a
point $o$, and
\[
q\left(  a_{1},o\right)  +q\left(  o,b_{1}\right)  =q\left(  a_{2},o\right)
+q\left(  o,b_{2}\right)  =1.
\]
Furthermore%
\[
q\left(  a_{1},b_{1}\right)  =-\left(  2q\left(  a_{1},o\right)  -1\right)
\left(  2q\left(  b_{1},o\right)  -1\right)  .
\]

\end{theorem}

\begin{proof}
This is a computer assisted calculation along the lines of the previous
theorems.$%
%TCIMACRO{\TeXButton{Proof box}{{\hspace{.1in} \rule{0.5em}{0.5em}}}}%
%BeginExpansion
{\hspace{.1in} \rule{0.5em}{0.5em}}%
%EndExpansion
$
\end{proof}

Note that this gives a situation where the sum of two quadrances between three
collinear points is $1,$ but no two of the points are necessarily
perpendicular, as suggested by the proof of the Complementary quadrances
spreads theorem.

\subsection{Null perspective theorems}

\begin{theorem}
[Null perspective]Suppose that $\alpha_{1},\alpha_{2}$ and $\alpha_{3}$ are
distinct null points, and that $d$ is any point on $\alpha_{1}\alpha_{3}$
distinct from $\alpha_{1}$ and $\alpha_{3}.$ Suppose further that $x$ and $y$
are points lying on $\alpha_{1}\alpha_{2}$ and that $z\equiv\left(  \alpha
_{2}\alpha_{3}\right)  \left(  xd\right)  $ and $w\equiv\left(  \alpha
_{2}\alpha_{3}\right)  \left(  yd\right)  $. Then%
\[
q\left(  x,y\right)  =q\left(  z,w\right)  .
\]

\end{theorem}

\begin{proof}
This is a computer assisted calculation.$%
%TCIMACRO{\TeXButton{Proof box}{{\hspace{.1in} \rule{0.5em}{0.5em}}}}%
%BeginExpansion
{\hspace{.1in} \rule{0.5em}{0.5em}}%
%EndExpansion
$
\end{proof}

\begin{theorem}
[Null subtended quadrance]Suppose that the line $L$ passes through the
distinct null points $\alpha_{1}$ and $\alpha_{2}.$ Then for a third null
point $\alpha_{3}$, and a line $M$ distinct from $\alpha_{1}\alpha_{3}$ and
$\alpha_{2}\alpha_{3}$, let $a_{1}\equiv\left(  \alpha_{1}\alpha_{3}\right)
M$ and $a_{2}\equiv\left(  \alpha_{2}\alpha_{3}\right)  M.$ Then $q\equiv
q\left(  a_{1},a_{2}\right)  $ and $S\equiv S\left(  L,M\right)  $ are related
by%
\[
qS=1.
\]
In particular $q$ is independent of $\alpha_{3}.$
\end{theorem}

\begin{proof}
From the Parametrization of null points theorem we know that we can write
$\alpha_{1}=\alpha\left(  t_{1}:u_{1}\right)  $, $\alpha_{2}=\alpha\left(
t_{2}:u_{2}\right)  $ and $\alpha_{3}=\alpha\left(  t_{3}:u_{3}\right)  .$
Then from the Join of null points theorem
\begin{align*}
L  & \equiv\alpha_{1}\alpha_{2}=\left(  t_{1}t_{2}-u_{1}u_{2}:t_{1}u_{2}%
+t_{2}u_{1}:t_{1}t_{2}+u_{1}u_{2}\right) \\
\alpha_{1}\alpha_{3}  & =\left(  t_{1}t_{3}-u_{1}u_{3}:t_{1}u_{3}+t_{3}%
u_{1}:t_{1}t_{3}+u_{1}u_{3}\right) \\
\alpha_{2}\alpha_{3}  & =\left(  t_{2}t_{3}-u_{2}u_{3}:t_{2}u_{3}+t_{3}%
u_{2}:t_{2}t_{3}+u_{2}u_{3}\right)  .
\end{align*}
Suppose that $M=\left(  l:m:n\right)  $. Then computing with the Join of
points theorem gives%
\begin{align*}
a_{1}  & =\left(  \alpha_{1}\alpha_{3}\right)  M\\
& =\left[  m\left(  t_{1}t_{3}+u_{1}u_{3}\right)  -n\left(  t_{1}u_{3}%
+t_{3}u_{1}\right)  :n\left(  t_{1}t_{3}-u_{1}u_{3}\right)  -l\left(
t_{1}t_{3}+u_{1}u_{3}\right)  :m\left(  t_{1}t_{3}-u_{1}u_{3}\right)
-l\left(  t_{1}u_{3}+t_{3}u_{1}\right)  \right]
\end{align*}
and%
\begin{align*}
a_{2}  & =\left(  \alpha_{2}\alpha_{3}\right)  M\\
& =\left[  m\left(  t_{2}t_{3}+u_{2}u_{3}\right)  -n\left(  t_{2}u_{3}%
+t_{3}u_{2}\right)  :\allowbreak n\left(  t_{2}t_{3}-u_{2}u_{3}\right)
-l\left(  t_{2}t_{3}+u_{2}u_{3}\right)  :m\left(  t_{2}t_{3}-u_{2}%
u_{3}\right)  -l\left(  t_{2}u_{3}+t_{3}u_{2}\right)  \right]  .
\end{align*}
Then a computation using the definition of the quadrance between points gives
\[
q\equiv q\left(  a_{1},a_{2}\right)  =-\frac{\left(  t_{1}u_{2}-t_{2}%
u_{1}\right)  ^{2}\left(  l^{2}+m^{2}-n^{2}\right)  }{\left(  l\left(
t_{1}^{2}-u_{1}^{2}\right)  +2mt_{1}u_{1}-n\left(  t_{1}^{2}+u_{1}^{2}\right)
\right)  \left(  l\left(  t_{2}^{2}-u_{2}^{2}\right)  +2mt_{2}u_{2}-n\left(
t_{2}^{2}+u_{2}^{2}\right)  \right)  }.
\]
A computation using the definition of spread between lines gives%
\[
S\equiv S\left(  L,M\right)  =-\frac{\allowbreak\left(  l\left(  t_{1}%
^{2}-u_{1}^{2}\right)  +m2t_{1}u_{1}-n\left(  t_{1}^{2}+u_{1}^{2}\right)
\right)  \left(  l\left(  t_{2}^{2}-u_{2}^{2}\right)  +2mt_{2}u_{2}-n\left(
t_{2}^{2}+u_{2}^{2}\right)  \right)  }{\left(  t_{1}u_{2}-t_{2}u_{1}\right)
^{2}\left(  l^{2}+m^{2}-n^{2}\right)  }.
\]
Comparing these two expressions we see that
\[
qS=1.%
%TCIMACRO{\TeXButton{Proof box}{{\hspace{.1in} \rule{0.5em}{0.5em}}}}%
%BeginExpansion
{\hspace{.1in} \rule{0.5em}{0.5em}}%
%EndExpansion
\]

\end{proof}

\subsection{Four null points}

Here are two theorems that will play an important role in the further
development of the subject, and provide a link between hyperbolic geometry and
the theory of cyclic quadrilaterals.

\begin{theorem}
[Fully nil quadrangle diagonal]Suppose that $\alpha_{1},\alpha_{2},\alpha_{3}$
and $\alpha_{4}$ are distinct null points, and that $d=\left(  \alpha
_{1}\alpha_{2}\right)  \left(  \alpha_{3}\alpha_{4}\right)  ,$ $e=\left(
\alpha_{1}\alpha_{3}\right)  \left(  \alpha_{2}\alpha_{4}\right)  $ and
$f=\left(  \alpha_{1}\alpha_{4}\right)  \left(  \alpha_{2}\alpha_{3}\right)
$. Then $d,e$ and $f$ are not collinear, and%
\[
d^{\perp}=ef\qquad e^{\perp}=df\qquad\mathit{and}\qquad f^{\perp}=de.
\]

\end{theorem}

\begin{proof}
From the Parametrization of null points theorem, we may write $\alpha
_{1}=\alpha\left(  t_{1}:u_{1}\right)  $, $\alpha_{2}=\alpha\left(
t_{2}:u_{2}\right)  $, $\alpha_{3}=\alpha\left(  t_{3}:u_{3}\right)  $ and
$\alpha_{4}=\alpha\left(  t_{4}:u_{4}\right)  $. The Null diagonal point
theorem allows us to write down $d,e$ and $f,$ namely%
\[
d=\left[
\begin{array}
[c]{c}%
\left(  t_{1}u_{2}+t_{2}u_{1}\right)  \left(  t_{3}t_{4}+u_{3}u_{4}\right)
-\left(  t_{3}u_{4}+t_{4}u_{3}\right)  \left(  t_{1}t_{2}+u_{1}u_{2}\right) \\
:\left(  t_{1}t_{2}+u_{1}u_{2}\right)  \left(  t_{3}t_{4}-u_{3}u_{4}\right)
-\left(  t_{3}t_{4}+u_{3}u_{4}\right)  \left(  t_{1}t_{2}-u_{1}u_{2}\right) \\
:\left(  t_{1}u_{2}+t_{2}u_{1}\right)  \left(  t_{3}t_{4}-u_{3}u_{4}\right)
-\left(  t_{3}u_{4}+t_{4}u_{3}\right)  \left(  t_{1}t_{2}-u_{1}u_{2}\right)
\end{array}
\right]
\]%
\[
e=\left[
\begin{array}
[c]{c}%
\left(  t_{1}u_{3}+u_{1}t_{3}\right)  \left(  t_{2}t_{4}+u_{2}u_{4}\right)
-\left(  t_{2}u_{4}+u_{2}t_{4}\right)  \left(  t_{1}t_{3}+u_{1}u_{3}\right) \\
:\left(  t_{1}t_{3}+u_{1}u_{3}\right)  \left(  t_{2}t_{4}-u_{2}u_{4}\right)
-\left(  t_{2}t_{4}+u_{2}u_{4}\right)  \left(  t_{1}t_{3}-u_{1}u_{3}\right) \\
:\left(  t_{1}u_{3}+u_{1}t_{3}\right)  \left(  t_{2}t_{4}-u_{2}u_{4}\right)
-\left(  t_{2}u_{4}+u_{2}t_{4}\right)  \left(  t_{1}t_{3}-u_{1}u_{3}\right)
\end{array}
\right]
\]%
\[
f=\left[
\begin{array}
[c]{c}%
\left(  t_{1}u_{4}+u_{1}t_{4}\right)  \left(  t_{2}t_{3}+u_{2}u_{3}\right)
-\left(  t_{2}u_{3}+u_{2}t_{3}\right)  \left(  t_{1}t_{4}+u_{1}u_{4}\right) \\
:\left(  t_{1}t_{4}+u_{1}u_{4}\right)  \left(  t_{2}t_{3}-u_{2}u_{3}\right)
-\left(  t_{2}t_{3}+u_{2}u_{3}\right)  \left(  t_{1}t_{4}-u_{1}u_{4}\right) \\
:\left(  t_{1}u_{4}+u_{1}t_{4}\right)  \left(  t_{2}t_{3}-u_{2}u_{3}\right)
-\left(  t_{2}u_{3}+u_{2}t_{3}\right)  \left(  t_{1}t_{4}-u_{1}u_{4}\right)
\end{array}
\right]  .
\]
If the coefficients of $d$ above are $d_{1},d_{2}$ and $d_{3}$ respectively,
and similarly for $e$ and $f,$ then a computer calculation shows that
\begin{align*}
& \det%
\begin{pmatrix}
d_{1} & d_{2} & d_{3}\\
e_{1} & e_{2} & e_{3}\\
f_{1} & f_{2} & f_{3}%
\end{pmatrix}
\\
& =-8\left(  t_{1}u_{2}-t_{2}u_{1}\right)  \left(  t_{2}u_{3}-t_{3}%
u_{2}\right)  \left(  t_{3}u_{4}-t_{4}u_{3}\right)  \left(  t_{4}u_{1}%
-t_{1}u_{4}\right)  \left(  t_{3}u_{1}-t_{1}u_{3}\right)  \allowbreak\left(
t_{4}u_{2}-t_{2}u_{4}\right)
\end{align*}
so that $d,e$ and $f$ are non-collinear by the Collinear points theorem, since
the four proportions are distinct.

Now $d$ and $e$ are perpendicular points because of the identity%
\begin{align*}
& \left(  \left(  t_{1}u_{2}+t_{2}u_{1}\right)  \left(  t_{3}t_{4}+u_{3}%
u_{4}\right)  -\left(  t_{3}u_{4}+t_{4}u_{3}\right)  \left(  t_{1}t_{2}%
+u_{1}u_{2}\right)  \right) \\
& \times\left(  \left(  t_{1}u_{3}+u_{1}t_{3}\right)  \left(  t_{2}t_{4}%
+u_{2}u_{4}\right)  -\left(  t_{2}u_{4}+u_{2}t_{4}\right)  \left(  t_{1}%
t_{3}+u_{1}u_{3}\right)  \right) \\
& +\left(  \left(  t_{1}t_{2}+u_{1}u_{2}\right)  \left(  t_{3}t_{4}-u_{3}%
u_{4}\right)  -\left(  t_{3}t_{4}+u_{3}u_{4}\right)  \left(  t_{1}t_{2}%
-u_{1}u_{2}\right)  \right) \\
& \times\left(  \left(  t_{1}t_{3}+u_{1}u_{3}\right)  \left(  t_{2}t_{4}%
-u_{2}u_{4}\right)  -\left(  t_{2}t_{4}+u_{2}u_{4}\right)  \left(  t_{1}%
t_{3}-u_{1}u_{3}\right)  \right) \\
& -\left(  \left(  t_{1}u_{2}+t_{2}u_{1}\right)  \left(  t_{3}t_{4}-u_{3}%
u_{4}\right)  -\left(  t_{3}u_{4}+t_{4}u_{3}\right)  \left(  t_{1}t_{2}%
-u_{1}u_{2}\right)  \right) \\
& \times\left(  \left(  t_{1}u_{3}+u_{1}t_{3}\right)  \left(  t_{2}t_{4}%
-u_{2}u_{4}\right)  -\left(  t_{2}u_{4}+u_{2}t_{4}\right)  \left(  t_{1}%
t_{3}-u_{1}u_{3}\right)  \right) \\
& =0.
\end{align*}
Similarly $e$ and $f$ are perpendicular, and $d$ and $f$ are perpendicular. It
follows that
\[
d^{\perp}=ef\qquad e^{\perp}=df\qquad\text{\textrm{and}}\qquad f^{\perp}=de.%
%TCIMACRO{\TeXButton{Proof box}{{\hspace{.1in} \rule{0.5em}{0.5em}}}}%
%BeginExpansion
{\hspace{.1in} \rule{0.5em}{0.5em}}%
%EndExpansion
\]

\end{proof}

\begin{theorem}
[$48/64$]Suppose that $\alpha_{1},\alpha_{2},\alpha_{3}$ and $\alpha_{4}$ are
distinct null points, with diagonal spreads $P=S\left(  \alpha_{1}\alpha
_{2},\alpha_{3}\alpha_{4}\right)  $, $R=S\left(  \alpha_{1}\alpha_{3}%
,\alpha_{2}\alpha_{4}\right)  $ and $T=S\left(  \alpha_{1}\alpha_{4}%
,\alpha_{2}\alpha_{3}\right)  $. Then%
\[
PR+RT+TP=48\qquad\mathit{and}\qquad PRT=64.
\]

\end{theorem}

\begin{proof}
This is a computer calculation using the notation of the previous proof.$%
%TCIMACRO{\TeXButton{Proof box}{{\hspace{.1in} \rule{0.5em}{0.5em}}}}%
%BeginExpansion
{\hspace{.1in} \rule{0.5em}{0.5em}}%
%EndExpansion
$
\end{proof}

Note that as a consequence we have the relation%
\[
\frac{1}{P}+\frac{1}{R}+\frac{1}{T}=\frac{3}{4}.
\]
Furthermore we find that the numbers $48$ and $64$ are constants of nature.

These results suggest that the theory of null quadrangles and quadrilaterals
in universal hyperbolic geometry is quite rich, and this turns out to be the
case, as will be discussed in a future paper.

\section{Conclusion}

Although it takes some getting used to, this new setting for such a venerable
subject opens up many new directions for research, and encourages us to
reconsider the \textit{true role of algebraic geometry} in modern mathematics,
where the term \textit{algebraic geometry} is used in the rather broad sense
of meaning an approach to geometry based on algebra.

Certainly one important direction is to use this understanding of hyperbolic
geometry to begin a \textit{deeper and more} \textit{systematic exploration of
relativistic geometries}, a subject that has seen remarkably little
development in the hundred years since Einstein's introduction of special
relativity, despite its obvious importance in understanding the world in which
we live.

In subsequent papers I hope to also explore aspects of \textit{triangle
geometry, quadrilaterals, circles, isometries} and \textit{tesselations }in
the universal setting. Hopefully others will also find these topics attractive
for investigation.

\end{document}